\documentclass[12pt,a4paper,twoside]{article}

\title{\sc Complete Low Frequency Asymptotics 
for Time-Harmonic Generalized Maxwell Equations in Nonsmooth Exterior Domains}
\def\shorttitle{Low Frequency Asymptotics for Maxwell's Equations}
\def\pauthor{Dirk Pauly}

\def\mylabelonoff{off}
\def\allowdisbrk{no}

\usepackage{a4,exscale,ifthen,amsfonts,amssymb,amsmath,amscd,graphicx}
\usepackage[english]{babel}
\usepackage[mathscr]{eucal}

\author{{\sf\pauthor}}
\numberwithin{equation}{section}

\newenvironment{acknow}{{\vspace*{1cm}\noindent\bf Acknowledgements }}{}
\newcommand{\bewboxw}{\mbox{}\hfill $\square$ \\}

\newenvironment{proof}{{\noindent\bf Proof }}{\bewboxw}

\newcommand{\keywords}[1]{{\noindent\bf Key Words }#1}
\newcommand{\amsclass}[1]{{\noindent\bf AMS MSC-Classifications }#1}

\ifthenelse{\equal{\mylabelonoff}{on}}
{\newcommand{\mylabel}[1]{\label{#1}\fbox{{\rm #1}}}}{\newcommand{\mylabel}[1]{\label{#1}\makebox[0mm][]{}}}
\ifthenelse{\equal{\allowdisbrk}{yes}}
{\allowdisplaybreaks}{}

\newcommand{\paper}[7]{\bibitem{#1} #2, `#3', {\it #4}, #5, (#6), #7.}

\newcommand{\book}[6]{\bibitem{#1} #2, {\it #3}, #4, #5, (#6).}

\newcommand{\dissavail}[6]{\bibitem{#1} #2, `#3', {\sf Dissertation}, #4, (#5), available from {\tt #6}.}

\newcommand{\repjyu}[8]{\bibitem{#1} #2, `#3', {\it Reports of the Department of Mathematical Information Technology}, 
University of Jyv\"askyl\"a, Series #4. Scientific Computing, No. #4. #5/#6, ISBN #7, ISSN #8.}

\usepackage{amsfonts,amssymb,amsmath,amscd}
\usepackage[mathscr]{eucal}

\newcommand{\schluss}{\ifodd\value{page}\newpage\thispagestyle{empty}\makebox[0mm][]{}\color{sehrhell}.\fi\end{document}}
\newcommand{\non}{\nonumber}
\newcommand{\normabst}{\hspace{-0.4ex}}

\newcommand{\streltxtty}[2]{\stackrel{\text{\tiny #1}}{#2}}

\newcommand{\ol}[1]{\overline{#1}}
\newcommand{\ul}[1]{\underline{#1}}

\newcommand{\ub}{\underbrace}
\newcommand{\qtext}[1]{\quad\text{#1}\quad}
\newcommand{\qqtext}[1]{\qquad\text{#1}\qquad}
\newcommand{\peh}{\frac{1}{2}}

\newcommand{\Nh}{\frac{N}{2}}
\newcommand{\meh}{{-\peh}}
\newcommand{\me}{{-1}}

\newcommand{\intom}{\int_\Omega}

\newcommand{\om}{{\Omega}}

\newcommand{\omq}{\ol{\om}}
\newcommand{\omb}{\om_\mathrm{b}}

\newcommand{\dom}{{\p\om}}

\newcommand{\ohne}{\setminus}

\newcommand{\eps}{\varepsilon}

\newcommand{\epsd}{\hat{\eps}}
\newcommand{\epsdd}{\hat{\hat{\eps}}}

\newcommand{\mud}{\hat{\mu}}

\newcommand{\Lambdad}{\hat{\Lambda}}

\newcommand{\homg}[3]{{{}^{#3}\mathrm{h}^{#1}_{#2}}}
\newcommand{\psim}[1]{\overset{#1}{\sim}}

\newcommand{\fJ}{\mathbf{J}}

\newcommand{\bSpme}{{\mathbb{S}^\pm_\text{I}}}
\newcommand{\bSpmz}{{\mathbb{S}^\pm_\text{II}}}
\newcommand{\bSpmd}{{\mathbb{S}^\pm_\text{III}}}
\newcommand{\bX}{{\mathbb{X}}}
\newcommand{\bY}{{\mathbb{Y}}}
\newcommand{\bXs}{\tilde{\mathbb{X}}}
\newcommand{\bYs}{\tilde{\mathbb{Y}}}

\newcommand{\EH}{(E,H)}
\newcommand{\EHn}{(E_n,H_n)}

\newcommand{\eh}{(e,h)}

\newcommand{\FG}{(F,G)}
\newcommand{\FGd}{(\hat{F},\hat{G})}

\newcommand{\fg}{(f,g)}
\newcommand{\bon}[1]{\overset{\circ}{b}{}^{#1}}
\newcommand{\fE}{\mathbb{E}}
\newcommand{\fH}{\mathbb{H}}

\newcommand{\Esm}{E_{\sigma,m}}
\newcommand{\Hsm}{H_{\sigma,m}}

\newcommand{\Hsn}{H_{\sigma,n}}
\newcommand{\Egn}{E_{\gamma,n}}
\newcommand{\Hgn}{H_{\gamma,n}}

\newcommand{\FrGd}{(F^r,G^d)}
\newcommand{\FdGr}{(F^d,G^r)}

\newcommand{\To}{\longrightarrow}

\newcommand{\equi}{\Leftrightarrow}
\newcommand{\Equi}{\Longleftrightarrow}

\newcommand{\restr}[2]{\left.#1 \right|_{#2}}

\newcommand{\Abb}[5]{\begin{array}{ccccc}#1&:&#2&\longrightarrow&#3\\{}&{}&#4&\longmapsto&#5\end{array}}

\newcommand{\bds}{\begin{displaystyle}}
\newcommand{\eds}{\end{displaystyle}}
\newcommand{\beq}{\begin{equation}}
\newcommand{\eeq}{\end{equation}}
\newcommand{\beqa}{\begin{eqnarray}}
\newcommand{\eeqa}{\end{eqnarray}}
\newcommand{\beqanon}{\begin{eqnarray*}}
\newcommand{\eeqanon}{\end{eqnarray*}}
\newcommand{\bamath}{$$\begin{array}}
\newcommand{\eamath}{\end{array}$$}
\newcommand{\ba}{\begin{array}}
\newcommand{\ea}{\end{array}}
\newcommand{\bc}{\begin{center}}
\newcommand{\ec}{\end{center}}
\newcommand{\bmini}{\begin{minipage}}
\newcommand{\emini}{\end{minipage}}
\newcommand{\paulesatz}[6]{\setlength{\parindent}{#2em}\vspace*{#4mm}{\bf #1}\setlength{\parindent}{#3em}{\it\,#6}\vspace*{#5mm}}

\newtheorem{lem}{Lemma}[section]
\newtheorem{theo}[lem]{Theorem}
\newtheorem{cor}[lem]{Corollary}
\newtheorem{defini}[lem]{Definition}
\newtheorem{rem}[lem]{Remark}

\newcommand{\set}[2]{\{#1 \;\text{\bf :}\;  #2\}}

\newcommand{\setb}[2]{\big\{#1 \;\text{\bf :}\;  #2\big\}}

\newcommand{\SN}{S^{N-1}}

\DeclareMathOperator{\B}{B}

\newcommand{\pb}[2]{\overset{#2}{\B}{}^{#1}}

\newcommand{\bonq}{\pb{q}{\circ}}

\newcommand{\bonqom}{\pb{q}{\circ}(\Omega)}
\newcommand{\bqpeom}{\B^{q+1}(\Omega)}

\newcommand{\rz}{\mathbb{R}}

\newcommand{\nz}{\mathbb{N}}
\newcommand{\zz}{\mathbb{Z}}

\newcommand{\cz}{\mathbb{C}}
\newcommand{\czp}{{\cz_+}}

\newcommand{\czpomd}{\cz_{+,\hat{\omega}}}
\newcommand{\nzn}{{\nz_0}}

\newcommand{\rN}{{\rz^N}}
\newcommand{\rNon}{{\rN\ohne\{0\}}}
\newcommand{\rzp}{{\rz_+}}
\newcommand{\pI}{\mathbb{I}}
\newcommand{\pP}{\mathbb{P}}

\newcommand{\cIb}{\bar{\cI}}
\newcommand{\cJb}{\bar{\cJ}}

\newcommand{\calO}{{\mathscr{O}}}

\newcommand{\cB}{{\mathcal{B}}}
\newcommand{\cA}{{\mathcal{A}}}
\newcommand{\calH}{{\mathscr{H}}}

\newcommand{\calS}{{\mathscr{S}}}

\newcommand{\calR}{{\mathscr{R}}}
\DeclareMathOperator{\cR}{{\mathcal{R}}}
\newcommand{\calD}{{\mathscr{D}}}
\newcommand{\calM}{{\mathscr{M}}}

\newcommand{\calL}{{\mathcal{L}}}

\newcommand{\calK}{{\mathscr{K}}}
\newcommand{\cK}{\mathcal{K}}

\newcommand{\cS}{{\mathcal{S}}}
\newcommand{\cN}{{\mathcal{N}}}

\DeclareMathOperator{\cI}{{\mathscr{I}}}
\DeclareMathOperator{\cJ}{{\mathscr{J}}}

\newcommand{\zmat}[4]{\begin{bmatrix}#1&#2\\#3&#4\end{bmatrix}}

\DeclareMathOperator{\p}{\partial}

\DeclareMathOperator{\rhoh}{\check{\rho}}
\DeclareMathOperator{\tauh}{\check{\tau}}
\DeclareMathOperator{\loes}{\calL}
\DeclareMathOperator{\loesom}{\calL_\omega}
\DeclareMathOperator{\loesn}{\calL_0}
\DeclareMathOperator{\loc}{loc}
\DeclareMathOperator{\vox}{vox}

\DeclareMathOperator{\MAX}{{{\calM}{\rm ax}}}

\DeclareMathOperator{\statMax}{\mathfrak{Max}}

\DeclareMathOperator{\Lin}{Lin}
\DeclareMathOperator{\id}{Id}

\DeclareMathOperator{\imt}{Im}
\DeclareMathOperator{\supp}{supp}

\DeclareMathOperator{\rot}{rot}
\DeclareMathOperator{\Rot}{Rot}

\DeclareMathOperator{\pdiv}{div}

\DeclareMathOperator{\ie}{i}

\DeclareMathOperator{\pd}{d\hspace{-0.4ex}}

\DeclareMathOperator{\reg}{reg}

\DeclareMathOperator{\curl}{curl}

\DeclareMathOperator{\ei}{\mathbf{e}}

\DeclareMathOperator{\homd}{hom}

\newcommand{\pcoverset}[2]{\overset{#2}{\mathrm{C}}{}^{#1}}
\newcommand{\pc}[1]{\pcoverset{#1}{}}

\newcommand{\cu}{\pc{\infty}}

\newcommand{\lz}{\mathrm{L}^2}

\newcommand{\Lz}[1]{\lz_{#1}}

\newcommand{\Lzs}{\Lz{s}}

\newcommand{\Lzt}{\Lz{t}}

\newcommand{\qc}[3]{\overset{#3}{\mathrm{C}}{}^{#1,#2}}

\newcommand{\cqn}[1]{\qc{#1}{q}{\circ}}

\newcommand{\cqun}{\cqn{\infty}}

\newcommand{\cqunom}{\cqun(\Omega)}

\newcommand{\cqpen}[1]{\qc{#1}{q+1}{\circ}}

\newcommand{\cqpeun}{\cqpen{\infty}}

\newcommand{\cqpeunom}{\cqpeun(\Omega)}

\newcommand{\qlz}[1]{\mathrm{L}^{2,#1}}

\newcommand{\qLz}[2]{\qlz{#1}_{#2}}
\newcommand{\qLzom}[2]{\qLz{#1}{#2}(\Omega)}
\newcommand{\Lzq}[1]{\qLz{q}{#1}}
\newcommand{\Lzqom}[1]{\Lzq{#1}(\Omega)}
\newcommand{\qLzs}[1]{\qLz{#1}{s}}

\newcommand{\qLzt}[1]{\qLz{#1}{t}}

\newcommand{\Lzqs}{\qLzs{q}}
\newcommand{\Lzqsom}{\Lzqs(\Omega)}
\newcommand{\Lzqt}{\qLzt{q}}
\newcommand{\Lzqtom}{\Lzqt(\Omega)}

\newcommand{\Lzqloc}{\Lzq{\loc}}
\newcommand{\Lzqlocom}{\Lzqloc(\Omega)}

\newcommand{\Lzqpe}[1]{\qLz{q+1}{#1}}
\newcommand{\Lzqpeom}[1]{\Lzqpe{#1}(\Omega)}

\newcommand{\qh}[4]{\overset{#4}{\mathrm{H}}{}^{#1,#2}_{#3}}
\newcommand{\qH}[4]{\overset{#4}{\mathbf{H}}{}^{#1,#2}_{#3}}

\newcommand{\qhom}[4]{\qh{#1}{#2}{#3}{#4}(\Omega)}

\newcommand{\pr}[4]{{}_{#3}\overset{#4}{\mathrm{R}}{}^{#1}_{#2}}

\newcommand{\rqvox}{\pr{q}{\vox}{}{}}

\newcommand{\rqsn}{\pr{q}{s}{0}{}}
\newcommand{\rqpesn}{\pr{q+1}{s}{0}{}}

\newcommand{\rqpevoxn}{\pr{q+1}{\vox}{0}{}}
\newcommand{\ronq}[1]{\pr{q}{#1}{}{\circ}}
\newcommand{\ronqpe}[1]{\pr{q+1}{#1}{}{\circ}}
\newcommand{\ronqs}{\pr{q}{s}{}{\circ}}
\newcommand{\ronqpes}{\pr{q+1}{s}{}{\circ}}
\newcommand{\ronqt}{\pr{q}{t}{}{\circ}}

\newcommand{\ronqn}[1]{\pr{q}{#1}{0}{\circ}}
\newcommand{\ronqpen}[1]{\pr{q+1}{#1}{0}{\circ}}

\newcommand{\ronqtn}{\pr{q}{t}{0}{\circ}}
\newcommand{\ronqpetn}{\pr{q+1}{t}{0}{\circ}}

\newcommand{\ronqvoxn}{\pr{q}{\vox}{0}{\circ}}
\newcommand{\ronqpevoxn}{\pr{q+1}{\vox}{0}{\circ}}

\newcommand{\rqvoxom}{\rqvox(\Omega)}

\newcommand{\rqpevoxnom}{\rqpevoxn(\Omega)}
\newcommand{\ronqom}[1]{\ronq{#1}(\Omega)}
\newcommand{\ronqpeom}[1]{\ronqpe{#1}(\Omega)}
\newcommand{\ronqsom}{\ronqs(\Omega)}
\newcommand{\ronqpesom}{\ronqpes(\Omega)}
\newcommand{\ronqtom}{\ronqt(\Omega)}

\newcommand{\ronqnom}[1]{\ronqn{#1}(\Omega)}
\newcommand{\ronqpenom}[1]{\ronqpen{#1}(\Omega)}

\newcommand{\ronqtnom}{\ronqtn(\Omega)}
\newcommand{\ronqpetnom}{\ronqpetn(\Omega)}

\newcommand{\ronqvoxnom}{\ronqvoxn(\Omega)}
\newcommand{\ronqpevoxnom}{\ronqpevoxn(\Omega)}
\newcommand{\pR}[4]{{}_{#3}\overset{#4}{\mathbf{R}}{}^{#1}_{#2}}
\newcommand{\pRq}[1]{\pR{q}{#1}{}{}}

\newcommand{\Rqpes}{\pR{q+1}{s}{}{}}
\newcommand{\Rqt}{\pR{q}{t}{}{}}

\newcommand{\Ronq}[1]{\pR{q}{#1}{}{\circ}}

\newcommand{\Ronqt}{\pR{q}{t}{}{\circ}}

\newcommand{\Rqtom}{\Rqt(\Omega)}

\newcommand{\Ronqom}[1]{\Ronq{#1}(\Omega)}

\newcommand{\Ronqtom}{\Ronqt(\Omega)}

\newcommand{\Ronqkmeh}{\Ronq{<\meh}}
\newcommand{\Ronqkmehom}{\Ronqkmeh(\Omega)}

\newcommand{\Rqkmeh}{\pRq{<\meh}}

\newcommand{\pdi}[4]{{}_{#3}\overset{#4}{\mathrm{D}}{}^{#1}_{#2}}
\newcommand{\pdq}[1]{\pdi{q}{#1}{}{}}
\newcommand{\dqpe}[1]{\pdi{q+1}{#1}{}{}}
\newcommand{\dqs}{\pdi{q}{s}{}{}}
\newcommand{\dqpes}{\pdi{q+1}{s}{}{}}

\newcommand{\dqpet}{\pdi{q+1}{t}{}{}}

\newcommand{\dqvox}{\pdi{q}{\vox}{}{}}

\newcommand{\dqn}[1]{\pdi{q}{#1}{0}{}}

\newcommand{\dqsn}{\pdi{q}{s}{0}{}}
\newcommand{\dqpesn}{\pdi{q+1}{s}{0}{}}
\newcommand{\dqtn}{\pdi{q}{t}{0}{}}
\newcommand{\dqpetn}{\pdi{q+1}{t}{0}{}}

\newcommand{\dqvoxn}{\pdi{q}{\vox}{0}{}}
\newcommand{\dqpevoxn}{\pdi{q+1}{\vox}{0}{}}

\newcommand{\dqom}[1]{\pdq{#1}(\Omega)}
\newcommand{\dqpeom}[1]{\dqpe{#1}(\Omega)}
\newcommand{\dqsom}{\dqs(\Omega)}
\newcommand{\dqpesom}{\dqpes(\Omega)}

\newcommand{\dqpetom}{\dqpet(\Omega)}

\newcommand{\dqvoxom}{\dqvox(\Omega)}

\newcommand{\dqnom}[1]{\dqn{#1}(\Omega)}

\newcommand{\dqtnom}{\dqtn(\Omega)}
\newcommand{\dqpetnom}{\dqpetn(\Omega)}

\newcommand{\dqvoxnom}{\dqvoxn(\Omega)}
\newcommand{\dqpevoxnom}{\dqpevoxn(\Omega)}

\newcommand{\pDi}[4]{{}_{#3}\overset{#4}{\mathbf{D}}{}^{#1}_{#2}}
\newcommand{\pDq}[1]{\pDi{q}{#1}{}{}}
\newcommand{\Dqpe}[1]{\pDi{q+1}{#1}{}{}}
\newcommand{\Dqs}{\pDi{q}{s}{}{}}

\newcommand{\Dqpet}{\pDi{q+1}{t}{}{}}

\newcommand{\Dqom}[1]{\pDq{#1}(\Omega)}
\newcommand{\Dqpeom}[1]{\Dqpe{#1}(\Omega)}

\newcommand{\Dqpetom}{\Dqpet(\Omega)}

\newcommand{\Dqpekmeh}{\Dqpe{<\meh}}
\newcommand{\Dqpekmehom}{\Dqpekmeh(\Omega)}

\newcommand{\dH}[3]{{}_{#3}\calH^{#1}_{#2}}

\newcommand{\dhqepsom}[1]{\dH{q}{#1}{\eps}(\Omega)}

\DeclareMathOperator{\Reg}{Reg}
\newcommand{\regom}[3]{\Reg^{#1,#2}_{#3}(\Omega)}
\newcommand{\regqom}[2]{\regom{q}{#1}{#2}}
\newcommand{\regqnom}[1]{\regqom{0}{#1}}
\newcommand{\regqjom}[1]{\regqom{j}{#1}}
\newcommand{\regqJom}[1]{\regqom{\fJ}{#1}}
\newcommand{\regqnsom}{\regqnom{s}}
\newcommand{\regqjsom}{\regqjom{s}}
\newcommand{\regqJsom}{\regqJom{s}}
\newcommand{\regqntom}{\regqnom{t}}

\newcommand{\turmd}[5]{{{}^{#5}D^{#1,#2}_{#3,#4}}}
\newcommand{\turmr}[5]{{{}^{#5}R^{#1,#2}_{#3,#4}}}

\newcommand{\skp}[2]{\langle#1,#2\rangle}

\newcommand{\skpb}[2]{\big\langle#1,#2\big\rangle}

\newcommand{\norm}[1]{|\normabst|#1|\normabst|}
\newcommand{\normb}[1]{\big|\normabst\big|#1\big|\normabst\big|}
\newcommand{\normB}[1]{\Big|\normabst\Big|#1\Big|\normabst\Big|}

\newcommand{\qqLz}[3]{\qLz{#1,#2}{#3}}
\newcommand{\Lzqqpe}[1]{\qqLz{q}{q+1}{#1}}

\newcommand{\Lzqqpet}{\Lzqqpe{t}}
\newcommand{\Lzqqpes}{\Lzqqpe{s}}
\newcommand{\Lzqqpeloc}{\Lzqqpe{\loc}}

\newcommand{\qqLzom}[3]{\qLzom{#1,#2}{#3}}
\newcommand{\Lzqqpeom}[1]{\qqLzom{q}{q+1}{#1}}

\newcommand{\Lzqqpetom}{\Lzqqpeom{t}}
\newcommand{\Lzqqpesom}{\Lzqqpeom{s}}
\newcommand{\Lzqqpelocom}{\Lzqqpeloc(\om)}

\newcommand{\Lzqqpeoml}{\tilde{\mathrm{L}}^{2,q,q+1}(\om)}
\renewcommand{\Esm}{E^+_{\sigma,m}}

\renewcommand{\Egn}{E^+_{\gamma,n}}
\renewcommand{\Hsm}{H^+_{\sigma,m}}
\renewcommand{\Hsn}{H^+_{\sigma,n}}
\renewcommand{\Hgn}{H^+_{\gamma,n}}

\newcommand{\Etgn}{\tilde{E}^+_{\gamma,n}}
\newcommand{\Htsm}{\tilde{H}^+_{\sigma,m}}

\newcommand{\skpoml}[2]{\skp{#1}{#2}_{\Lzqqpeoml}}
\newcommand{\skpboml}[2]{\skpb{#1}{#2}_{\Lzqqpeoml}}

\newcommand{\skpbomol}[2]{\skpb{#1}{#2}_{\Lzqqpeom{}}}

\newcommand{\epsmn}{e^{+,0}_{\sigma,m}}
\newcommand{\epsme}{e^{+,1}_{\sigma,m}}
\newcommand{\epsmz}{e^{+,2}_{\sigma,m}}
\newcommand{\epgnz}{e^{+,2}_{\gamma,n}}
\newcommand{\hpsmz}{h^{+,2}_{\sigma,m}}
\newcommand{\hpgnuz}{h^{+,2}_{\gamma,\nu}}
\newcommand{\emsn}{e^-_{\sigma,n}}

\newcommand{\hmsm}{h^-_{\sigma,m}}

\newcommand{\epmsn}{e^\pm_{\sigma,n}}

\newcommand{\hpmsm}{h^\pm_{\sigma,m}}

\newcommand{\emsnn}{e^{-,0}_{\sigma,n}}
\newcommand{\emsmn}{e^{-,0}_{\sigma,m}}
\newcommand{\hmsmn}{h^{-,0}_{\sigma,m}}
\newcommand{\emsme}{e^{-,1}_{\sigma,m}}
\newcommand{\emsne}{e^{-,1}_{\sigma,n}}
\newcommand{\hmsme}{h^{-,1}_{\sigma,m}}
\newcommand{\emsmz}{e^{-,2}_{\sigma,m}}
\newcommand{\emsnz}{e^{-,2}_{\sigma,n}}
\newcommand{\hmsmz}{h^{-,2}_{\sigma,m}}
\newcommand{\emgnz}{e^{-,2}_{\gamma,n}}
\newcommand{\hmgnz}{h^{-,2}_{\gamma,n}}
\newcommand{\epmsnn}{e^{\pm,0}_{\sigma,n}}
\newcommand{\epmsmn}{e^{\pm,0}_{\sigma,m}}
\newcommand{\hpmsmn}{h^{\pm,0}_{\sigma,m}}

\newcommand{\epmsne}{e^{\pm,1}_{\sigma,n}}
\newcommand{\hpmsme}{h^{\pm,1}_{\sigma,m}}
\newcommand{\epmsmz}{e^{\pm,2}_{\sigma,m}}
\newcommand{\epmsnz}{e^{\pm,2}_{\sigma,n}}
\newcommand{\hpmsmz}{h^{\pm,2}_{\sigma,m}}

\newcommand{\hmsml}{h^{-,\ell}_{\sigma,m}}
\newcommand{\emsnl}{e^{-,\ell}_{\sigma,n}}

\newcommand{\emgnlz}{e^{-,\ell+2}_{\gamma,n}}
\newcommand{\hmgnlz}{h^{-,\ell+2}_{\gamma,n}}

\newcommand{\emsmke}{e^{-,k+1}_{\sigma,m}}
\newcommand{\hmsmke}{h^{-,k+1}_{\sigma,m}}
\newcommand{\emsmkz}{e^{-,k+2}_{\sigma,m}}
\newcommand{\hmsmkz}{h^{-,k+2}_{\sigma,m}}

\newcommand{\hpmsml}{h^{\pm,\ell}_{\sigma,m}}
\newcommand{\epmsnl}{e^{\pm,\ell}_{\sigma,n}}
\newcommand{\hpmsmkl}{h^{\pm,k+\ell}_{\sigma,m}}
\newcommand{\epmsnkl}{e^{\pm,k+\ell}_{\sigma,n}}
\newcommand{\hpmsmlz}{h^{\pm,\ell+2}_{\sigma,m}}

\newcommand{\epmsnlz}{e^{\pm,\ell+2}_{\sigma,n}}

\newcommand{\hpmsmle}{h^{\pm,\ell+1}_{\sigma,m}}
\newcommand{\epmsnle}{e^{\pm,\ell+1}_{\sigma,n}}
\newcommand{\Esml}{E^{+,\ell}_{\sigma,m}}
\newcommand{\Esmle}{E^{+,\ell+1}_{\sigma,m}}
\newcommand{\Hsml}{H^{+,\ell}_{\sigma,m}}
\newcommand{\Hsmsk}{H^{+,k}_{\sigma,\tilde{m}}}
\newcommand{\Esmkeml}{E^{+,k+1-\ell}_{\sigma,m}}
\newcommand{\Hsmkeml}{H^{+,k+1-\ell}_{\sigma,m}}
\newcommand{\Egnl}{E^{+,\ell}_{\gamma,n}}
\newcommand{\Hgnl}{H^{+,\ell}_{\gamma,n}}
\newcommand{\Egnul}{E^{+,\ell}_{\gamma,\nu}}
\newcommand{\Esmk}{E^{+,k}_{\sigma,m}}
\newcommand{\Hsnk}{H^{+,k}_{\sigma,n}}
\newcommand{\Hgnn}{H^{+,0}_{\gamma,n}}
\newcommand{\Hsnn}{H^{+,0}_{\sigma,n}}
\newcommand{\Esmn}{E^{+,0}_{\sigma,m}}
\newcommand{\Hsmn}{H^{+,0}_{\sigma,m}}
\newcommand{\Hgne}{H^{+,1}_{\gamma,n}}
\newcommand{\Esme}{E^{+,1}_{\sigma,m}}

\newcommand{\Esmz}{E^{+,2}_{\sigma,m}}

\newcommand{\HgnJe}{H^{+,\fJ+1}_{\gamma,n}}
\newcommand{\EsmJe}{E^{+,\fJ+1}_{\sigma,m}}
\newcommand{\Hsmk}{H^{+,k}_{\sigma,m}}
\newcommand{\Esmj}{E^{+,j}_{\sigma,m}}
\newcommand{\Hsnj}{H^{+,j}_{\sigma,n}}
\newcommand{\Hsmj}{H^{+,j}_{\sigma,m}}
\newcommand{\Esmke}{E^{+,k+1}_{\sigma,m}}
\newcommand{\Hsmke}{H^{+,k+1}_{\sigma,m}}
\newcommand{\Hgnle}{H^{+,\ell+1}_{\gamma,n}}
\newcommand{\Esmkmzse}{E^{+,k-2\sigma+1}_{\sigma,m}}
\newcommand{\Hsmkmzse}{H^{+,k-2\sigma+1}_{\sigma,m}}
\newcommand{\bSp}{\mathrm{S}^+}
\newcommand{\bSm}{\mathrm{S}^-}
\newcommand{\bSpm}{\mathrm{S}^\pm}
\renewcommand{\bSpme}{\mathrm{S}^\pm_{\mathrm{I}}}
\renewcommand{\bSpmz}{\mathrm{S}^\pm_{\mathrm{II}}}
\renewcommand{\bSpmd}{\mathrm{S}^\pm_{\mathrm{III}}}
\newcommand{\bqpe}{\pb{q+1}{}}
\renewcommand{\FdGr}{(F_d,G_r)}
\renewcommand{\FrGd}{(F_r,G_d)}
\newcommand{\FdGrt}{(\tilde{F}_d,\tilde{G}_r)}
\newcommand{\FrGdt}{(\tilde{F}_r,\tilde{G}_d)}
\DeclareMathOperator{\Cor}{Cor}
\newcommand{\corom}[2]{\Cor^{#1,#2}(\Omega)}
\DeclareMathOperator{\Tri}{Tri}
\newcommand{\triom}[2]{\Tri^{#1}_{#2}(\Omega)}
\newcommand{\triqsom}{\triom{q}{s}}
\newcommand{\cloes}{\mathscr{L}}
\newcommand{\cloesom}{\cloes_\omega}

\newcommand{\bD}[3]{{}_{#3}\mathbb{D}^{#1}_{#2}}
\newcommand{\bR}[3]{{}_{#3}\overset{\circ}{\mathbb{R}}{}^{#1}_{#2}}

\newcommand{\bDom}[3]{{}_{#3}\mathbb{D}^{#1}_{#2}(\Omega)}
\newcommand{\bRom}[3]{{}_{#3}\overset{\circ}{\mathbb{R}}{}^{#1}_{#2}(\Omega)}
\newcommand{\bDqsom}{\bD{q}{s}{0}(\Omega)}
\newcommand{\bRqsom}{\bR{q}{s}{0}(\Omega)}

\newcommand{\bRqpesom}{\bR{q+1}{s}{0}(\Omega)}

\newcommand{\bWom}[2]{\mathbb{W}^{#1}_{#2}(\Omega)}
\newcommand{\bWqsom}{\bWom{q}{s}}

\usepackage{algorithm,algorithmic}
\renewcommand{\fH}{\mathbf{H}}

\newcommand{\gt}{\gamma_\tau}
\newcommand{\gn}{\gamma_\nu}
\newcommand{\chgt}{\check{\gamma}_\tau}

\newcommand{\cRqdom}{\cR^q(\dom)}
\newcommand{\Loesom}{\cS_\omega}

\newcommand{\paulytimeharmetadef}{\cite[(2.1), (2.2), (2.3)]{paulytimeharm}}
\newcommand{\paulytimeharmdefeps}{\cite[Definition 2.1 and 2.2]{paulytimeharm}}
\newcommand{\paulytimeharmmlcpdef}{\cite[Definition 2.4]{paulytimeharm}}
\newcommand{\paulytimeharmtheofred}{\cite[Theorem 3.5]{paulytimeharm}}
\newcommand{\paulytimeharmcorfred}{\cite[Corollary 3.9]{paulytimeharm}}
\newcommand{\paulytimeharmPzero}{\cite[Lemma 5.2]{paulytimeharm}}
\newcommand{\paulytimeharmPzeroiv}{\cite[Lemma 5.2 (iv)]{paulytimeharm}}
\newcommand{\paulytimeharmPzeroiii}{\cite[Lemma 5.2 (iii)]{paulytimeharm}}
\newcommand{\paulytimeharmcorlomcont}{\cite[Corollary 5.5]{paulytimeharm}}
\newcommand{\paulytimeharmpartint}{\cite[(2.5)]{paulytimeharm}}
\newcommand{\paulytimeharmsecthprob}{\cite[section 3]{paulytimeharm}}
\newcommand{\paulytimeharmsecboundary}{\cite[section 6]{paulytimeharm}}
\newcommand{\paulytimeharmsecasym}{\cite[section 5]{paulytimeharm}}
\newcommand{\paulytimeharmfundrep}{\cite[Theorem 5.1]{paulytimeharm}}

\newcommand{\paulystaticsectower}{\cite[section 2]{paulystatic}}
\newcommand{\paulystaticsecgenstatic}{\cite[section 4]{paulystatic}}
\newcommand{\paulystatictheoaltmax}{\cite[Theorem 4.6]{paulystatic}}
\newcommand{\paulystaticintdirichleteps}{\cite[Lemma 3.8]{paulystatic}}
\newcommand{\paulystaticbFormendirichletsenkrecht}{\cite[section 4]{paulystatic}}
\newcommand{\paulystaticinttower}{\cite[Remark 2.5]{paulystatic}}
\newcommand{\paulystaticitloes}{\cite[Theorem 5.10]{paulystatic}}
\newcommand{\paulystaticitloesrem}{\cite[Remark 5.11]{paulystatic}}
\newcommand{\paulystaticitloescor}{\cite[Corollary 5.12]{paulystatic}}
\newcommand{\paulystaticitloesremcor}{\cite[Theorem 5.10, Remark 5.11, Corollary 5.12]{paulystatic}}
\newcommand{\paulystaticremodd}{\cite[Remark 2.4]{paulystatic}}
\newcommand{\paulystaticsecordlemone}{\cite[Lemma 7.1]{paulystatic}}
\newcommand{\paulystaticsecordrem}{\cite[Remark 7.4]{paulystatic}}
\newcommand{\paulystaticsphharmexpa}{\cite[Theorem 2.6]{paulystatic}}
\newcommand{\paulystaticcoefesti}{\cite[Remark 2.2 (v)]{paulystatic}}

\newcommand{\paulydecosecres}{\cite[section 3]{paulydeco}}

\newcommand{\paulydecotheomaindeltaeps}{\cite[Theorem 3.2 (iv)]{paulydeco}}

\begin{document}

\date{2007}
\maketitle{}

\begin{abstract}
We continue the study of the operator of generalized Maxwell equations and
completely discover the behavior of the solutions of the time-harmonic equations
as the frequency tends to zero. Thereby we identify degenerate operators in terms
of special `polynomially growing' solutions of a corresponding static problem,
which must be added to the `usual' Neumann series in order to describe the
low frequency asymptotic adequately.\\
\\
\amsclass{35Q60, 78A25, 78A30}
\end{abstract}

\tableofcontents

\section{Introduction and main results}

We continue our studies started in \cite{paulytimeharm} and \cite{paulystatic}
on the low frequency behaviour of the solution operator
$$\Abb{\loesom}{\Lzqqpeom{s}}{\Lzqqpeom{t}}{\FG}{\EH}\qqtext{,}\forall\quad s,-t>1/2$$
of the generalized\footnote{Here `generalized' means the framework of alternating differential forms.}
time-harmonic Maxwell equations
$$\pdiv H+\ie\omega\eps E=F\qqtext{,}\rot E+\ie\omega\mu H=G\qquad,$$
shortly written as
\beq(M+\ie\omega\Lambda)\EH=\FG\qquad,\mylabel{Maxeq}\eeq
with homogeneous (Dirichlet) tangential or electric boundary condition
$$\gt E=0\qquad,$$
which models a perfect conductor, and (Maxwell) radiation conditions
$$(-1)^{qN}*\pd r\wedge*H+E\in\Lzqom{\hat{t}}\qtext{,}\pd r\wedge E+H\in\Lzqpeom{\hat{t}}
\qtext{,}\hat{t}>-1/2\quad,$$
shortly written as
$$(\tilde{S}+\id)\EH\in\Lzqqpeom{\hat{t}}\qquad,$$
in an exterior domain $\om\subset\rN$ with $N\in\nz$ (i.e. a connected open set with compact complement).
Here $r(x)=|x|$ for $x\in\rN$ and $\wedge$ denotes the exterior product as well as
$*$ the usual Hodge star isomorphism.
$\gt:=\iota^*$ denotes the tangential trace and
$\iota:\dom\hookrightarrow\ol{\om}$ the natural embedding of the boundary.
We note that the range of $\loesom$ is even contained in $\Ronqtom\times\Dqpetom$\,.
Here $\EH$\,, $\FG$ are pairs of alternating differential forms of rank $q$ resp. $q+1$\,. Moreover,
we introduce the matrix-type operators acting on such pairs
$$M=\zmat{0}{\pdiv}{\rot}{0}\qqtext{,}\Lambda=\zmat{\eps}{0}{0}{\mu}
\qqtext{,}\tilde{S}=\zmat{0}{\tilde{T}}{\tilde{R}}{0}$$
with $\tilde{T}=(-1)^{qN}*\tilde{R}*$ and $\tilde{R}=\pd r\wedge$\,.
Following Weyl \cite{weyl} and to remind of the electro-magnetic background it has become customary
to denote the exterior derivative $\pd$ by $\rot$ and the co-derivative $\delta$ by $\pdiv$\,.
Hence on $(q+1)$-forms we have
$$\pdiv=(-1)^{qN}*\rot*\qquad.$$
We note in passing that the differential operators $\rot$\,, $\pdiv$\,, $M$ and the `multiplication'
operators $\tilde{R}$\,, $\tilde{T}$\,, $\tilde{S}$ are related to each other by at least two properties:
For any smooth function $\varphi$ we have
$$C_{\rot,\varphi(r)}=\varphi'(r)\tilde{R}\qqtext{,}C_{\pdiv,\varphi(r)}=\varphi'(r)\tilde{T}
\qqtext{,}C_{M,\varphi(r)}=\varphi'(r)\tilde{S}\qquad,$$
where $C_{A,B}$ denotes the commutator of the two operators $A$ and $B$\,,
as well as $\rot$\,, $\pdiv$\,, $M$ correspond in the Fourier image to
$\ie r\tilde{R}$\,, $\ie r\tilde{T}$\,, $\ie r\tilde{S}$\,.
Furthermore, the frequencies $\omega$ will be taken from
the upper half plane $\czp=\set{z\in\cz}{\imt z\geq0}$\,. We define
$$\qLzom{q,p}{s}:=\Lzqom{s}\times\qLzom{p}{s}\qqtext{,}q,p\in\zz\qtext{,}s\in\rz\qquad,$$
where $\Lzqsom$ is the Hilbert space of all differential forms $E\in\Lzqlocom$
with $\rho^sE\footnote{Here $\rho:=\rho(r):=(1+r^2)^{1/2}$\,.}\in\Lzqom{}$\,,
equipped with the scalar product
$$\skp{E}{H}_{\Lzqsom}:=\intom\rho^{2s}E\wedge*\bar{H}\qquad,$$
where the bar denotes complex conjugation.

We are going to model inhomogeneous, anisotropic and nonsmooth media
by linear $\text{\rm L}^\infty$-transformations $\eps$ and $\mu$ on $q$- and $(q+1)$-forms,
i.e. dielectricity and permeability.
Moreover, our right hand sides $\FG$ from $\Lzqqpeom{s}$ under consideration do not have to be
necessarily compactly supported.

The study of wave scattering at low frequencies was pioneered by Lord Rayleigh \cite{rayleigh}.
His contributions provide the foundation on which almost all subsequent work is based.
Low frequency asymptotics for Maxwell's boundary value problem have been given, for instance,
by M\"uller and Niemeyer \cite{muellernie}, Stevenson \cite{stevenson}, Kleinman \cite{kleinman},
Werner \cite{wernere,wernerz,wernerd,wernerv,werners}, Kress \cite{kress}, Ramm \cite{ramm},

Kriegsmann and Reiss \cite{kriegsmann}, Ramm, Weaver, Weck and Witsch \cite{dissmax},
Athanasiadis, Costakis and Stratis \cite{athanasiadissolvandlowfreq}
as well as by Picard \cite{asymmax} and Weck and Witsch \cite{dissmaxbd}.
We also should mention the book of Dassios and Kleinman \cite{dassios}.
Ramm and Somersalo \cite{rammsomersalo} and Lassas \cite{lassas} considered the low frequency limit
from the point of view of inverse problems.

In none of the works cited above the calculation of the higher order terms in a suitable
expansion in terms of the frequency is analyzed.
Ammari and N\'ed\'elec proved in \cite{ammarinedelecondes,ammarinedelecscatt} such expansions.
They reformulated the exterior boundary value problem in a truncated bounded domain
using an `exterior electromagnetic operator',
called by Monk \cite{monk} the `electric to magnetic Calderon operator'
or by Colton and Kress \cite{coltonkress} the `electric to magnetic boundary component map',
which is the counterpart of the Dirichlet to Neumann operator
for Helmholtz' equation. Unfortunately due to the asymptotic expansion of the exterior electromagnetic
operator their method requires exact nonlocal radiation conditions, which leads to nonlocal
boundary conditions on a sphere. Thus it is not possible to identify the expected Neumann series part
of the corresponding static solution operator in the solution. Furthermore, they discuss only the case
$$\big(\zmat{0}{-\curl}{\curl}{0}-\ie\omega(\id+\Lambdad)\big)\EH=-\ie\omega\Lambdad\FG\qquad,$$
where $\FG$ is the time independent part of a time-harmonic incoming wave and $\Lambdad$ some
compactly supported perturbation.

To overcome these problems and limitations and to identify the `usual' Neumann series part of a static solution operator
Weck and Witsch started in \cite{lowfreqlimit,asyanal,exacttwo,exacthigher} a detailed analysis
of the low frequency behavior of solutions of Helmholtz' equation.
In \cite{complete} their new method was completed.
They identified degenerate correction operators in terms of special `polynomially growing'
solutions of a corresponding static problem, which must be added to the `usual' Neumann series
in order to describe the low frequency asymptotic adequately.
With the help of \cite{sphharm}, where a calculus in spherical coordinates suited for differential forms
has been established, they were able to apply and extend their methods
to the case of generalized linear elasticity.

Now in this paper we transfer their method to the case of generalized Maxwell equations.
Hereby we again utilize \cite{sphharm} as an important tool and also the preliminary works
\cite{paulytimeharm,paulystatic,paulydeco}. Thus throughout this paper we will use
the notations introduced in these papers.

All these low frequency investigations are not only motivated by the problem in its own right,
but also by its applications to the large time behavior of solutions to the initial boundary value
problems for the (generalized) wave equation
and to the existence proofs for nonlinear (generalized) wave equations.
In this context we may refer to Eidus' principle of limiting amplitude \cite{eiduslampl}
and, for instance, the papers of Werner \cite{wernerf} as well as Morgenr\"other and Werner
\cite{morgenwernere,morgenwernerz}.

In \cite{paulytimeharm} we studied the time-harmonic solutions of \eqref{Maxeq}.
Since the linear operator
$$\Abb{\calM}{\Ronqom{}\times\Dqpeom{}\subset{}_\Lambda\Lzqqpeom{}}{{}_\Lambda\Lzqqpeom{}}{\EH}{\ie\Lambda^\me M\EH}$$
is selfadjoint, we were able to obtain radiation solutions $\EH$ for nonvanishing real frequencies
and right hand sides $\FG\in\Lzqqpeom{>\frac{1}{2}}$ by means of Eidus' limiting absorption principle
\cite{eidusla} (approaching from the upper half plane $\czp$).
These solutions are elements of $\Lzqqpeom{<\meh}$ and satisfy the Maxwell radiation condition, i.e.
$(\tilde{S}+\id)\EH\in\Lzqqpeom{>\meh}$
\big(see \paulytimeharmsecthprob\, for details\big).
In other words the resolvent $(\calM-\omega)^\me$ of $\calM$
and hence also $\loesom=\ie(\calM-\omega)^\me\Lambda^\me$ may be extended continuously to the real axis.
Then using the fundamental solution of Helmholtz' equation
in the whole space $\rN$ we showed that eventually eigenvalues of $\calM$
do not accumulate even at $\omega=0$\,. This makes $\loesom$ well defined on the whole of $\Lzqqpesom$
for small frequencies $0\neq\omega\in\czp$\,. Finally we proved in \paulytimeharmcorlomcont\,
that $\loesom$ restricted to the closed subspace $\regqnsom$\footnote{See Definition \ref{Regdefi}}
of $\Lzqqpesom$ converges to the static solution operator $\loesn$ as $\omega$ tends to zero
in the norm of bounded linear operators from $\regqnsom$
to $\Ronqtom\times\Dqpetom$ for all $s\in(1/2,N/2)$ and $t<s-(N+1)/2$\,.
Here $\regqnsom$ consists of solenoidal resp. irrotational forms and
$$\Abb{\loesn}{\regqnom{}}{\big(\ronqom{-1}\times\dqpeom{-1}\big)\cap\Lambda^\me\regqnom{-1}}{\FG}{\EH}\qquad,$$
where $\EH\in\Lzqqpeom{-1}$ is the unique solution of the (decoupled) static Maxwell problem
\begin{align*}
\rot E&=G&&,&\pdiv\eps E&=0&&,&\iota^*E&=0&&,&\eps E&\bot\bonqom&&,\\
\pdiv H&=F&&,&\rot\mu H&=0&&,&\iota^*\mu H&=0&&,&\mu H&\bot\bqpeom&&,
\end{align*}
which may shortly be written as
$$\EH\in\big(\ronqom{-1}\times\dqpeom{-1}\big)\cap\Lambda^\me\regqnom{-1}\quad\wedge\quad M\EH=\FG\quad.$$
The special forms from $\bonqom$ resp. $\bqpeom$ possess compact resp. bounded supports in $\om$
and they play the role of the Dirichlet forms $\dhqepsom{}$ resp. $\dH{q+1}{}{\mu^\me}(\om)$\,,
where $\dH{q}{}{\nu}(\om)=\dH{q}{0}{\nu}(\om)$ and for $t\in\rz$ (in classical terms)
$$\dH{q}{t}{\nu}(\om)=\setb{E\in\Lzqtom}{\rot E=0\,,\,\pdiv\nu E=0\,,\,\iota^*E=0}\qquad.$$
Due to the existence of a nontrivial kernel of the Maxwell operator
these (or other) orthogonality constraints are necessary.

In the bounded domain case it is just an easy exercise to show that $\loesom$
is approximated by Neumann's series of $\loesn$ or $\loes=\Lambda\loesn$ for small frequencies $\omega$\,, i.e.
\beq\loesom=-(-\ie\omega)^\me\Pi+\sum_{j=0}^{\infty}(-\ie\omega)^j\loesn\loes^{j}\Pi_{\reg}\qquad,\mylabel{bddom}\eeq
where $\Pi$ and $\Pi_{\reg}=\id-\Lambda\Pi$ are projections onto the kernel of $M$
and its orthogonal complement in $\Lzqqpeom{}$\,.
In the case of an exterior domain this low frequency asymptotic holds no longer true,
because due to Poincare's estimate for Maxwell equations
the static solution operator maps data from a polynomially weighted Sobolev space to solutions
belonging to a less weighted Sobolev space. So a priori it is not clear, in which way one
may define higher powers of a static solution operator.

In \cite{paulystatic} we took care of some electro-magneto static problems and were able to prove
that an (not obvious and relatively complicated)
iteration process of a suitable static solution operator $\loes$ still holds true \paulystaticitloes.
This gives meaning to the powers $\loes^j$ of $\loes$
as continuous linear operators on subspaces of $\regqnom{\loc}$ even for exterior domains.
As a byproduct we proved a generalized spherical harmonics expansion suited for Maxwell equations,
which will be used frequently in this paper as well.

Now in this paper, which is the third and last one of our little series,
we analyze the solution formula \eqref{bddom} and try to give meaning to it in exterior domains
in the sense of an asymptotic expansion
\beq\loesom+(-\ie\omega)^\me\Pi-\sum_{j=0}^{\fJ-1}(-\ie\omega)^j\loesn\loes^{j}\Pi_{\reg}=\calO\big(|\omega|^\fJ\big)\qtext{,}\fJ\in\nzn\quad.\mylabel{exdom}\eeq
Thereby we follow closely the ideas of Weck and Witsch from \cite{complete} and \cite{linelae,linelaz}.
Due to our exterior boundary value problems there arise three major complications:

\begin{enumerate}
\item With growing $\fJ$ we have to use stronger norms for the data and obtain estimates in weaker norms
for the solutions.
\item As $\Pi$ and $\Pi_{\reg}$ already indicate we need weighted Hodge-Helmholtz decompositions
of $\Lzqqpesom$ respecting inhomogeneities $\Lambda$\,. In \cite{paulydeco} we presented results,
which will meet our needs. In fact we proved topological direct decompositions
$$\Lzqqpesom=\big(\Lambda\triqsom\dotplus\regqom{-1}{s}\big)\cap\Lzqqpesom\qquad,$$
where $\triqsom=\Pi\,\Lzqqpesom$ and $\regqom{-1}{s}=\Pi_{\reg}\Lzqqpesom$\,. We note
\begin{align*}
\triqsom&\subset\ronqtnom\times\dqpetnom\qquad,\\
\regqom{-1}{s}&\subset\regqntom\subset\dqtnom\times\ronqpetnom
\end{align*}
are only subspaces of $\Lzqqpetom$ with $t\leq s$ and $t<N/2$ and even \ul{not}
of $\Lzqqpesom$ if $s\geq N/2$\,. \big(See Lemma \ref{decolem}, \cite{paulydeco}\big)
\item We have to correct \eqref{exdom} by special operators $\Gamma_j$\,.
\end{enumerate}
More precisely for $\fJ\in\nzn$ and $s,-t>1/2$ we shall look for asymptotic estimates like
\begin{align}\begin{split}
&\qquad\big|\!\big|\loesom\FG+(-\ie\omega)^\me\Pi\FG-\sum_{j=0}^{\fJ-1}(-\ie\omega)^j\loesn\loes^{j}\Pi_{\reg}\FG\\
&-\sum_{j=0}^{\fJ-N}(-\ie\omega)^{j+N-1}\Gamma_j\FG\big|\!\big|_{\Lzqqpetom}
=\calO\big(|\omega|^\fJ\big)\normb{\FG}_{\Lzqqpesom}\qquad.
\end{split}\mylabel{asymcor}\end{align}
Hereby the $\calO$-symbols are always meant for frequencies $\omega\to0$
and uniformly with respect to $\omega\in\czpomd\ohne\{0\}$ and $\FG$\,,
where $\czpomd:=\setb{\omega\in\czp}{|\omega|\leq\hat{\omega}}$ for some $\hat{\omega}>0$\,.

Throughout this paper we will make the following\vspace*{2mm}\\
{\bf General Assumptions:}
\begin{itemize}
\item We restrict our considerations to ranks of forms
$$1\leq q\leq N-2$$
and \ul{odd} space dimensions $\nz\ni N\geq3$\,. Hence of course the most interesting case
of the classical Maxwell equations, $N=3$ and $q=1$\,, is covered.
The treatment of even dimensions (especially $N=2$) would
increase the complexity of our calculations considerably due to the appearance of logarithmic
terms in the fundamental solution for Helmholtz' equation (Hankel functions).
But there is no reasonable doubt that our methods can be used to obtain similar results
in any even dimension as well.
\item We fix a radius $r_0>0$ and some radii $r_n:=2^nr_0$\,, $n\in\nz$\,, such that
$\rN\ohne\Omega\subset U_{r_0}$\,. Moreover, we remind of the cut-off functions
$\mbox{\boldmath$\eta$}$\,, $\hat{\eta}$ and $\eta$ from {\paulytimeharmetadef}.
$\eta$ satisfies $\supp\eta=\ol{A_{r_1}}$\,, $\supp\nabla\eta=\ol{A_{r_1}\cap U_{r_2}}$\,.
Here $U_r:=\setb{x\in\rN}{|x|<r}$ and $A_r:=\setb{x\in\rN}{|x|>r}$ for $r>0$\,.

\item For simplicity $\Omega\subset\rN$ may have a Lipschitz boundary.
In fact, $\om$ only needs to have the Maxwell local compactness property {\sf MLCP}
from \paulytimeharmmlcpdef, i.e. the inclusions
$$\Ronqom{}\cap\Dqom{}\hookrightarrow\Lzqloc(\omq)$$
have to be compact for all $q$\,, as well as the static Maxwell property {\sf SMP}
from \paulystaticsecgenstatic, i.e.
the existence of special forms $\bonqom$ and $\bqpeom$ must be guaranteed.
Anyhow, Lipschitz domains possess these properties.
\item We assume $\eps=\id+\epsd$ and $\mu=\id+\mud$ to be $\tau$-$\pc{1}$-admissible
\big(see \paulytimeharmdefeps\big) linear transformations on $q$- resp. $(q+1)$-forms
with some rate of decay $\tau>0$\,, which will vary throughout this paper.
The greek letter $\tau$ always stands for the order of decay of the perturbations $\epsd$ and $\mud$\,.
Clearly we then have $\Lambda=\id+\Lambdad$\,.
Hence our transformations may have $\text{\rm L}^\infty$-entries,
which are only assumed to be $\pc{1}$ in $A_{r_0}$ and asymptotically homogeneous, i.e.
$\p^\alpha\Lambdad$ decays like $r^{-\tau-|\alpha|}$ for all multi-indices $\alpha$ with $|\alpha|\leq1$
with some order $\tau$ at infinity.
\end{itemize}

We may describe our results briefly as follows:

\begin{enumerate}
\item We shall identify \ul{de}g\ul{enerate} correction operators $\Gamma_j$ by a recursion which involves only
special solutions $\Esm$ and $\Hsn$ as well as their powers $\Esmk=\loes^{k}\Lambda(\Esm,0)$ and
$\Hsnk=\loes^{k}\Lambda(0,\Hsn)$ of $\loes$
of our homogeneous static boundary value problems
\begin{align*}
\rot\Esm&=0&&,&\pdiv\eps\Esm&=0&&,&\iota^*\Esm&=0&&,&\eps\Esm&\bot\,\bonqom&&,\\
\pdiv\Hsn&=0&&,&\rot\mu\Hsn&=0&&,&\iota^*\mu\Hsn&=0&&,&\mu\Hsn&\bot\,\bqpeom&&,
\end{align*}
but with inhomogeneities at infinity, namely
\begin{align*}
\Esm&-\turmd{q}{0}{\sigma}{m}{+}&&\text{'decays', i.e. belongs to}&&\Lzqom{>-\Nh}&&,\\
\Hsn&-\turmr{q+1}{0}{\sigma}{n}{+}&&\text{'decays', i.e. belongs to}&&\Lzqpeom{>-\Nh}&&,
\end{align*}
where the special growing tower forms $\turmd{q}{k}{\sigma}{m}{+}$ and $\turmr{q+1}{k}{\sigma}{n}{+}$
from \paulystaticsectower\, behave like $r^{k+\sigma}$ at infinity.
\big(See Lemma \ref{EsmHgnlemma}, Remark \ref{EsmHgnlemmarem}, Lemma \ref{loesEsmHgnlemma},
Definition \ref{asymsatzlokaldef}, Definition \ref{asymsatzlokalzweidef}\big)
\item On the `trivial' subspace $\triqsom$ the solution operator $\loesom$
behaves like the division by the frequency, i.e.
$$\ie\omega\loesom\Lambda\FG=\FG\qqtext{,}\forall\quad\FG\in\triqsom\qquad.$$
\big(See \eqref{loesompi}\big)
\item We shall identify closed subspaces $\regqJsom$ of $\Lzqqpesom$ \big(and of $\regqnsom$\big)\,,
the `spaces of regular convergence', for whose elements $\FG$ the `usual' Neumann expansion
$$\big|\!\big|\loesom\FG-\sum_{j=0}^{\fJ-1}(-\ie\omega)^j\loesn\loes^{j}\FG\big|\!\big|_{\Lzqqpetom}
=\calO\big(|\omega|^\fJ\big)\normb{\FG}_{\Lzqqpesom}$$
holds true. We are also able to characterize the spaces of regular convergence by orthogonality relations
with the aid of the special static solutions $\Esmk$ and $\Hsnk$\,.
\big(See Theorem \ref{Regsatz}, Lemma \ref{Regortho}\big)
\item For $\FG\in\regqom{-1}{s}$ we obtain the corrected Neumann expansion
$$\big|\!\big|\loesom\FG-\sum_{j=0}^{\fJ-1}(-\ie\omega)^j\loesn\loes^{j}\FG
-\sum_{j=0}^{\fJ-N}(-\ie\omega)^{j+N-1}\Gamma_j\FG\big|\!\big|_{\Lzqqpetom}$$
$$=\calO\big(|\omega|^\fJ\big)\normb{\FG}_{\Lzqqpesom}$$
and for general $\FG\in\Lzqqpesom$ we get the fully corrected Neumann expansion \eqref{asymcor}.
(See Main Theorem)
\item Concerning our media $\Lambda=\id+\Lambdad$
we shall distinguish between two kinds of assumptions on our inhomogeneities:
\begin{enumerate}
\item $\Lambdad$ has compact support. Then we always may choose $r_0$ in a way,
such that $\supp\Lambdad\subset U_{r_0}$\,.
\item $\Lambdad$ `decays' with a rate $\tau>0$ at infinity
in the sense of $\Lambda$ is $\tau$-$\pc{1}$-admissible.
\end{enumerate}
In the first case our results will hold for any $\fJ$\,, whereas in the second case only $\fJ\leq\hat{\fJ}$
with some $\hat{\fJ}$ depending on $\tau$ are allowed.
\end{enumerate}

Due to \cite{paulystatic}, \cite{sphharm} and originally \cite{mcowen} we have to exclude
a discrete set of `bad' weights, namely
$$\pI:=(\nzn+N/2)\cup(1-N/2-\nzn)=\set{n+N/2,1-n-N/2}{n\in\nzn}\quad.$$
Our main result of this paper reads as follows:

\paulesatz{Main Theorem}{0}{1.5}{5}{5}{
Let $\fJ\in\nz$ and $s\in\rz\ohne\pI$ as well as
\begin{align*}
s&>\fJ+1/2\qquad,\\
t&<\min\{N/2-\fJ-2\,,\,-1/2\}\qquad,\\
\tau&>\max\big\{(N+1)/2\,,\,s-t\big\}\qquad.
\end{align*}
Then there exists some $\hat{\omega}>0$\,, such that the asymptotic
$$\loesom+(-\ie\omega)^\me\Pi-\sum_{j=0}^{\fJ-1}(-\ie\omega)^j\loesn\loes^{j}\Pi_{\reg}
-\sum_{j=0}^{\fJ-N}(-\ie\omega)^{j+N-1}\Gamma_j=\calO\big(|\omega|^\fJ\big)$$
holds uniformly with respect to $\czpomd\ohne\{0\}\ni\omega\to0$ in the norm of bounded linear operators
from $\Lzqqpesom$ to $\Lzqqpetom$\,.
This asymptotic also holds true in the special case $\fJ=0$\,,
if we replace the assumptions on $t$ and $\tau$ by $t\leq s-(N+1)/2$ and $t<-1/2$ as well as
$\tau>\max\big\{(N+1)/2,s+1-N/2\big\}$\,.
}

We note that the correction operators $\Gamma_j$ only occur for asymptotic orders $\fJ\geq N$\,.

\paulesatz{Remark A}{0}{1.5}{5}{0}{
These asymptotics remain valid even in stronger norms.
In the norm of bounded linear operators from $\Lzqqpesom$ to $\Ronqtom\times\Dqpetom$ we obtain 
the desired estimate if we assume additionally $t\leq s-(N+1)/2$ in the case $\fJ=1$\,.
In the (strongest) norm of bounded linear operators from $\Lzqqpesom$ to $\ronqtom\times\dqpetom$
the asymptotic estimate holds true if we assume additionally $t<-3/2$
and moreover $t\leq s-(N+3)/2$ if $\fJ\in\{0,1\}$\,.
}

\paulesatz{Remark B}{0}{1.5}{5}{0}{
Using the estimate (ii) instead of (i) from Theorem \ref{Regsatz}
during our considerations in section \ref{asymsec} we would
achieve asymptotics with the small $o$-symbol instead of $\calO$\,.
As an example we obtain (for nearly the same $s$\,, $t$ and $\tau$)
$$\loesom+(-\ie\omega)^\me\Pi-\sum_{j=0}^{\fJ}(-\ie\omega)^j\loesn\loes^{j}\Pi_{\reg}
-\sum_{j=0}^{\fJ-N+1}(-\ie\omega)^{j+N-1}\Gamma_j=o\big(|\omega|^\fJ\big)\quad.$$
}

Choosing here $\fJ=1$ we may easily conclude the differentiability of $\loesom$ in $\omega=0$
as an operator acting on $\regqom{-1}{s}=\Pi_{\reg}\Lzqqpesom$\,. We obtain

\paulesatz{Corollary}{0}{1.5}{5}{0}{
Let $s\in(3/2,\infty)\ohne\pI$\,, $t<\min\{N/2-3,-1/2\}$ and $\tau>\max\big\{(N+1)/2\,,\,s-t\big\}$\,. Then
$$\czpomd\ni\omega\longmapsto\ie\loesom\in B\big(\regqom{-1}{s},\Ronqtom\times\Dqpetom\big)$$
is differentiable in $\omega=0$ with derivative $\Lambda^\me\loes^2$\,.
}

\paulesatz{Remark C}{0}{1.5}{5}{0}{
Formally our solution of \eqref{Maxeq} satisfies the perturbed Helmholtz type equation
$$\Big(\zmat{\eps^\me\pdiv\mu^\me\rot}{0}{0}{\mu^\me\rot\eps^\me\pdiv}+\omega^2\Big)\EH
=(\Lambda^\me M-\ie\omega)\Lambda^\me\FG$$
and
$$\ie\omega\pdiv\eps E=\pdiv F\qqtext{,}\ie\omega\rot\mu H=\rot G\qquad,$$
which imply
$$\Big(\zmat{\eps^\me\pdiv\mu^\me\rot+\rot\pdiv\eps}{0}{0}{\mu^\me\rot\eps^\me\pdiv+\pdiv\rot\mu}
+\omega^2\Big)\EH$$
$$=\Big(\frac{1}{\ie\omega}\zmat{\rot\pdiv}{0}{0}{\pdiv\rot}\Lambda+\Lambda^\me M-\ie\omega\Big)
\Lambda^\me\FG\qquad.$$
We note $\Delta=\pdiv\rot+\rot\pdiv$\,.
Due to this formula and under certain regularity restrictions for $\FG$
the cases $q=0$ and $q=N-1$ are equivalent to scalar (perturbed) Helmholtz problems for $E$ and $H$\,,
since $E$ for $q=0$ and $H$ for $q=N-1$ are scalar functions.
The first case $q=0$\,, i.e. a Helmholtz equation with homogeneous Dirichlet boundary condition for $E$\,,
has already been discussed in \cite{complete} or \cite{linelaz} and
even for the most complicated case $N=2$ in \cite{peter}.
The other case $q=N-1$ corresponds to a Helmholtz equation with homogeneous or inhomogeneous
Dirichlet boundary condition $\ie\omega\mu H=G$ on $\p\om$ for $H$
and can be handled analogously to the case $q=0$ using an adequate extension operator.

However, also in the case $q=N-1$ our techniques work.
The only difference is that now some exceptional `tower forms' occur.
Due to their appearance we have to tackle some additional difficulties
and the correction operators occur already at the power $\omega^{N-2}$ instead of $\omega^{N-1}$
in the case $1\leq q\leq N-2$\,. At this point we note in passing
that for more regular data from $\regqnsom$ the correction operators appear primarily
at the power $\omega^{N}$ for $1\leq q\leq N-2$ and at $\omega^{N-1}$ if $q=N-1$\,.
}

\paulesatz{Remark D}{0}{1.5}{5}{0}{
We easily obtain a low frequency asymptotic for inhomogeneous boundary data as well.
For this purpose we utilize the tangential trace and extension operators
$\gt:=\mathcal{T}$ and $\chgt:=\mathcal{T}^\me$\,,
which recently have been studied by Weck \cite{wecklip}.
Let us remind of the solution operator
$$\Abb{\Loesom}{\Lzqqpesom\times\cRqdom}{\Rqtom\times\Dqpetom\subset\Lzqqpetom}{(F,G,\lambda)}{\EH}$$
for $s,-t>1/2$ from \paulytimeharmsecboundary\, of the Maxwell system
$$(M+\ie\omega\Lambda)\EH=\FG\in\Lzqqpesom\qqtext{,}\gt E=\lambda\in\cRqdom\qquad,$$
where $\cRqdom:=R^{-1/2,q}(\dom)=\setb{\varkappa\in\qh{\meh}{q}{}{}(\dom)}{\Rot\varkappa\in\qh{\meh}{q+1}{}{}(\dom)}$
and $\Rot:=\pd$ denotes the exterior derivative on the boundary manifold $\dom$\,.
Again $\Loesom$ is well defined for small frequencies $\omega$ and connected to $\loesom$ via
$$\Loesom(F,G,\lambda)=(\chgt\lambda,0)+\loesom\FG-\ie\omega\loesom(\eps\chgt\lambda,0)-\loesom(0,\rot\chgt\lambda)\qquad.$$
Thus $\Loesom$ inherits its low frequency behavior from $\loesom$\,.
We note that $\chgt\lambda\in\rqvoxom$ and
$$(\eps\chgt\lambda,\rot\chgt\lambda)\in\dqvoxom\times\rqpevoxnom\qquad.$$
If we assume a more regular boundary $\dom$\,, e.g. $\pc{3}$\,,
and a slightly more regular dielectrical coefficient $\eps$\,, i.e. $\pc{1}$ at least in an arbitrarily
thin shell of $\om$\,, we can also guarantee $\eps\chgt\lambda\in\dqvoxnom\cap\bonqom^\bot$
by a modification of the extension operator $\chgt$\,, i.e.
$$(\eps\chgt\lambda,0)\in\regqnom{\vox}\qquad.$$
Due to $\gt\rot=\Rot\gt$\,,
i.e. the tangential trace commutes with the exterior derivative,
the condition $\Rot\lambda=0$ would imply homogeneous boundary data for the form $\rot\chgt\lambda$\,,
i.e. $\rot\chgt\lambda\in\ronqpevoxnom$\,. Furthermore, for $b\in\bqpeom$
$$\skp{\rot\chgt\lambda}{b}_{\Lzqpeom{}}=\skp{\lambda}{\gn b}_{\Lzq{}(\dom)}$$
holds by Stokes theorem in the sense of the $\qh{\meh}{q}{}{}(\dom)$-$\qh{\frac{1}{2}}{q}{}{}(\dom)$-duality.
Hence, we have $\rot\chgt\lambda\bot\bqpeom$\,, if and only if $\lambda\bot\gn\bqpeom$\,.
Here $\gn:=\pm\circledast\gt*$ denotes the normal trace and
$\circledast$ the Hodge star operator on the manifold $\dom$\,.
We note that by the regularity assumption for the boundary we have the inclusion $\bqpeom\subset\qhom{1}{q+1}{}{}$
and thus $\gn b\in\qh{\frac{1}{2}}{q}{}{}(\dom)$\,. (Here $\dom\in\pc{2}$ is sufficient.)
Thus, for $\lambda\in\cRqdom$ perpendicular to $\gn\bqpeom$ and
with vanishing rotation we also have $(0,\rot\chgt\lambda)\in\regqnom{\vox}$\,, i.e.
$$(\eps\chgt\lambda,\rot\chgt\lambda)\in\regqnom{\vox}\qquad.$$
These features of $\chgt$ and constraints for $\lambda$ would enhance the asymptotic since then the pair
$(\eps\chgt\lambda,\rot\chgt\lambda)$ would be an element
of the kernels of $\Pi$ and $\tilde{\Gamma}_j$
(See Definition \ref{asymsatzlokalzweidef} and Remark \ref{asymsatzlokalzweirem}).
Moreover, it might be of interest to know if there are modifications of the extension operator $\chgt$
and constraints for the boundary data $\lambda$\,, which even imply
\beq(\eps\chgt\lambda,\rot\chgt\lambda)\in\regqom{\fJ}{\vox}\qqtext{,}\fJ\geq1\qquad.\mylabel{chgtlregJ}\eeq
This property would enhance the asymptotic of $\Loesom$ once more because then all
correction operators would vanish on the extended boundary forms.
We can give a positive answer to this question as well.
By Lemma \ref{Regortho} and Remark \ref{Regorthobem} \eqref{chgtlregJ} holds, if and only if
$(\eps\chgt\lambda,\rot\chgt\lambda)$ is an element of $\regqnom{\vox}$ and perpendicular
\big(in $\Lzqqpeom{}$\big) to all
$\Lambda^\me\Esmj\,,\,\Lambda^\me\Hsnj\in\Lzqqpeom{-s}$ with $1\leq j\leq\fJ$\,.
Again these constraints can be transfered to $\lambda$\,.
For $\pc{\fJ+1}$ regular coefficients $\Lambda$ (at least in a thin shell) it is possible to modify $\chgt$ in a way
that \eqref{chgtlregJ} holds, if $\lambda\in\cRqdom$ is irrotational as well as
perpendicular \big(in $\Lzq{}(\dom)$\big) to $\gn\bqpeom$ and
$\gn\pi_2\Lambda^\me\Esmj\,,\,\gn\pi_2\Lambda^\me\Hsnj$
for all $\Esmj,\Hsmj\in\Lzqqpeom{-s}$ with $1\leq j\leq\fJ$\,. Here $\pi_2$ denotes the projection on the second component.
More details can be found in \cite[chapter 4]{paulyasymrep}.
We note that only the term $\rot\chgt\lambda$ produces the constraints for $\lambda$\,.

It is an interesting open question if all these properties hold for Lipschitz boundaries as well.
}

\section{The spaces of regular convergence}

First let us remind of the special tower forms
$$\turmd{q}{k}{\sigma}{m}{\pm}\qqtext{,}\turmr{q}{k}{\sigma}{m}{\pm}$$
and their properties from \paulystaticsectower, which will be used frequently throughout this paper.
The main tool for their construction is the spherical coordinate calculus developed in \cite{sphharm}.
Hence we shall use also many notations and results from this paper.
From this point of view the paper at hand demonstrates also an application of \cite{sphharm}.

Utilizing \paulystaticitloescor\, we define some special data spaces recursively by

\begin{defini}\mylabel{Regdefi}
Let $\fJ\in\nz$ and $s\in(\fJ-N/2,\infty)\ohne\pI$ as well as $\tau>\max\{0,s-N/2\}$ and $\tau\geq\fJ-s-1$\,.
For $j=1,\dots,\fJ$ we define the `spaces of regular convergence' via
\begin{align*}
\regqnsom&:=\bDqsom\times\bRqpesom\qquad,\\
\regqjsom&:=\setb{\FG\in\regqom{j-1}{s}}{\loes^j\FG\in\regqnom{s-j}}\qquad.
\end{align*}
We will call $\regqjsom$ the `space of regular convergence of order $j$'.
\end{defini}

\begin{rem}\mylabel{Regdefibem}
We have
$$\regqjsom:=\setb{\FG\in\regqnom{s}}{\loes^j\FG\in\Lzqqpeom{s-j}}\qquad.$$
In words, the space of regular convergence $\regqJsom$ is characterized by the following property:
For $\FG\in\regqJsom$ and $j=0,\dots,\fJ$ no tower-forms $\eta D^q_I$ or $\eta R^{q+1}_J$
appear in the powers $\loes^j\FG$\,.
\end{rem}

Clearly for the selfadjoint operator $\calM$ introduced in \cite{paulytimeharm} the resolvent-formula holds for nonreal frequencies.
Our next step is to show that this formula still holds true for real frequencies and $\loesom$ acting on $\regqJsom$ up to the
order $\fJ$\,. Then $\loesom$ is approximated by the usual Neumann sum up to the order $\fJ$\,.

For the purpose of a short notation let us put for $J\in\nzn$
$$\loes_{\omega,J}:=\loesom-\sum_{j=0}^{J}(-\ie\omega)^j\loesn\loes^j\qqtext{,}\loes_{\omega,-1}:=\loesom\qquad.$$

\begin{theo}\mylabel{Regsatz}
Let $\fJ\in\nzn$\,, $s\in(\fJ+1/2,\infty)\ohne\pI$ and $\tau>\max\big\{(N+1)/2,s-N/2\big\}$\,.
Moreover, let $\hat{\omega}$ be as in \paulytimeharmPzero.
Then for all $\omega\in\czpomd\ohne\{0\}$ on $\regqJsom$
$$\loes_{\omega,\fJ-1}=(-\ie\omega)^\fJ\loesom\loes^\fJ\qqtext{,}\loes_{\omega,\fJ}=(-\ie\omega)^\fJ(\loesom-\loesn)\loes^\fJ\qquad.$$
Furthermore, for $s\in(\fJ+1/2,\fJ+N/2)\ohne\pI$ and $\tilde{t}<t:=s-\fJ-(N+1)/2$
\begin{align*}
\text{\rm\bf (i)}&&\normb{\loes_{\omega,\fJ-1}\FG}_{\Lzqqpetom}&=\calO\big(|\omega|^\fJ\big)\normb{\FG}_{\Lzqqpesom}\qquad,\\
\text{\rm\bf (ii)}&&\normb{\loes_{\omega,\fJ}\FG}_{\Lzqqpeom{\tilde{t}}}&=o\big(|\omega|^\fJ\big)\normb{\FG}_{\Lzqqpesom}
\end{align*}
hold uniformly with respect to $\FG\in\regqJsom$\,.
\end{theo}

\begin{proof}
$\loesom\loes^\fJ\FG$ is well defined by \paulytimeharmtheofred, since
$$\loes^j\FG\in\regqnom{s-j}\subset\Lzqqpeom{>\frac{1}{2}}$$
holds for $j=0,\dots,\fJ$\,. Thus also
$$\EH:=\sum_{j=0}^{\fJ-1}(-\ie\omega)^j\loesn\loes^j\FG+(-\ie\omega)^\fJ\loesom\loes^\fJ\FG$$
is well defined. Because of $s>\fJ+1/2>\fJ+1-N/2$ even $\fJ+1$ powers of $\loes$ may be applied to $\FG$
by \paulystaticitloescor. We get
$$\EH=\sum_{j=0}^{\fJ}(-\ie\omega)^j\loesn\loes^j\FG+(-\ie\omega)^\fJ(\loesom-\loesn)\loes^\fJ\FG\qquad.$$
Furthermore, $\EH\in\Ronqkmehom\times\Dqpekmehom$ satisfies the radiation condition.
Since $M\loesn=\id$ and $(M+\ie\omega\Lambda)\loesom=\id$ we obtain $(M+\ie\omega\Lambda)\EH=\FG$\,,
which yields $\EH=\loesom\FG$\,.

Noting $s-\fJ\in(1/2,N/2)$ we may apply \paulytimeharmPzeroiv\,
to $\loes^\fJ\FG\in\regqnom{s-\fJ}$ and obtain uniformly in $\omega\in\czpomd\ohne\{0\}$ and $\FG\in\regqJsom$ the estimate
$$\normb{\loesom\loes^\fJ\FG}_{\Lzqqpetom}\leq c\normb{\loes^\fJ\FG}_{\Lzqqpeom{s-\fJ}}\leq c\normb{\FG}_{\Lzqqpesom}\qquad.$$
\big(We observe that $\loes^\fJ$ is continuous by \paulystaticitloescor.\big)
Analogously using \paulytimeharmcorlomcont\, we may estimate again uniformly in $\omega\in\czpomd\ohne\{0\}$ and $\FG\in\regqJsom$
\begin{align*}
\normb{(\loesom-\loesn)\loes^\fJ\FG}_{\Lzqqpeom{\tilde{t}}}
&\leq\norm{\loesom-\loesn}_{B_{s-\fJ,\tilde{t}}}\normb{\loes^\fJ\FG}_{\Lzqqpeom{s-\fJ}}\\
&\leq c\norm{\loesom-\loesn}_{B_{s-\fJ,\tilde{t}}}\normb{\FG}_{\Lzqqpesom}\qquad,
\end{align*}
which completes the proof since $\norm{\loesom-\loesn}_{B_{s-\fJ,\tilde{t}}}\xrightarrow{\omega\to0}0$\,.
\end{proof}

Now our aim is to characterize $\regqJsom$ utilizing orthogonality constraints.
To realize this we need special growing Dirichlet-forms, which will be defined in the following lemma.
Let us remind of the topological isomorphisms
$$\statMax_\eps:={}_\eps\statMax^q_{s-1}\qqtext{,}\widetilde{\statMax}_\mu:={}^\mu\statMax^{q+1}_{s-1}$$
for some $s\in(1-N/2,\infty)\ohne\pI$ introduced in \paulystatictheoaltmax.

\begin{lem}\mylabel{EsmHgnlemma}
Let $\sigma\in\nzn$ as well as $\tau>\sigma$ and $\tau\geq N/2-1$\,. Then for all indices
$(m,n)\in\{1,\dots,\mu^q_\sigma\}\times\{1,\dots,\mu^{q+1}_\sigma\}$ the `special growing Dirichlet-forms'
\begin{align*}
\Esm\,:=&\,(\id-\statMax_\eps^\me\statMax_\eps)\eta\,\turmd{q}{0}{\sigma}{m}{+}\qquad,\\
\Hsn\,:=&\,(\id-\widetilde{\statMax}_\mu^\me\widetilde{\statMax}_\mu)\eta\,\turmr{q+1}{0}{\sigma}{n}{+}
\end{align*}
are well defined and belong to $\Lzqom{<-\Nh-\sigma}$ resp. $\Lzqpeom{<-\Nh-\sigma}$\,.
These are the unique solutions of the electro-magneto static problems
\begin{align*}
\statMax_\eps\Esm&=(0,0,0)&&,&\Esm-\turmd{q}{0}{\sigma}{m}{+}&\,\text{ decays}&&,\\
\widetilde{\statMax}_\mu\Hsn&=(0,0,0)&&,&\Hsn-\turmr{q+1}{0}{\sigma}{n}{+}&\,\text{ decays}&&.
\end{align*}
\end{lem}

\begin{rem}\mylabel{EsmHgnlemmarem}
To be more precise: $\Esm$ and $\Hsn$ are the unique solutions of
\begin{align*}
\Esm\,\in&\,\,\dhqepsom{\loc}\cap\bonqom^{\bot_\eps}&&,&\Esm-\turmd{q}{0}{\sigma}{m}{+}&\in\Lzqom{>-\Nh}&&,\\
\mu\Hsn\,\in&\,\,\dH{q+1}{\loc}{\mu^\me}(\Omega)\cap\bqpeom^\bot&&,&\Hsn-\turmr{q+1}{0}{\sigma}{n}{+}&\in\Lzqpeom{>-\Nh}&&.
\end{align*}
Here on one hand we used $\statMax_\eps$ resp. $\widetilde{\statMax}_\mu$ as formal mappings, e.g.
$$\statMax_\eps=\big(\pdiv\eps\,\cdot\,,\rot\,\cdot\,,
\skp{\eps\,\cdot\,}{\bonq_1},\dots,\skp{\eps\,\cdot\,}{\bonq_{d^q}}\big)\qquad,$$
and on the other hand $\statMax_\eps^\me$ resp. $\widetilde{\statMax}_\mu^\me$ as the inverse operators.
\end{rem}

\begin{proof}
Uniqueness is clear by \paulystaticintdirichleteps\, and the properties of $\bonqom$\,, $\bqpeom$\,,
see \paulystaticbFormendirichletsenkrecht. Let us assume the well definedness of $\Esm$ for a moment.
Since $\statMax_\eps\Esm=(0,0,0)$ holds we obtain $\Esm\in\dhqepsom{\loc}\cap\bonqom^{\bot_\eps}$\,.
Moreover, we have $\Esm-\turmd{q}{0}{\sigma}{m}{+}\in\Lzqom{>-\Nh}$ because $\statMax_\eps^\me$ maps
in fact to $\Lzqom{>-\Nh}$ and finally the integrability of $\Esm$ is determined by the form
$\eta\turmd{q}{0}{\sigma}{m}{+}$\,, which belongs to $\Lzqom{<-\Nh-\sigma}$ by \paulystaticinttower.
Analogously we may handle $\Hsn$\,, which would prove the lemma.

So it remains to show that $\Esm$ is well defined.
(Then surely $\Hsn$ is well defined as well by similar arguments.)
With $\supp\eta\cap\supp\bon{q}_\ell=\emptyset$ and \paulystaticinttower
\begin{align*}
\statMax_\eps\eta\turmd{q}{0}{\sigma}{m}{+}=&\,\big(\pdiv(\eps\eta\turmd{q}{0}{\sigma}{m}{+}),\rot(\eta\turmd{q}{0}{\sigma}{m}{+}),0\big)\\
=&\,\big(C_{\pdiv,\eta}\turmd{q}{0}{\sigma}{m}{+}+\pdiv(\epsd\eta\turmd{q}{0}{\sigma}{m}{+}),C_{\rot,\eta}\turmd{q}{0}{\sigma}{m}{+},0\big)\\
\in&\,\,\bDom{q-1}{<-\sigma-\Nh+\tau+1}{0}\times\bRom{q+1}{\vox}{0}\times\cz^{d^q}
\end{align*}
holds. Now $-\sigma-N/2+\tau+1>1-N/2$ since $\tau>\sigma$ and thus
$$\statMax_\eps\eta\turmd{q}{0}{\sigma}{m}{+}\in\bWom{q}{>1-\Nh}\qquad.$$
Because also $\tau\geq N/2-1$ and using \paulystatictheoaltmax\, we get the well definedness of
$\statMax_\eps^\me\statMax_\eps\eta\turmd{q}{0}{\sigma}{m}{+}$ and thus of $\Esm$\,.
We note
\beq\statMax_\eps\eta\turmd{q}{0}{\sigma}{m}{+}\in\bWqsom\qquad,\mylabel{statMaxepsints}\eeq
if $\tau>s+\sigma+N/2-1$\,.
\end{proof}

Our next step is to define powers of $\loes$ on the special forms
$\Lambda(\Esm,0)$ and $\Lambda(0,\Hsn)$ using \paulystaticitloes.
Let us introduce a new notation. For $k\in\nzn$ we define
$$\Esmk:=\loes^{k}\Lambda(\Esm,0)\qqtext{,}\Hsmk:=\loes^{k}\Lambda(0,\Hsm)\qquad.$$
The next lemma shows that these definitions are well defined.

\begin{lem}\mylabel{loesEsmHgnlemma}
Let $\fJ,\sigma\in\nzn$\,, $s\in(\fJ+1-N/2,\infty)\ohne\pI$ as well as
$\tau>\sigma+s+N/2-1$ and $\tau\geq\fJ-s$\,. Moreover, let
$(m,n)\in\{1,\dots,\mu^q_\sigma\}\times\{1,\dots,\mu^{q+1}_\sigma\}$ and $j\in\{0,\dots,\fJ\}$\,.
Then $\fJ$ powers of $\loes$ on $\Esmn$ and $\Hsnn$ are well defined and for \ul{even} $j$
\begin{align*}
\Esmj-\eta(\turmd{q}{j}{\sigma}{m}{+},0)
&\in\big(\dqom{s-j-1}\boxplus\eta\calD^q(\cIb^{q,\leq j}_{s-j-1})\big)\times\{0\}\qquad,\\
\Hsnj-\eta(0,\turmr{q+1}{j}{\sigma}{n}{+})
&\in\{0\}\times\big(\ronqpeom{s-j-1}\boxplus\eta\calR^{q+1}(\cJb^{q+1,\leq j}_{s-j-1})\big)
\intertext{as well as for \ul{odd} $j$}
\Esmj-\eta(0,\turmr{q+1}{j}{\sigma}{m}{+})
&\in\{0\}\times\big(\ronqpeom{s-j-1}\boxplus\eta\calR^{q+1}(\cJb^{q+1,\leq j}_{s-j-1})\big)\qquad,\\
\Hsnj-\eta(\turmd{q}{j}{\sigma}{n}{+},0)
&\in\big(\dqom{s-j-1}\boxplus\eta\calD^q(\cIb^{q,\leq j}_{s-j-1})\big)\times\{0\}
\end{align*}
hold. Furthermore,
$$\Esmj,\Hsnj\in\Lzqqpeom{<-\sigma-j-\Nh}$$
and thus
$$\Esmj,\Hsnj\in\Lzqqpeom{-t}\qquad\Equi\qquad t>\sigma+j+N/2\qquad.$$
More precisely: There exist unique constants $\xi^{j,\sigma,}_{},\zeta^{j,\sigma,}_{}\in\cz$
and unique forms
\begin{align*}
e^j_{\sigma,}&\in\big(\eps\ronqom{s-j-1}\big)\cap\dqom{s-j-1}\cap\bonqom^\bot\qquad,\\
h^j_{\sigma,}&\in\big(\mu\dqpeom{s-j-1}\big)\cap\ronqpeom{s-j-1}\cap\bqpeom^\bot\qquad,
\intertext{such that for \ul{even} $j$}
\Esmj&=\eta(\turmd{q}{j}{\sigma}{m}{+},0)+\sum_{I\in\cIb^{q,\leq j}_{s-j-1}}\xi^{j,\sigma,m}_{I}\eta(D^q_I,0)+(e^j_{\sigma,m},0)\qquad,\\
\Hsnj&=\eta(0,\turmr{q+1}{j}{\sigma}{n}{+})+\sum_{J\in\cJb^{q+1,\leq j}_{s-j-1}}\zeta^{j,\sigma,n}_{J}\eta(0,R^{q+1}_J)+(0,h^j_{\sigma,n})
\intertext{and for \ul{odd} $j$}
\Esmj&=\eta(0,\turmr{q+1}{j}{\sigma}{m}{+})+\sum_{J\in\cJb^{q+1,\leq j}_{s-j-1}}\zeta^{j,\sigma,m}_{J}\eta(0,R^{q+1}_J)+(0,h^j_{\sigma,m})\qquad,\\
\Hsnj&=\eta(\turmd{q}{j}{\sigma}{n}{+},0)+\sum_{I\in\cIb^{q,\leq j}_{s-j-1}}\xi^{j,\sigma,n}_{I}\eta(D^q_I,0)+(e^j_{\sigma,n},0)\qquad.
\end{align*}
\end{lem}

\begin{rem}\mylabel{loesEsmHgnlemmabem}
By \paulystaticremodd\, we even have for \ul{odd} $j\leq\fJ$
\begin{align*}
\Esmj-\eta(0,\turmr{q+1}{j}{\sigma}{m}{+})
&\in\{0\}\times\bR{q+1}{s-j-1}{0}(\cJb^{q+1,\leq j}_{s-j-1},\om)\qquad,\\
\Hsnj-\eta(\turmd{q}{j}{\sigma}{n}{+},0)
&\in\bD{q}{s-j-1}{0}(\cIb^{q,\leq j}_{s-j-1},\om)\times\{0\}\qquad,
\end{align*}
since then also $\eta\turmr{q+1}{j}{\sigma}{m}{+}$ is irrotational and $\eta\turmd{q}{j}{\sigma}{n}{+}$ solenoidal.
Moreover, the coefficients satisfy the following recursion:
\begin{align*}
\zeta^{j+1,\sigma,\ell}_{{}_1I}&=\xi^{j,\sigma,\ell}_I&&,&I&\in\cIb^{q,\leq j}_{s-j-1}&&,&\ell&=m,n\qquad,\\
\xi^{j+1,\sigma,\ell}_{{}_1J}&=\zeta^{j,\sigma,\ell}_{J}&&,&J&\in\cJb^{q+1,\leq j}_{s-j-1}&&,&\ell&=m,n
\intertext{Then clearly the next recursion holds as well:}
\xi^{j+2,\sigma,\ell}_{{}_2I}&=\xi^{j,\sigma,\ell}_I&&,&I&\in\cIb^{q,\leq j}_{s-j-1}&&,&\ell&=m,n\qquad,\\
\zeta^{j+2,\sigma,\ell}_{{}_2J}&=\zeta^{j,\sigma,\ell}_J&&,&J&\in\cJb^{q+1,\leq j}_{s-j-1}&&,&\ell&=m,n
\end{align*}
\end{rem}

\begin{proof}
We only have to show that $\Esmn$ and $\Hsnn$ are elements of the domain of definition
of $\loes$ (and then clearly of $\loes^j$) from \paulystaticitloes. Then all our assertions follow
by \paulystaticitloes\,, \paulystaticitloesrem\, and \paulystaticinttower.
We note that the integrability of the forms is always determined by the integrability of the tower forms
with positive sign. Again we only discuss $\Esmn=(\eps\Esm,0)$\,, for example.

By Lemma \ref{EsmHgnlemma} and \eqref{statMaxepsints} as well as \paulystatictheoaltmax\, we observe
\beq\eps\Esm\in\bD{q}{s-1}{0}\big(\{I\}\cup\cIb^{q,0}_{s-1},\om\big)\qqtext{,}I:=(+,0,\sigma,m)\qquad,\mylabel{mistake}\eeq
by \paulystaticinttower\, since $\tau>\sigma+s+N/2-1$\,. Hence utilizing \paulystaticitloes\,
$\loes^j$ may be applied to $\Esmn$ and the lemma would be proved.

Unfortunately we ignored a trifle in this argument.
Here the same problem occurs as in \cite{paulystatic}, namely the appearance of the exceptional tower forms,
which was solved by a second order approach in this paper. A similar approach will help here.
The point is that \eqref{mistake} only holds true for $q\neq1$ in the first sight.
In fact for $q=1$ and $s\geq N/2$ we have to deal with the exceptional tower form
$\check{D}^{1,1}_{s-1}=\turmr{1}{1}{0}{1}{-}=R^1_{\check{I}}$ with $\check{I}:=(-,1,0,1)$\,,
which would cancel our iteration process in the case of appearance.
Now in this special case \eqref{mistake} reads correctly as:
$\eps\Esm$ is an element of $\bDom{1}{\loc}{0}$ and contained in
$$\big(\eps^\me\pr{1}{s-1}{}{\circ}(\om)\cap\pdi{1}{s-1}{}{}(\om)\big)\boxplus\eta\calD^1\big(\{I\}\cup\cIb^{q,0}_{s-1}\big)\boxplus\eta\calR^1\big(\{\check{I}\}\big)\qquad.$$
It remains to show that $R^1_{\check{I}}$ does not occur even in the exceptional case.
We try the ansatz
$$U_{\sigma,m}:=\eta\turmr{2}{1}{\sigma}{m}{+}+u_{\sigma,m}\qquad,$$
to find a solution of the problem
$$\rot\eps^\me\pdiv U_{\sigma,m}=0\qqtext{,}U_{\sigma,m}-\turmr{2}{1}{\sigma}{m}{+}\in\qLzom{2}{>-\Nh}\qquad.$$
Thus we are led to search a solution of
\beq\rot\eps^\me\pdiv u_{\sigma,m}=-\rot\eps^\me\pdiv\eta\turmr{2}{1}{\sigma}{m}{+}\qqtext{,}u_{\sigma,m}\in\qLzom{2}{>-\Nh}\qquad.\mylabel{asm}\eeq
Using once more $\tau>\sigma+s+N/2-1$ and \paulystaticinttower\, we obtain that
$$\rot\eps^\me\pdiv\eta\turmr{2}{1}{\sigma}{m}{+}=C_{\rot\pdiv,\eta}\turmr{2}{1}{\sigma}{m}{+}+\rot\epsdd\pdiv\eta\turmr{2}{1}{\sigma}{m}{+}$$
is an element of $\qLzom{2}{<-\sigma-\Nh+\tau+1}\subset\qLzom{2}{s}$\,, where $\eps^\me=\id+\,\epsdd$\, is $\tau$-$\pc{1}$-admissible as well.
Therefore, the $2$-form $\rot\eps^\me\pdiv\eta\turmr{2}{1}{\sigma}{m}{+}\in\bRom{2}{s}{0}$
lies in the range of ${}_{\rot}\Delta^{2}_{s-2}$ from \paulystaticsecordlemone\, and we get some
$u_{\sigma,m}\in D({}_{\rot}\Delta^{2}_{s-2})$ solving \eqref{asm}. But then
\begin{align*}
\tilde{E}_{\sigma,m}:=\pdiv U_{\sigma,m}
&\in\Big(\pdi{1}{s-1}{}{}(\om)\boxplus\eta\calD^1\big(\{I\}\cup\cIb^{q,0}_{s-1}\big)\Big)\cap\bDom{1}{\loc}{0}\\
&=\,\bD{1}{s-1}{0}\big(\{I\}\cup\cIb^{q,0}_{s-1},\om\big)
\end{align*}
\big(Compare to \paulystaticsecordrem.\big) and
$$\eps^\me\tilde{E}_{\sigma,m}\in\dH{1}{<-\Nh-\sigma}{\eps}(\Omega)\cap\pb{1}{\circ}(\Omega)^{\bot_\eps}\qtext{,}\eps^\me\tilde{E}_{\sigma,m}-\turmd{1}{0}{\sigma}{m}{+}\in\qLzom{1}{>-\Nh}\quad,$$
i.e. $\eps^\me\tilde{E}_{\sigma,m}=\Esm$ by Lemma \ref{EsmHgnlemma} and Remark \ref{EsmHgnlemmarem}.
So in fact $\Esm$ and $\eps\Esm$ do not feature exceptional tower forms.
\end{proof}

\subsection{Compactly supported inhomogeneities}

In this subsection we develop some results especially for compactly supported inhomogeneities $\Lambda$\,.
In fact we assume $r_0$ to be so large, such that
\beq\supp\Lambdad\subset U_{r_0}\mylabel{comper}\eeq
holds. Then in particular $\Lambda=\id$ on $\supp\eta$\,.

\begin{cor}\mylabel{loesEsmHgnkor}
Let $\fJ,\sigma\in\nzn$ and $(m,n)\in\{1,\dots,\mu^q_\sigma\}\times\{1,\dots,\mu^{q+1}_\sigma\}$ 
as well as $j$ in $\{0,\dots,\fJ\}$\,.
Then there exist unique constants $\xi^{j,\sigma,}_{},\zeta^{j,\sigma,}_{}\in\cz$\,,
such that in $\supp\eta$ for \ul{even} $j$
\begin{align*}
\Esmj&=(\turmd{q}{j}{\sigma}{m}{+},0)+\sum_{I\in\cIb^{q,\leq j}}\xi^{j,\sigma,m}_{I}(D^q_I,0)\qquad,\\
\Hsnj&=(0,\turmr{q+1}{j}{\sigma}{n}{+})+\sum_{J\in\cJb^{q+1,\leq j}}\zeta^{j,\sigma,n}_{J}(0,R^{q+1}_J)
\intertext{and for \ul{odd} $j$}
\Esmj&=(0,\turmr{q+1}{j}{\sigma}{m}{+})+\sum_{J\in\cJb^{q+1,\leq j}}\zeta^{j,\sigma,m}_{J}(0,R^{q+1}_J)\qquad,\\
\Hsnj&=(\turmd{q}{j}{\sigma}{n}{+},0)+\sum_{I\in\cIb^{q,\leq j}}\xi^{j,\sigma,n}_{I}(D^q_I,0)
\end{align*}
hold. These series converge uniformly in $\supp\eta$ together with all their derivatives even after
multiplication by arbitrary powers of $r$\,.
(Compare with \cite[p. 1033]{sphharm}, \cite[Theorem 1]{linelaz} and \paulystaticsphharmexpa.)
The constants $\xi^{j,\sigma,}_{}$ and $\zeta^{j,\sigma,}_{}$
coincide with those of Lemma \ref{loesEsmHgnlemma}, whenever they co-exist.
\end{cor}

\begin{proof}
We show the representation for some even $j$ and $(E,0):=\Esmj$\,.
The other representations may be proved in a similar way.

We have $\pdiv E=0$ and $(M\Lambda^\me)^{j+1}(E,0)=(0,0)$ in $\om$\,. Hence $M^{j+1}(E,0)$ vanishes in $\supp\eta$\,.
For $\fJ+N/2\leq s\notin\pI$ we see by Lemma \ref{loesEsmHgnlemma}
$$\tilde{E}:=E-\turmd{q}{j}{\sigma}{m}{+}-\sum_{I\in\cIb^{q,\leq j}_{s-j-1}}\xi^{j,\sigma,m}_{I} D^q_I\in\Lzq{s-j-1}(\supp\eta)\subset\Lzq{\Nh-1}(\supp\eta)$$
and
$$\pdiv\restr{\tilde{E}}{\supp\eta}=0\qqtext{,}M^{j+1}\restr{(\tilde{E},0)}{\supp\eta}=(0,0)\qquad.$$
Now the generalized spherical harmonics expansion \paulystaticsphharmexpa\, yields with unique constants
$\xi^{j,\sigma,m}_{}\in\cz$ the representation
$$\restr{\tilde{E}}{\supp\eta}=\sum_{I\in\cIb^{q,\leq j}\ohne\cIb^{q,\leq j}_{s-j-1}}\xi^{j,\sigma,m}_{I} D^q_I$$
and therefore the assertion follows immediately.
\end{proof}

Next we want to characterize $\regqJsom$ by orthogonality constraints using the growing Dirichlet-forms.
For this we need some special properties of the tower-forms. As in \cite{paulystatic} we introduce the
first order differential operator with compactly supported coefficients
$C:=C_{\Delta,\eta}:=\Delta\eta-\eta\Delta$\,, the commutator of $\Delta$ and the multiplication by $\eta$\,.

\begin{lem}\mylabel{Tuermeorthoreg}
Let $u$ and $v$ be regular tower-forms corresponding to some finite set of indices.
Then the cut-off function $\eta$ may be chosen, such that
$$\skp{Cu}{v}_{\Lzq{}}=0$$
holds, except for the special cases
$$\skpb{C\,\,\turmd{q}{k}{\sigma}{m}{\theta}}{\turmd{q}{\ell}{\gamma}{n}{\vartheta}}_{\Lzq{}}\neq0
\qquad\equi\qquad
\skpb{C\,\,\turmr{q}{k}{\sigma}{m}{\theta}}{\turmr{q}{\ell}{\gamma}{n}{\vartheta}}_{\Lzq{}}\neq0$$
$$\equi\qquad\sigma=\gamma\,\,,\,\,m=n\,\,,\,\,\theta\vartheta=-\,\,,\,\,(k,\ell)\in\big\{(0,2),(1,1),(2,0)\big\}$$
and
$$\skpb{C\,\,\turmd{q}{k}{\sigma}{m}{\theta}}{\turmr{q}{\ell}{\gamma}{n}{\vartheta}}_{\Lzq{}}\neq0$$
$$\equi\qquad\sigma=\gamma\,\,,\,\,m=n\,\,,\,\,\theta\vartheta=-\,\,,\,\,(k,\ell)\in\big\{(0,2),(2,0)\big\}\qquad.$$
\end{lem}

\begin{rem}\mylabel{Tuermeorthoregrem}
In the special cases we have
\begin{align*}
&\qquad\skpb{C\,\turmd{q}{k}{\sigma}{m}{-}}{\turmd{q}{\ell}{\sigma}{m}{+}}_{\Lzq{}}
=-\skpb{C\,\turmd{q}{\ell}{\sigma}{m}{+}}{\turmd{q}{k}{\sigma}{m}{-}}_{\Lzq{}}
=\skpb{C\,\turmr{q}{\ell}{\sigma}{m}{-}}{\turmr{q}{k}{\sigma}{m}{+}}_{\Lzq{}}\\
&=-\skpb{C\,\turmr{q}{k}{\sigma}{m}{+}}{\turmr{q}{\ell}{\sigma}{m}{-}}_{\Lzq{}}
=\begin{cases}-\frac{q+\sigma}{N+2\sigma}&,\,(k,\ell)=(0,2)\\1&,\,(k,\ell)=(1,1)\\-\frac{q'+\sigma}{N+2\sigma}&,\,(k,\ell)=(2,0)\\\end{cases}
\intertext{and}
&\qquad\skpb{C\,\turmd{q}{k}{\sigma}{m}{-}}{\turmr{q}{\ell}{\sigma}{m}{+}}_{\Lzq{}}
=\skpb{C\,\turmd{q}{k}{\sigma}{m}{+}}{\turmr{q}{\ell}{\sigma}{m}{-}}_{\Lzq{}}
=\skpb{C\,\turmr{q}{\ell}{\sigma}{m}{-}}{\turmd{q}{k}{\sigma}{m}{+}}_{\Lzq{}}\\
&=\skpb{C\,\turmr{q}{\ell}{\sigma}{m}{+}}{\turmd{q}{k}{\sigma}{m}{-}}_{\Lzq{}}
=\ie\frac{\omega^{q-1}_\sigma}{N+2\sigma}\begin{cases}-1&,\,(k,\ell)=(0,2)\\1&,\,(k,\ell)=(2,0)\\\end{cases}\qquad.
\end{align*}
\end{rem}

\begin{proof}
From \cite[(31)]{linelaz} with $a=b=1$ we have $\skp{Cu}{v}_{\Lzq{}}=-\skp{u}{Cv}_{\Lzq{}}$ for suitable q-forms $u,v$\,.
Using the spherical calculus presented in \cite{sphharm} we compute for tower-forms $u,v$
\begin{align}\begin{split}
&\quad\qquad\skp{Cu}{v}_{\Lzq{}}=\int_{\rzp}r^{N-1}\skp{Cu}{v}_{(r)}\,dr\\
&=\int_{\rzp}r^{N-1}\Big(\skpb{\rho C\rhoh\,\rho u(r)}{\rho v(r)}_{\qLz{q-1}{}(\SN)}+\skpb{\tau C\tauh\,\tau u(r)}{\tau v(r)}_{\Lzq{}(\SN)}\Big)\,dr\\
&=\int_{\rzp}r^{N-1}(\Gamma_{\hat{\eta}}r^{\alpha})r^{\beta}\Big(\skpb{\rho u(1)}{\rho v(1)}_{\qLz{q-1}{}(\SN)}+\skpb{\tau u(1)}{\tau v(1)}_{\Lzq{}(\SN)}\Big)\,dr\\
&=\skp{u}{v}_{(1)}\int_{\rzp}r^{N-1+\beta}\Gamma_{\hat{\eta}}r^{\alpha}\,dr
=(\alpha-\beta)\skp{u}{v}_{(1)}\int_{\rzp}r^{N-2+\alpha+\beta}\hat{\eta}'(r)\,dr\qquad.
\end{split}\mylabel{Cuv}\end{align}
\big(Here we define
$\skp{u}{v}_{(r)}:=\skpb{\rho u(r)}{\rho v(r)}_{\qLz{q-1}{}(\SN)}+\skpb{\tau u(r)}{\tau v(r)}_{\Lzq{}(\SN)}$
and also $\alpha:=\homd(u)$\,, $\beta:=\homd(v)$\,.
We note $\rho u(r)=r^\alpha\rho u(1)$ and $\tau u(r)=r^\beta\tau u(1)$\,.
Furthermore, we denote by $\Gamma_\phi$ for suitable $\phi,\varphi$
the first order ordinary differential operator
$\Gamma_\phi\varphi(t):=2\phi'(t)\varphi'(t)+\phi'(t)+(N-1)t^\me\phi'(t)$\,.\big)
Since the spherical eigen-forms $T^q_{\sigma,m}$ and $S^q_{\sigma,m}$ present an orthonormal system
in $\Lzq{}(\SN)$\,, the expression $\skp{u}{v}_{(1)}$ only may differ from zero in the cases
\begin{align*}
u&=\turmd{q}{k}{\sigma}{m}{\theta}&&,&v=\turmd{q}{\ell}{\sigma}{m}{\vartheta}&&,&&k-\ell\text{ even}&&,\\
u&=\turmr{q}{k}{\sigma}{m}{\theta}&&,&v=\turmr{q}{\ell}{\sigma}{m}{\vartheta}&&,&&k-\ell\text{ even}&&,\\
u&=\turmd{q}{k}{\sigma}{m}{\theta}&&,&v=\turmr{q}{\ell}{\sigma}{m}{\vartheta}&&,&&k,\ell\text{ even}&&,\\
u&=\turmr{q}{k}{\sigma}{m}{\theta}&&,&v=\turmd{q}{\ell}{\sigma}{m}{\vartheta}&&,&&k,\ell\text{ even}
\end{align*}
with $\theta,\vartheta\in\{+,-\}$\,.
We may assume additionally that our tower forms under consideration are at most of height $K$ and index $Z$\,.
According to \cite[Lemma 2 (i)]{complete} we may choose the cut-off function $\eta$
(resp. $\hat{\eta}$\,, $\mbox{\boldmath$\eta$}$), such that for given $\hat{j}\in\nzn$
\beq\int_\rz\hat{\eta}'(r)\,r^j\,dr=\delta_{0,j}\qqtext{,}-\hat{j}\leq j\leq\hat{j}\qtext{,}j\in\zz\qquad{,}\mylabel{etadachortho}\eeq
holds. Let us pick some $\hat{j}\geq N+2(1+K+Z)$\,. In the four cases above we have degrees of homogeneities
$\alpha=\homg{k}{\sigma}{\theta}$ and $\beta=\homg{\ell}{\sigma}{\vartheta}$\,.
Because of $N-2+\alpha+\beta\in[-\hat{j},\hat{j}]$ and \eqref{etadachortho}
the integral \eqref{Cuv} can only differ from zero, if $N-2+\alpha+\beta=0$\,.
But if $\theta\vartheta=+$\,, then either $N-2+\homg{k}{\sigma}{+}+\homg{\ell}{\sigma}{+}=N-2+k+\ell+2\sigma\neq0$
or $N-2+\homg{k}{\sigma}{-}+\homg{\ell}{\sigma}{-}=N-2+k+\ell-2\sigma-2N\neq0$\,,
since $k+\ell$ is even and $N$ odd. So only $\theta\vartheta=-$ is possible and we get
$$N-2+\homg{k}{\sigma}{-}+\homg{\ell}{\sigma}{+}=N-2+\homg{k}{\sigma}{+}+\homg{\ell}{\sigma}{-}=-2+k+\ell=0$$
$$\equi\qquad(k,\ell)\in\big\{(0,2),(1,1),(2,0)\big\}\qquad,$$
where the possibility of $k=\ell=1$ has to be excluded in the two last cases, where $k,\ell$ are even.
Thus we have proved the essential assertions of the lemma.

Let us finally calculate one of the special integrals as an example:
\begin{align*}
&\qquad\skpb{C\,\turmd{q}{k}{\sigma}{m}{-}}{\turmd{q}{\ell}{\sigma}{m}{+}}_{\Lzq{}}=\big(\homg{k}{\sigma}{-}-\homg{\ell}{\sigma}{+}\big)\skpb{\turmd{q}{k}{\sigma}{m}{-}}{\turmd{q}{\ell}{\sigma}{m}{+}}_{(1)}\\
&=(k-\ell-2\sigma-N)\\
&\qquad
\begin{cases}{}^-\alpha^{q,0}_\sigma{}^+\alpha^{q,1}_\sigma\big((\omega^{q-1}_\sigma)^2+(q'+\homg{0}{\sigma}{-})(q'+\homg{2}{\sigma}{+})\big)&,\,(k,\ell)=(0,2)\\
{}^-\alpha^{q+1,0}_\sigma{}^+\alpha^{q+1,0}_\sigma&,\,(k,\ell)=(1,1)\\
{}^-\alpha^{q,1}_\sigma{}^+\alpha^{q,0}_\sigma\big((\omega^{q-1}_\sigma)^2+(q'+\homg{2}{\sigma}{-})(q'+\homg{0}{\sigma}{+})\big)&,\,(k,\ell)=(2,0)
\end{cases}\\
&\\
&=\begin{cases}\frac{(-2-2\sigma-N)((\omega^{q-1}_\sigma)^2+(q'-\sigma-N)(q'+2+\sigma))}{2(2+2\sigma+N)(-2\sigma-N)}=-\frac{q+\sigma}{N+2\sigma}&,\,(k,\ell)=(0,2)\\
-\frac{-2\sigma-N}{2\sigma+N}=1&,\,(k,\ell)=(1,1)\\
\frac{(2-2\sigma-N)((\omega^{q-1}_\sigma)^2+(q'+2-\sigma-N)(q'+\sigma))}{2(2-2\sigma-N)(-2\sigma-N)}=-\frac{q'+\sigma}{N+2\sigma}&,\,(k,\ell)=(2,0)
\end{cases}
\end{align*}
\end{proof}

\begin{lem}\mylabel{Tuermeorthoaus}
In the same sense Lemma \ref{Tuermeorthoreg} holds for all tower-forms, if one pays attention to
$\turmd{0}{0}{0}{1}{-}=0$ and $\turmr{N}{0}{0}{1}{-}=0$\,.
Besides in the special cases we get for the exceptional tower-forms
\begin{align*}
&\qquad\skpb{C\,\turmd{0}{2}{0}{1}{-}}{\turmd{0}{0}{0}{1}{+}}_{\lz}
=-\skpb{C\,\turmd{0}{0}{0}{1}{+}}{\turmd{0}{2}{0}{1}{-}}_{\lz}
=\skpb{C\,\turmr{N}{2}{0}{1}{-}}{\turmr{N}{0}{0}{1}{+}}_{\qLz{N}{}}\\
&=-\skpb{C\,\turmr{N}{0}{0}{1}{+}}{\turmr{N}{2}{0}{1}{-}}_{\qLz{N}{}}
=-\skpb{C\,\turmr{1}{1}{0}{1}{-}}{\turmr{1}{1}{0}{1}{+}}_{\qLz{1}{}}
=\skpb{C\,\turmr{1}{1}{0}{1}{+}}{\turmr{1}{1}{0}{1}{-}}_{\qLz{1}{}}\\
&=-\skpb{C\,\turmd{N-1}{1}{0}{1}{-}}{\turmd{N-1}{1}{0}{1}{+}}_{\qLz{N-1}{}}
=\skpb{C\,\turmd{N-1}{1}{0}{1}{+}}{\turmd{N-1}{1}{0}{1}{-}}_{\qLz{N-1}{}}=1\qquad.
\end{align*}
\end{lem}

Summing up we obtain

\begin{rem}\mylabel{Tuermeorthobem}
For tower-forms $u,v$ the scalar product $\skp{Cu}{v}_{\Lzq{}}$ can only differ from zero, if
$u$ and $v$ possess different signs $\pm$ as well as equal eigenvalue and counting indices $\sigma$ and $m$\,.
Additionally in the cases $u=\turmd{q}{k}{\sigma}{m}{\pm}$\,, $v=\turmd{q}{\ell}{\sigma}{m}{\mp}$
or $u=\turmr{q}{k}{\sigma}{m}{\pm}$\,, $v=\turmr{q}{\ell}{\sigma}{m}{\mp}$
the heights $(k,\ell)$ must belong to $\big\{(0,2),(1,1),(2,0)\big\}$ as well as in the cases
$u=\turmd{q}{k}{\sigma}{m}{\pm}$\,, $v=\turmr{q}{\ell}{\sigma}{m}{\mp}$ or in reverse order
even to $\big\{(0,2),(2,0)\big\}$\,.
\end{rem}

Now let us return to our static solutions.
We put
$$\Lzqqpeoml:={}_{\Lambda^\me}\Lzqqpeom{}:={}_{\eps^\me}\Lzqom{}\times{}_{\mu^\me}\Lzqpeom{}$$
with scalar product $\skpoml{\,\cdot\,}{\,\cdot\,}=\skp{\Lambda^\me\,\cdot\,}{\,\cdot\,}_{\Lzqqpeom{}}$\,,
see \paulytimeharmsecthprob.

\begin{lem}\mylabel{Regorthorechnung}
Let $s\in(2-N/2,\infty)\ohne\pI$ and $\FG\in\regqnsom$ with representation
$$\loesn\FG=\EH+\sum_{I\in\cIb^{q,0}_{s-1}}{\tt e}_I\eta(D^q_I,0)+\sum_{J\in\cJb^{q+1,0}_{s-1}}{\tt h}_J\eta(0,R^{q+1}_J)\qquad,$$
where $\EH\in\big(\ronqom{s-1}\cap\eps^\me\dqom{s-1}\big)\times\big(\mu^\me\ronqpeom{s-1}\cap\dqpeom{s-1}\big)$
and ${\tt e}_I,{\tt h}_J\in\cz$\,.
Then for all $I=(-,0,\sigma,m)\in\cIb^{q,0}_{s-1}$ and $J=(-,0,\gamma,n)\in\cJb^{q+1,0}_{s-1}$
\begin{align*}
\text{\rm\bf (i)}&&\skpb{\FG}{\Esmn}_{\Lzqqpeoml{}}&=\skpb{\FG}{\Hgnn}_{\Lzqqpeoml{}}=0\qquad,\\
\text{\rm\bf (ii)}&&\skpb{\FG}{\Esme}_{\Lzqqpeoml{}}&={\tt e}_I\qquad,\\
\text{\rm\bf (iii)}&&\skpb{\FG}{\Hgne}_{\Lzqqpeoml{}}&={\tt h}_J\qquad.
\end{align*}
\end{lem}

\begin{rem}\mylabel{Regorthorechnungbem}
It is sufficient to choose $\hat{j}\geq2(s+1)$ in \eqref{etadachortho}.
\end{rem}

\begin{proof}
We set $(0,\Htsm):=\Lambda^\me\Esme=\loesn\Esmn$ and $(\Etgn,0):=\Lambda^\me\Hgne=\loesn\Hgnn$\,.
Let us look at $\Esm$ and $\Htsm$\,.
For $(-,0,\sigma,m)\in\cIb^{q,0}_{s-1}$ we have $s>\sigma+1+N/2$ and thus only weights $s$ larger than
$1+N/2$ have to be considered. According to Lemma \ref{loesEsmHgnlemma}
\beq\Esm\in\dH{q}{-s+1}{\eps}(\Omega)\subset\ronqnom{-s+1}\qqtext{,}
\Htsm\in\dqpeom{-s}\qquad.\mylabel{intbarkeitEsmHsm}\eeq
Therefore, all scalar products under consideration are well defined.
By \paulytimeharmpartint\, and Lemma \ref{EsmHgnlemma} we get
$$\skpb{M\EH}{\Esmn}_{\Lzqqpeoml{}}=\skp{\pdiv H}{\Esm}_{\Lzqom{}}=0$$
and thus with $(0,R^{q+1}_J)=M(D^q_{{}_1J},0)$
\begin{align*}
&\qquad\skpb{\FG}{\Esmn}_{\Lzqqpeoml{}}=\skpb{M\loesn\FG}{(\Esm,0)}_{\Lzqqpeom{}}\\
&=\sum_{I\in\cIb^{q,0}_{s-1}}{\tt e}_I\ub{\skpb{M\eta(D^q_I,0)}{(\Esm,0)}_{\Lzqqpeom{}}}_{=\skp{(0,\rot\eta D^q_I)}{(\Esm,0)}_{\Lzqqpeom{}}=0}\\
&\qquad+\sum_{J\in\cJb^{q+1,0}_{s-1}}{\tt h}_J\skpb{M^2\eta(D^q_{{}_1J},0)}{(\Esm,0)}_{\Lzqqpeom{}}\\
&\qquad-\sum_{J\in\cJb^{q+1,0}_{s-1}}{\tt h}_J\ub{\skpb{M C_{M,\eta}(D^q_{{}_1J},0)}{(\Esm,0)}_{\Lzqqpeom{}}}_{=\skp{\pdiv C_{\rot,\eta}D^q_{{}_1J}}{\Esm}_{\Lzqom{}}=0}\qquad.
\end{align*}
For $J=(-,0,\gamma,n)\in\cJb^{q+1,0}_{s-1}$ we have
$\pdiv\eta D^q_{{}_1J}=\pdiv\eta\turmd{q}{1}{\gamma}{n}{-}=0$
by \paulystaticremodd\, and therefore $\pdiv\rot\eta D^q_{{}_1J}=\Delta\eta D^q_{{}_1J}=C D^q_{{}_1J}$\,.
This shows
$$\skpb{\FG}{\Esmn}_{\Lzqqpeoml{}}=\sum_{J\in\cJb^{q+1,0}_{s-1}}{\tt h}_J\skp{CD^q_{{}_1J}}{\Esm}_{\Lzqom{}}$$
recalling $M^2\eh=(\pdiv\rot e,\rot\pdiv h)$\,.

According to \paulystaticitloesremcor\, and under the present assumptions $\loesn\loes\FG$ is also well defined
and has the representation
$$\loesn\loes\FG=(\tilde{E},\tilde{H})+\sum_{I\in\cIb^{q,\leq1}_{s-2}}\tilde{{\tt e}}_I\eta(D^q_I,0)+\sum_{J\in\cJb^{q+1,\leq1}_{s-2}}\tilde{{\tt h}}_J\eta(0,R^{q+1}_J)$$
with $(\tilde{E},\tilde{H})\in\big(\ronqom{s-2}\cap\eps^\me\dqom{s-2}\big)\times\big(\mu^\me\ronqpeom{s-2}\cap\dqpeom{s-2}\big)$
and $\tilde{{\tt e}}_I,\tilde{{\tt h}}_J\in\cz$\,.
Moreover, $\tilde{{\tt h}}_{{}_1I}={\tt e}_I$ and $\tilde{{\tt e}}_{{}_1J}={\tt h}_J$ hold
for $I\in\cIb^{q,0}_{s-1}$ and $J\in\cJb^{q+1,0}_{s-1}$\,. From
$$\Lambda^\me M(\tilde{E},\tilde{H})\in\big(\ronqom{s-1}\cap\eps^\me\dqnom{s-1}\big)\times\big(\mu^\me\ronqpenom{s-1}\cap\dqpeom{s-1}\big)$$
as well as by \eqref{intbarkeitEsmHsm} and \paulytimeharmpartint\, we get
$$\skpb{M\Lambda^\me M(\tilde{E},\tilde{H})}{\loesn\Lambda(\Esm,0)}_{\Lzqqpeom{}}
=-\skpb{M(\tilde{E},\tilde{H})}{(\Esm,0)}_{\Lzqqpeom{}}=0\qquad.$$
Using this and $\FG=M\Lambda^\me M\loesn\loes\FG$ we derive
\begin{align*}
&\qquad\skpb{\FG}{\Esme}_{\Lzqqpeoml{}}\\
&=\sum_{I\in\cIb^{q,\leq1}_{s-2}}\tilde{{\tt e}}_I\ub{\skpb{M^2\eta(D^q_I,0)}{(0,\Htsm)}_{\Lzqqpeom{}}}_{=0}\\
&\qquad+\sum_{J\in\cJb^{q+1,\leq1}_{s-2}}\tilde{{\tt h}}_J\skpb{M^2\eta(0,R^{q+1}_J)}{(0,\Htsm)}_{\Lzqqpeom{}}\\
&=\sum_{J\in\cJb^{q+1,0}_{s-2}}\tilde{{\tt h}}_J\skpb{M^2\eta(0,R^{q+1}_J)}{(0,\Htsm)}_{\Lzqqpeom{}}\\
&\qquad+\sum_{I\in\cIb^{q,0}_{s-1}}\tilde{{\tt h}}_{{}_1I}\skpb{M^2\eta(0,R^{q+1}_{{}_1I})}{(0,\Htsm)}_{\Lzqqpeom{}}
\end{align*}
because $\cJb^{q+1,\leq1}_{s-2}=\cJb^{q+1,0}_{s-2}\,\dot{\cup}\,\cJb^{q+1,1}_{s-2}=\cJb^{q+1,0}_{s-2}\,\dot{\cup}\,{}_1(\cIb^{q,0}_{s-1})$\,.
Applying once more \paulystaticremodd\, we obtain
\begin{align*}
M^2\eta(0,R^{q+1}_{{}_1I})&=\Delta\eta(0,R^{q+1}_{{}_1I})=C(0,R^{q+1}_{{}_1I})\qquad,\\
M^2\eta(0,R^{q+1}_J)&=M^2\eta M(D^q_{{}_1J},0)=MM^2\eta(D^q_{{}_1J},0)-M^2C_{M,\eta}(D^q_{{}_1J},0)\\
&=MC(D^q_{{}_1J},0)-M^2C_{M,\eta}(D^q_{{}_1J},0)\qquad.
\end{align*}
Partial integration yields
$$\skpb{MC(D^q_{{}_1J},0)}{(0,\Htsm)}_{\Lzqqpeom{}}=-\skpb{C(D^q_{{}_1J},0)}{(\Esm,0)}_{\Lzqqpeom{}}$$
and clearly all terms of the sum like
$$\skpb{M^2C_{M,\eta}(D^q_{{}_1J},0)}{(0,\Htsm)}_{\Lzqqpeom{}}$$
vanish by (two times) partial integration.
Finally we get for $(-,0,\sigma,m)\in\cIb^{q,0}_{s-1}$
\begin{align*}
\skpb{\FG}{\Esmn}_{\Lzqqpeoml{}}
&=\sum_{J\in\cJb^{q+1,0}_{s-1}}{\tt h}_J\skp{C D^q_{{}_1J}}{\Esm}_{\Lzqom{}}\qquad,\\
\skpb{\FG}{\Esme}_{\Lzqqpeoml{}}
&=-\sum_{J\in\cJb^{q+1,0}_{s-2}}\tilde{{\tt h}}_J\skp{C D^q_{{}_1J}}{\Esm}_{\Lzqom{}}\\
&\qquad+\sum_{I\in\cIb^{q,0}_{s-1}}{\tt e}_I\skp{CR^{q+1}_{{}_1I}}{\Htsm}_{\Lzqpeom{}}\qquad.
\intertext{Analogously for $(-,0,\gamma,n)\in\cJb^{q+1,0}_{s-1}$ one sees}
\skpb{\FG}{\Hgnn}_{\Lzqqpeoml{}}
&=\sum_{I\in\cIb^{q,0}_{s-1}}{\tt e}_I\skp{C R^{q+1}_{{}_1I}}{\Hgn}_{\Lzqpeom{}}\qquad,\\
\skpb{\FG}{\Hgne}_{\Lzqqpeoml{}}
&=-\sum_{I\in\cIb^{q,0}_{s-2}}\tilde{{\tt e}}_I\skp{C R^{q+1}_{{}_1I}}{\Hgn}_{\Lzqpeom{}}\\
&\qquad+\sum_{J\in\cJb^{q+1,0}_{s-1}}{\tt h}_J\skp{CD^q_{{}_1J}}{\Etgn}_{\Lzqom{}}\qquad.
\end{align*}
Now all integrals on the right hand sides only extend over $\supp\nabla\eta$\,.
Thus we may insert the expansions from Corollary \ref{loesEsmHgnkor} for
$$\Esm\,,\,\Htsm\,,\,\Hgn\,,\,\Etgn\qquad.$$
Using the orthogonality properties from Lemma \ref{Tuermeorthoreg}
and Remark \ref{Tuermeorthoregrem} we finally obtain
$$\skp{CD^q_{{}_1J}}{\Esm}_{\Lzqom{}}=\skp{CR^{q+1}_{{}_1I}}{\Hgn}_{\Lzqpeom{}}=0$$
and
$$\skp{CR^{q+1}_{{}_1I}}{\Htsm}_{\Lzqpeom{}}=\delta_{I,(-,0,\sigma,m)}\qtext{,}
\skp{CD^q_{{}_1J}}{\Etgn}_{\Lzqom{}}=\delta_{J,(-,0,\gamma,n)}\quad.$$
We note that by \paulystaticinttower\, we only have to consider tower-forms with maximal heights
$K=1$ and maximal eigenvalue index $Z\leq s-1-N/2$\,. 
Thus the choice $\hat{j}\geq2(s+1)\geq N+2(1+K+Z)$ is sufficient according to \eqref{etadachortho}.
\end{proof}

Now we are ready to characterize the spaces of regular convergence by orthogonality constraints.

\begin{lem}\mylabel{Regortho}
Let $\fJ\in\nz$ and $s\in(\fJ+1-N/2,\infty)\ohne\pI$ as well as $\FG\in\regqnsom$\,.
Then $\FG$ belongs to $\regqJsom$\,, if and only if
$$\skpboml{\FG}{\Esmke}=\skpboml{\FG}{\Hgnle}=0$$
holds for all $(k,\sigma,m)\in\Theta^{q,\fJ}_s$ and $(\ell,\gamma,n)\in\Theta^{q+1,\fJ}_s$\,, where
$$\Theta^{q,\fJ}_s:=\setb{(k,\sigma,m)\in\nz_0^3}{k\leq \fJ-1\,\wedge\,\sigma<s-N/2-k-1\,\wedge\,1\leq m\leq\mu^q_\sigma}\quad.$$
Moreover, $\regqJsom$ is a closed subspace of $\regqnsom$ and $\Lzqqpesom$\,.
\end{lem}

\begin{rem}\mylabel{Regorthobem}
We have the characterizations
\begin{align*}
\Theta^{q,\fJ}_s&=\setb{(k,\sigma,m)\in\{0,\dots,\fJ-1\}\times\nzn\times\nz}{\Esmke\in\Lzqqpeom{-s}}\quad,\\
\Theta^{q+1,\fJ}_s&=\setb{(\ell,\gamma,n)\in\{0,\dots,\fJ-1\}\times\nzn\times\nz}{\Hgnle\in\Lzqqpeom{-s}}\quad.
\end{align*}
\end{rem}

\begin{proof}
The assertions of the remark follow by Lemma \ref{loesEsmHgnlemma}.
The proof of the lemma is a straightforward induction over $\fJ$\,.
The induction start is given by Lemma \ref{Regorthorechnung} since
$\Theta^{q,1}_s=\setb{(0,\sigma,m)}{(-,0,\sigma,m)\in\cIb^{q,0}_{s-1}}$ and
$\Theta^{q+1,1}_s=\setb{(0,\gamma,n)}{(-,0,\gamma,n)\in\cJb^{q+1,0}_{s-1}}$\,.
For the step we note by definition
\begin{align*}
\regqom{\fJ+1}{s}&=\setb{\FG\in\regqJsom}{\loes^{\fJ+1}\FG\in\regqnom{s-\fJ-1}}\\
&=\setb{\FG\in\regqJsom}{\loes^\fJ\FG\in\regqom{1}{s-\fJ}}
\end{align*}
and according to the induction start we obtain $\FG\in\regqom{\fJ+1}{s}$\,,
if and only if $\FG\in\regqJsom$ and
$$\skpboml{\loes^\fJ\FG}{\Esme}=\skpboml{\loes^\fJ\FG}{\Hgne}=0$$
holds for all $(0,\sigma,m)\in\Theta^{q,1}_{s-\fJ}$ and $(0,\gamma,n)\in\Theta^{q+1,1}_{s-\fJ}$\,.
Since $\loes^\fJ\FG\in\regqom{1}{s-\fJ}$ we get
$$\loesn\loes^{\fJ-1}\FG\in\ronqom{s-\fJ}\times\dqpeom{s-\fJ}$$
and with Lemma \ref{loesEsmHgnlemma}
$$\Lambda^\me\loes^2\Lambda\Esmn\in\ronqom{-s-1+\fJ}\times\{0\}\qquad,$$
because $(0,\sigma,m)\in\Theta^{q,1}_{s-\fJ}$ implies $\sigma<s-1-\fJ-N/2$\,, i.e. $s+1-\fJ>\sigma+2+N/2$\,.
Using partial integration, i.e. \paulytimeharmpartint, we compute
\begin{align*}
&\qquad\skpboml{\loes^\fJ\FG}{\Esme}\\
&=\skpb{\loesn\loes^{\fJ-1}\FG}{M\Lambda^\me\loes^2\Esmn}_{\Lzqqpeom{}}
=-\skpboml{\loes^{\fJ-1}\FG}{\Esmz}
\end{align*}
and therefore repeating this argument
\begin{align*}
\skpboml{\loes^\fJ\FG}{\Esme}&=(-1)^\fJ\skpboml{\FG}{\EsmJe}\qquad.
\intertext{Analogously we conclude}
\skpboml{\loes^\fJ\FG}{\Hgne}&=(-1)^\fJ\skpboml{\FG}{\HgnJe}\qquad.
\end{align*}
Finally we obtain $\FG\in\regqom{\fJ+1}{s}$\,, if and only if $\FG\in\regqJsom$ and
$$\skpboml{\FG}{\EsmJe}=\skpboml{\FG}{\HgnJe}=0$$
holds for all $(0,\sigma,m)\in\Theta^{q,1}_{s-\fJ}$ and $(0,\gamma,n)\in\Theta^{q+1,1}_{s-\fJ}$
and the induction hypothesis for $\regqJsom$ completes the proof.
\end{proof}

We are looking for projectors onto $\regqJsom$ and thus for a dual basis of
$$\Esmk\qqtext{,}\Hsmk\qquad.$$
For $\ell,\sigma\in\nzn$ and $(m,n)\in\{1,\dots,\mu^q_\sigma\}\times\{1,\dots,\mu^{q+1}_\sigma\}$
let us define
\begin{align*}
\epmsn&:=\eta\turmd{q}{1}{\sigma}{n}{\pm}&&,&\hpmsm&:=\eta\turmr{q+1}{1}{\sigma}{m}{\pm}&&,\\
\epmsnl&:=M^\ell(\epmsn,0)&&,&\hpmsml&:=M^\ell(0,\hpmsm)&&.
\end{align*}

\begin{lem}\mylabel{Projlemmae}
Let $\ell,k\in\nzn$\,. Then
$\epmsnl$ and $\hpmsml$ are $\cu$-forms on $\rN$ and belong to $\regqnom{<\mp(\Nh+\sigma)-1+\ell}$\,.
Furthermore, $\epmsnlz$ and $\hpmsmlz$ are compactly supported and thus elements of $\regqnom{\vox}$
as well as $\regqom{\ell}{s}$ for $s\in(\ell-N/2,\infty)\ohne\pI$\,.
Moreover, for $\ell\geq2$
$$\loes^k\epmsnkl=\epmsnl\qqtext{,}\loes^k\hpmsmkl=\hpmsml$$
hold and these equations even remain valid for the negative forms $\emsnl$ and $\hmsml$ if $\ell=0,1$\,.
\end{lem}

\begin{proof}
According to \paulystaticremodd\, we have $\pdiv\epmsn=0$ and $\rot\hpmsm=0$ and hence by \paulystaticinttower\,
$\epmsnn\,,\,\hpmsmn\in\regqnom{<\mp(\Nh+\sigma)-1}$\,. Furthermore,
\begin{align}
\epmsne&=\eta(0,\turmr{q+1}{0}{\sigma}{n}{\pm})
+C_{M,\eta}(\turmd{q}{1}{\sigma}{n}{\pm},0)\quad,\mylabel{Mesn}\\
\epmsnz&=C(\turmd{q}{1}{\sigma}{n}{\pm},0)
=C_{M,\eta}(0,\turmr{q+1}{0}{\sigma}{n}{\pm})
+MC_{M,\eta}(\turmd{q}{1}{\sigma}{n}{\pm},0)\quad,\mylabel{Mquadesn}\\
\hpmsme&=\eta(\turmd{q}{0}{\sigma}{m}{\pm},0)
+C_{M,\eta}(0,\turmr{q+1}{1}{\sigma}{m}{\pm})\quad,\mylabel{Mhsm}\\
\hpmsmz&=C(0,\turmr{q+1}{1}{\sigma}{m}{\pm})
=C_{M,\eta}(\turmd{q}{0}{\sigma}{m}{\pm},0)
+MC_{M,\eta}(0,\turmr{q+1}{1}{\sigma}{m}{\pm})\quad.\mylabel{Mquadhsm}
\end{align}
Thus $\epmsne\,,\,\hpmsme\in\regqnom{<\mp(\Nh+\sigma)}$ and for all $\ell\in\nzn$
$$\supp\epmsnlz\cup\supp\hpmsmlz\subset\supp\nabla\eta\qquad,$$
i.e. $\epmsnlz\,,\,\hpmsmlz\in\regqnom{\vox}$\,.
By \paulystaticitloes\, resp. \paulystaticitloescor\, any power of
$\loes$ is well defined on $\epmsnlz\,,\,\hpmsmlz$\,.
Because of the compact supports we obtain for $\ell\geq2$
\beq\loes\epmsnle=\epmsnl\qqtext{,}\loes\hpmsmle=\hpmsml\mylabel{loesMellgz}\eeq
and a short induction shows
\beq\loes^k\epmsnkl=\epmsnl\qqtext{,}\loes^k\hpmsmkl=\hpmsml\mylabel{loesMkell}\eeq
for all $k\in\nzn$\,.
The forms $\emsnn$\,, $\hmsmn$ and by \eqref{Mesn}, \eqref{Mhsm}
also $\emsne$\,, $\hmsme$ possess the `right shape', such that for the negative forms
according to \paulystaticitloescor\, the equations \eqref{loesMellgz} and \eqref{loesMkell}
hold true for $\ell=0,1$ as well.
Once more taking into account the compact supports of $\epmsnz$ and $\hpmsmz$ we get
$$\epmsnlz\,,\,\hpmsmlz\in\regqom{\ell}{s}$$
for all $\ell\in\nzn$ and $s\in(\ell-N/2,\infty)\ohne\pI$\,.
\end{proof}

\begin{lem}\mylabel{Projlemmaz}
Let $K,Z\in\nzn$\,. Then for all $\sigma\in\{0,\dots,Z\}$ and $k\in\{-1,\dots,K\}$
as well as all $(\ell,\gamma)\in\nz_0^2$ and appropriate $m,n$
\begin{align*}
\skpbomol{\emgnlz}{\Esmke}&=\skpbomol{\hmgnlz}{\Hsmke}=0\qquad,\\
\skpbomol{\emgnlz}{\Hsmke}&=\skpbomol{\hmgnlz}{\Esmke}=
(-1)^\ell\delta_{k,\ell}\delta_{\sigma,\gamma}\delta_{m,n}\qquad.
\end{align*}
\end{lem}

\begin{rem}\mylabel{Projlemmazbem}
It suffices to choose $\hat{j}\geq N+2(2+K+Z)$ in \eqref{etadachortho}.
\end{rem}

\begin{proof}
We note again that $\emgnz$ and $\hmgnz$ have compact supports.
For all $\ell\in\nzn$ partial integration and \eqref{Mquadesn} yield
\begin{align*}
\text{\rm S}^{k,\ell}_{\sigma,\gamma}:=&\,\skpbomol{\emgnlz}{\Esmke}\\
=&\,\skpbomol{M^\ell\emgnz}{\Esmke}\\
=&\,(-1)^\ell\skpbomol{C(\turmd{q}{1}{\gamma}{n}{-},0)}{ M^\ell\Esmke}\qquad.
\end{align*}
Since $ M\loes=M\loesn=\id$ on $\supp\eta$ and $M(\Esm,0)=(0,0)$ the scalar products
$\text{\rm S}^{k,\ell}_{\sigma,\gamma}$ vanish for $\ell\geq k+2$\,. However, for $\ell\leq k+1$ we get
$$\text{\rm S}^{k,\ell}_{\sigma,\gamma}
=(-1)^\ell\skpbomol{C(\turmd{q}{1}{\gamma}{n}{-},0)}{\Esmkeml}$$
and these scalar products can only differ from zero if $k+1-\ell$ is even. The integrals range only over
$\supp\nabla\eta$\,. Thus we may insert the representations from Corollary \ref{loesEsmHgnkor}
for $\Esmkeml$ and see that $\text{\rm S}^{k,\ell}_{\sigma,\gamma}=0$ holds by
Lemma \ref{Tuermeorthoreg} even in these cases. The same arguments force
$$\tilde{\text{\rm S}}^{k,\ell}_{\sigma,\gamma}:=\skpbomol{ \emgnlz}{\Hsmke}$$
to vanish for $\ell\geq k+2$\,. If $\ell\leq k+1$ we get
$$\tilde{\text{\rm S}}^{k,\ell}_{\sigma,\gamma}
=(-1)^\ell\skpbomol{C(\turmd{q}{1}{\gamma}{n}{-},0)}{\Hsmkeml}\qquad.$$
These scalar products can only differ from zero, if $k+1-\ell$ is odd.
Again we insert the representations from Corollary \ref{loesEsmHgnkor}
for $\Hsmkeml$\,.
But now in the case $k=\ell$ we get a term  $\turmd{q}{1}{\sigma}{m}{+}$\,, whose scalar product
with $C\,\turmd{q}{1}{\gamma}{n}{-}$ does not vanish if $(\sigma,m)=(\gamma,n)$
according to Lemma \ref{Tuermeorthoreg}. Therefore we obtain
$$\tilde{\text{\rm S}}^{k,\ell}_{\sigma,\gamma}
=(-1)^\ell\skpbomol{C(\turmd{q}{1}{\gamma}{n}{-},0)}{(\turmd{q}{k+1-\ell}{\sigma}{m}{+},0)}
=(-1)^\ell\delta_{k,\ell}\delta_{\sigma,\gamma}\delta_{m,n}\quad.$$
Similarly the assertions about the remaining two scalar products may be shown.
\end{proof}

We have found our projections.

\begin{theo}\mylabel{regsatz}
Let $\fJ\in\nz$ and $s\in(\fJ+1-N/2,\infty)\ohne\pI$\,. Then
$$\regqnsom=\regqJsom\dotplus\Upsilon^{q,\fJ}_s\qquad,$$
where $\Upsilon^{q,\fJ}_s:=\Lin\setb{\emsmkz\,,\,\hmgnlz}
{(k,\sigma,m)\in\Theta^{q+1,\fJ}_s\,,\,(\ell,\gamma,n)\in\Theta^{q,\fJ}_s}$\,.\\
\ul{More} p\ul{recisel}y\ul{:} Each $\FG\in\regqnsom$ can be decomposed uniquely as
$$\FG=(F_{\reg},G_{\reg})+(F_\Upsilon,G_\Upsilon)\qquad,$$
where $(F_{\reg},G_{\reg})\in\regqJsom$ and $(F_\Upsilon,G_\Upsilon)\in\Upsilon^{q,\fJ}_s$ are defined by
\begin{align*}
(F_\Upsilon,G_\Upsilon)&:=\sum_{(k,\sigma,m)\in\Theta^{q,\fJ}_s}(-1)^k\skpboml{\FG}{\Esmke}\hmsmkz\\
&\qquad+\sum_{(k,\sigma,m)\in\Theta^{q+1,\fJ}_s}(-1)^k\skpboml{\FG}{\Hsmke}\emsmkz\qquad.
\end{align*}
\end{theo}

\begin{rem}\mylabel{regsatzbem}
$\Upsilon^{q,\fJ}_s$ are finite dimensional subspaces of
$\big(\cqunom\times\cqpeunom\big)\cap\regqnom{\vox}$ and the corresponding projections
$\FG\mapsto(F_\Upsilon,G_\Upsilon)$ resp. $\FG\mapsto(F_{\reg},G_{\reg})$ are continuous.
Moreover, the choice $\hat{j}\geq2(s+\fJ+1)$ in \eqref{etadachortho} is sufficient.
\end{rem}

\begin{proof}
According to Lemma \ref{Projlemmae} we have $\Upsilon^{q,\fJ}_s\subset\regqnom{\vox}$\,.
Thus $\FG\in\regqnsom$ implies $(F_{\reg},G_{\reg})\,,\,(F_\Upsilon,G_\Upsilon)\in\regqnsom$\,.
Applying Lemma \ref{Projlemmaz} we obtain for all $(k,\sigma,m)\in\Theta^{q,\fJ}_s$ and
$(\ell,\gamma,n)\in\Theta^{q+1,\fJ}_s$
$$\skpboml{(F_{\reg},G_{\reg})}{\Esmke}=\skpboml{(F_{\reg},G_{\reg})}{\Hgnle}=0$$
and therefore $(F_{\reg},G_{\reg})\in\regqJsom$ by Lemma \ref{Regortho}, which yields
$$\regqnsom\subset\regqJsom+\Upsilon^{q,\fJ}_s\subset\regqnsom\qquad.$$
So it remains to show the directness of the sum.
Let us pick an element
$$\FG=\sum_{(k,\sigma,m)\in\Theta^{q+1,\fJ}_s}{\tt f}_{k,\sigma,m}\emsmkz
+\sum_{(k,\sigma,m)\in\Theta^{q,\fJ}_s}{\tt g}_{k,\sigma,m}\hmsmkz$$
of the intersection $\regqJsom\cap\Upsilon^{q,\fJ}_s$\,. Applying $\loes$ yields that
$$\loes\FG=\sum_{(k,\sigma,m)\in\Theta^{q+1,\fJ}_s}{\tt f}_{k,\sigma,m}\emsmke
+\sum_{(k,\sigma,m)\in\Theta^{q,\fJ}_s}{\tt g}_{k,\sigma,m}\hmsmke$$
belongs to $\regqom{\fJ-1}{s-1}\subset\Lzqqpeom{s-1}$ by Lemma \ref{Projlemmae}.
If $k>0$ the forms $\emsmke$ resp. $\hmsmke$ have compact supports.
But for $k=0$ with \eqref{Mesn}, \eqref{Mhsm} the forms $\emsme$\,, $\hmsme$ are no longer compactly supported.
However, they belong to $\Lzqqpeom{<\Nh+\sigma}$ but even not to $\Lzqqpeom{\Nh+\sigma}$\,.
Thus $\emsme$\,, $\hmsme$ are not elements of $\Lzqqpeom{s-1}$ since
$(0,\sigma,m)\in\Theta^{q+1,\fJ}_s$ resp. $(0,\sigma,m)\in\Theta^{q,\fJ}_s$
implies $N/2+\sigma<s-1$\,. The forms $\emsme$ and $\hmsme$ are linear independent.
Consequently the coefficients ${\tt f}_{0,\sigma,m}$\,, ${\tt g}_{0,\sigma,m}$
have to vanish. Repeating this argument with $\loes^j\FG$ for $j=2,\dots,\fJ$ finally shows
${\tt f}_{k,\sigma,m}={\tt g}_{k,\sigma,m}=0$
for all $(k,\sigma,m)\in\Theta^{q+1,\fJ}_s$ and $(k,\sigma,m)\in\Theta^{q,\fJ}_s$\,.
\end{proof}

We are ready to approach the low frequency asymptotics.

\section{Low frequency asymptotics}

We will prove the desired asymptotic expansion in four steps, which are:\mylabel{asymsec}
\begin{description}
\item[\sf step one:] proof in the reduced case, i.e.:\\
compactly supported perturbations $\Lambdad$\,;\\
right hand sides from $\regqnsom$\,;\\
estimates in local norms
\item[\sf step two:] replacing $\regqnsom$ by $\Lzqqpesom$
\item[\sf step three:] replacing local norms by weighted norms
\item[\sf step four:] replacing compactly supported perturbations $\epsd$\,, $\mud$
by asymptotically vanishing perturbations
\end{description}
Following this program we only drop the assumption of compactly supported perturbations of the medium
in the last step. Thus \eqref{comper} may be assumed during the first three steps.

\subsection{First step}

\begin{lem}\mylabel{asyme}
Let $\fJ\in\nzn$\,, $s\in(\fJ+1/2,\fJ+N/2)\ohne\pI$ and $t:=s-\fJ-(N+1)/2$\,. Then
$$\Big|\normabst\Big|\loes_{\omega,\fJ-1}\FG
-\sum_{(k,\sigma,m)\in\Theta^{q,\fJ}_s}(\ie\omega)^k\skpboml{\FG}{\Esmke}\loes_{\omega,\fJ-1-k}\hmsmz$$
$$-\sum_{(k,\sigma,m)\in\Theta^{q+1,\fJ}_s}(\ie\omega)^k\skpboml{\FG}{\Hsmke}\loes_{\omega,\fJ-1-k}\emsmz\Big|\normabst\Big|_{\Lzqqpetom}$$
$$=\calO\big(|\omega|^\fJ\big)\normb{\FG}_{\Lzqqpesom}$$
holds uniformly in $\omega\in\czpomd\ohne\{0\}$ and $\FG\in\regqnsom$\,.
\end{lem}

\begin{proof}
According to Theorem \ref{regsatz} we decompose $\FG\in\regqnsom$
$$\FG=(F_{\reg},G_{\reg})+(F_\Upsilon,G_\Upsilon)\in\regqJsom\dotplus\Upsilon^{q,\fJ}_s$$
and obtain by Theorem \ref{Regsatz} uniformly in $\omega$ and $(F_{\reg},G_{\reg})$
$$\normb{\loes_{\omega,\fJ-1}(F_{\reg},G_{\reg})}_{\Lzqqpetom}=\calO\big(|\omega|^\fJ\big)\normb{(F_{\reg},G_{\reg})}_{\Lzqqpesom}\qquad.$$
By Remark \ref{regsatzbem} the projections are continuous and thus
$$\normb{\loes_{\omega,\fJ-1}\FG-\loes_{\omega,\fJ-1}(F_\Upsilon,G_\Upsilon)}_{\Lzqqpetom}=\calO\big(|\omega|^\fJ\big)\normb{\FG}_{\Lzqqpesom}\qquad.$$
This shows that we only have to determine the asymptotics of the special forms
$\emsmkz$\,, $\hmsmkz$ for $k\leq\fJ-1$\,, which belong to $\regqom{k}{s}$ using Lemma \ref{Projlemmae}.
Theorem \ref{Regsatz} and Lemma \ref{Projlemmae} yield
\begin{align*}
\loes_{\omega,k-1}\emsmkz&=(-\ie\omega)^k\loesom\loes^k\emsmkz=(-\ie\omega)^k\loesom\emsmz\qquad,\\
\loes_{\omega,k-1}\hmsmkz&=(-\ie\omega)^k\loesom\loes^k\hmsmkz=(-\ie\omega)^k\loesom\hmsmz\qquad.
\intertext{Then for $1\leq k\leq\fJ-1$ we obtain}
\loes_{\omega,\fJ-1}\emsmkz&=\loes_{\omega,k-1}\emsmkz
-\sum_{j=k}^{\fJ-1}(-\ie\omega)^j\loesn\loes^j\emsmkz\\
&=(-\ie\omega)^k\loes_{\omega,\fJ-1-k}\emsmz
\intertext{since $\loes^j\emsmkz=\loes^{j-k}\emsmz$\,. Analogously we compute}
\loes_{\omega,\fJ-1}\hmsmkz&=(-\ie\omega)^k\loes_{\omega,\fJ-1-k}\hmsmz\qquad.
\end{align*}
\end{proof}

According to the latter lemma we only have to calculate the asymptotics of the special forms
$$\loes_{\omega,\fJ-1-k}\emsmz\qqtext{,}\loes_{\omega,\fJ-1-k}\hmsmz$$
for $\omega\in\czpomd\ohne\{0\}$\,, $0\leq k\leq\fJ-1$ and $\sigma<s-N/2-1$\,.

For this we will use a technique introduced by Weck and Witsch in \cite{asyanal},
\cite{exacttwo} and \cite{exacthigher}, which was completed in \cite{complete} resp. \cite{linelaz}.
The idea is to compare
$$\loesom\emsmz\qqtext{,}\loesom\hmsmz$$
with special radiating solutions of the homogeneous problem in $\rNon$ and then to identify the proper
static terms in their asymptotic expansions. For this procedure it is essential that the perturbation
$\Lambdad$ has got a compact support.

Let us define for $q\in\{0,\dots,N-1\}$\,, $\sigma\in\nzn$\,, $m=1,\dots$ as well as
$\omega\in\cz_+\ohne\{0\}$ and $\nu_\sigma:=N/2+\sigma$
\begin{align}
\fE^{1,\omega}_{\sigma,m}&:=\sum_{k=0}^{\infty}(-\ie\omega)^{2k}\;\turmd{q}{2k+1}{\sigma}{m}{-}
+\kappa^{q+1}_\sigma\,\omega^{2\nu_\sigma}\sum_{k=0}^{\infty}(-\ie\omega)^{2k}\;\turmd{q}{2k+1}{\sigma}{m}{+}
\quad,\mylabel{EfettEins}\\
\begin{split}
\fH^{1,\omega}_{\sigma,m}&:=\frac{\ie}{\omega}\rot\fE^{1,\omega}_{\sigma,m}\\
&\,\,=\sum_{k=0}^{\infty}(-\ie\omega)^{2k-1}\;\turmr{q+1}{2k}{\sigma}{m}{-}
+\kappa^{q+1}_\sigma\,\omega^{2\nu_\sigma}\sum_{k=0}^{\infty}(-\ie\omega)^{2k-1}\;\turmr{q+1}{2k}{\sigma}{m}{+}
\end{split}\mylabel{HfettEins}
\intertext{and}
\fH^{2,\omega}_{\sigma,m}&:=\sum_{k=0}^{\infty}(-\ie\omega)^{2k}\;\turmr{q+1}{2k+1}{\sigma}{m}{-}
+\kappa^q_\sigma\,\omega^{2\nu_\sigma}\sum_{k=0}^{\infty}(-\ie\omega)^{2k}\;\turmr{q+1}{2k+1}{\sigma}{m}{+}
\quad,\mylabel{HfettZwei}\\
\begin{split}
\fE^{2,\omega}_{\sigma,m}&:=\frac{\ie}{\omega}\pdiv\fH^{2,\omega}_{\sigma,m}\\
&\,\,=\sum_{k=0}^{\infty}(-\ie\omega)^{2k-1}\;\turmd{q}{2k}{\sigma}{m}{-}
+\kappa^q_\sigma\,\omega^{2\nu_\sigma}\sum_{k=0}^{\infty}(-\ie\omega)^{2k-1}\;\turmd{q}{2k}{\sigma}{m}{+}
\quad,\end{split}\mylabel{EfettZwei}
\end{align}
where $\kappa^q_\sigma:=2\ie\nu_\sigma4^{-\nu_\sigma}\frac{\Gamma(1-\nu_\sigma)}{\Gamma(1+\nu_\sigma)}(-1)^{\nu_\sigma+1/2+\delta_{q,0}+\delta_{q,N}}$
and $\Gamma$ denotes the gamma-function.

These series of $q$- resp. $(q+1)$-forms converge uniformly on compact subsets of $\rNon$
and there they define $\cu$-forms. Moreover, they solve
$$(M+\ie\omega)\EH=(0,0)\qqtext{and}(\Delta+\omega^2)\EH=(0,0)$$
in $\rNon$ since clearly $(\pdiv E,\rot H)=(0,0)$\,.
For real frequencies $\omega\neq0$ they fulfill Sommerfeld's (componentwise for Helmholtz' equation)
and Maxwell's radiation condition and for nonreal frequencies $\omega\in\cz_+\ohne\rz$ they decay
exponentially at infinity. Moreover, $(\fE^{n,\omega}_{\sigma,m},\fH^{n,\omega}_{\sigma,m})$\,,
$n=1,2$\,, belong to the Sobolev spaces
$\qH{k}{q,q+1}{<\meh}{}\big(A(1)\big)$ for any $k\in\nzn$ as well as
\begin{align}
\fE^{1,\omega}_{\sigma,m}&=\frac{\omega^{\nu_\sigma}}{\beta_\sigma} r^{1-\Nh} H^1_{\nu_\sigma}(\omega\,r)\tauh T^q_{\sigma,m}\qquad,\non\\
\fH^{1,\omega}_{\sigma,m}&=\frac{\omega^{\nu_\sigma-1}}{\beta_\sigma}r^{-\Nh}\bigg(-\omega^q_\sigma H^1_{\nu_\sigma}(\omega\,r)\tauh S^{q+1}_{\sigma,m}\mylabel{EHschlangeEins}\\
&\qquad+\ie\Big(\big(N/2-(q+1)'\big) H^1_{\nu_\sigma}(\omega\,r)+\omega\,r(H^1_{\nu_\sigma})'(\omega\,r)\Big)\rhoh T^q_{\sigma,m}\bigg)\non
\intertext{and}
\fE^{2,\omega}_{\sigma,m}&=\frac{\omega^{\nu_\sigma-1}}{\beta_\sigma}r^{-\Nh}\bigg(\omega^{q-1}_\sigma H^1_{\nu_\sigma}(\omega\,r)\rhoh T^{q-1}_{\sigma,m}\non\\
&\qquad+\ie\Big(\big(N/2-q\big) H^1_{\nu_\sigma}(\omega\,r)+\omega\,r(H^1_{\nu_\sigma})'(\omega\,r)\Big)\tauh S^q_{\sigma,m}\bigg)\qquad,\mylabel{EHschlangeZwei}\\
\fH^{2,\omega}_{\sigma,m}&=\frac{\omega^{\nu_\sigma}}{\beta_\sigma} r^{1-\Nh} H^1_{\nu_\sigma}(\omega\,r)\rhoh S^q_{\sigma,m}\non
\end{align}
hold, where $H^1_{\nu_\sigma}$ denotes Hankel's first function and
$\beta_\sigma:=\ie\frac{2^{\nu_\sigma}}{\Gamma(1-\nu_\sigma)}(-1)^{\nu_\sigma+1/2}$\,.
For details and proofs we refer to \cite[Sektion 5.5]{paulydiss}. Compare also with
\cite[(84)]{complete} and \cite[section 4]{linelaz}.

Now let us turn to the calculation of the asymptotics of $\loesom\emsmz$\,.
$$\eta(\fE^{1,\omega}_{\sigma,m},\fH^{1,\omega}_{\sigma,m})\in\qH{\infty}{q,q+1}{<\meh}{\circ}(\Omega)$$
(using an obvious notation) fulfills the radiation condition and solves
$$(M+\ie\omega\Lambda)\eta(\fE^{1,\omega}_{\sigma,m},\fH^{1,\omega}_{\sigma,m})=(M+\ie\omega)\eta(\fE^{1,\omega}_{\sigma,m},\fH^{1,\omega}_{\sigma,m})
=C_{M,\eta}(\fE^{1,\omega}_{\sigma,m},\fH^{1,\omega}_{\sigma,m})\qquad.$$
Hence
\beq\loesom C_{M,\eta}(\fE^{1,\omega}_{\sigma,m},\fH^{1,\omega}_{\sigma,m})=\eta(\fE^{1,\omega}_{\sigma,m},\fH^{1,\omega}_{\sigma,m})\qquad.\mylabel{loesomCEHfette}\eeq
Since $C_{M,\eta}$ has compactly supported coefficients
$$\loesom(M+\ie\omega\Lambda)C_{M,\eta}(\turmd{q}{1}{\sigma}{m}{-},0)=C_{M,\eta}(\turmd{q}{1}{\sigma}{m}{-},0)$$
holds and therefore
\beq\loesom M C_{M,\eta}(\turmd{q}{1}{\sigma}{m}{-},0)=(\id-\ie\omega\loesom)C_{M,\eta}(\turmd{q}{1}{\sigma}{m}{-},0)\qquad.\mylabel{loesMCturmd}\eeq
With \eqref{Mquadesn} we compute
\beqanon
\loesom\emsmz&=&\loesom M C_{M,\eta}(\turmd{q}{1}{\sigma}{m}{-},0)+\loesom C_{M,\eta}(0,\turmr{q+1}{0}{\sigma}{m}{-})\\
&\streltxtty{\eqref{loesMCturmd}}{=}& C_{M,\eta}(\turmd{q}{1}{\sigma}{m}{-},0)+\loesom C_{M,\eta}\big((0,\turmr{q+1}{0}{\sigma}{m}{-})-\ie\omega(\turmd{q}{1}{\sigma}{m}{-},0)\big)\\
&\streltxtty{\eqref{Mesn},\eqref{loesomCEHfette}}{=}&\emsme-\eta(0,\turmr{q+1}{0}{\sigma}{m}{-})-\ie\omega\eta(\fE^{1,\omega}_{\sigma,m},\fH^{1,\omega}_{\sigma,m})\\
&&\quad+\ie\omega\loesom C_{M,\eta}\big((\fE^{1,\omega}_{\sigma,m},\fH^{1,\omega}_{\sigma,m})
-\frac{\ie}{\omega}(0,\turmr{q+1}{0}{\sigma}{m}{-})-(\turmd{q}{1}{\sigma}{m}{-},0)\big)\\
&=&\emsme-\ie\omega\eta(\turmd{q}{1}{\sigma}{m}{-},0)\\
&&\quad+\ie\omega(\loesom C_{M,\eta}-\eta)\big((\fE^{1,\omega}_{\sigma,m},\fH^{1,\omega}_{\sigma,m})\\
&&\qquad\qquad\qquad-\frac{\ie}{\omega}(0,\turmr{q+1}{0}{\sigma}{m}{-})-(\turmd{q}{1}{\sigma}{m}{-},0)\big)\qquad.
\eeqanon
According to Lemma \ref{Projlemmae} we may write
$$\emsme=\loesn\emsmz\qqtext{,}\eta(\turmd{q}{1}{\sigma}{m}{-},0)=\emsmn=\loesn\loes\emsmz$$
and obtain
$$\loes_{\omega,1}\emsmz=\ie\omega(\loesom C_{M,\eta}-\eta)\big((\fE^{1,\omega}_{\sigma,m},\fH^{1,\omega}_{\sigma,m})
-\frac{\ie}{\omega}(0,\turmr{q+1}{0}{\sigma}{m}{-})-(\turmd{q}{1}{\sigma}{m}{-},0)\big)\qquad.$$
Now inserting the expansions \eqref{EfettEins} and \eqref{HfettEins} in each case the first term of the
$(-)$-series of $\fE^{1,\omega}_{\sigma,m}$ resp. $\fH^{1,\omega}_{\sigma,m}$ is killed and we achieve
\begin{align}
\begin{split}
\loes_{\omega,1}\emsmz&=(\eta-\loesom C_{M,\eta})
\big(\sum_{k=1}^{\infty}(-\ie\omega)^{2k}(M-\ie\omega)(\turmd{q}{2k+1}{\sigma}{m}{-},0)\\
&\qquad\qquad+\kappa_\sigma\,\omega^{N+2\sigma}\sum_{k=0}^{\infty}(-\ie\omega)^{2k}(M-\ie\omega)(\turmd{q}{2k+1}{\sigma}{m}{+},0)\big)
\end{split}\mylabel{loesomeinsesm}
\end{align}
with $\kappa_\sigma:=\kappa^{q+1}_\sigma=2\ie\nu_\sigma4^{-\nu_\sigma}\frac{\Gamma(1-\nu_\sigma)}{\Gamma(1+\nu_\sigma)}(-1)^{\nu_\sigma+1/2}$\,.

Since the series converge locally uniformly for $\omega\in\czpomd\ohne\{0\}$ in $\rNon$\,,
they converge in particular in $\Lzqqpe{\loc}\big(\ol{\Omega}\big)$\,.
Consequently, the series $C_{M,\eta}\sum\dots$ converge in $\Lzqqpesom$ for all $s\in\rz$ because of
the compact support of $C_{M,\eta}$\,. Therefore, the continuity of $\loesom$ yields the convergence of the series
$\loesom C_{M,\eta}\sum\dots=\sum\loesom C_{M,\eta}\dots$ in $\Lzqqpeom{<\meh}$\,.

Let $\omb$ denote a bounded subdomain of $\om$ with $\supp\nabla\eta\subset\omb$\,. We look at
$$\fg:=C_{M,\eta}(M-\ie\omega)(\turmd{q}{2k+1}{\sigma}{m}{\pm},0)=(C_{\pdiv,\eta}\turmr{q+1}{2k}{\sigma}{m}{\pm},-\ie\omega C_{\rot,\eta}\turmd{q}{2k+1}{\sigma}{m}{\pm})\quad.$$
By \paulystaticremodd\, we get
$$\pdiv f=-\pdiv\eta\turmd{q}{2k-1}{\sigma}{m}{\pm}=0\qtext{,}\rot g=\ie\omega\rot\eta\turmr{q+1}{2k}{\sigma}{m}{\pm}=\ie\omega C_{\rot,\eta}\turmr{q+1}{2k}{\sigma}{m}{\pm}$$
and moreover $\fg$ is perpendicular to $\bonqom\times\bqpeom$ because $\supp\fg\subset\supp\nabla\eta$\,.
Furthermore, every $\norm{\,\cdot\,}_{\Lzqqpesom}$-norm is equivalent to the
$\norm{\,\cdot\,}_{\Lzqqpe{}(\omb)}$-norm for $\fg$\,.
\paulystaticcoefesti\, and \paulytimeharmPzeroiii\, yield (with generic constants $c>0$)

\begin{align*}
\normb{\loesom\fg}_{\Lzqqpe{}(\omb)}&\leq c\,\Big(\normb{\fg}_{\Lzqqpesom}+\norm{C_{\rot,\eta}\turmr{q+1}{2k}{\sigma}{m}{\pm}}_{\Lzqqpesom}\Big)\\
&\leq c\,\big(\norm{\turmd{q}{2k+1}{\sigma}{m}{\pm}}_{\Lzqqpe{}(\omb)}+\norm{\turmr{q+1}{2k}{\sigma}{m}{\pm}}_{\Lzqqpe{}(\omb)}\big)\leq c
\end{align*}
and thus also
$$\normb{(\eta-\loesom C_{M,\eta})(M-\ie\omega)(\turmd{q}{2k+1}{\sigma}{m}{\pm},0)}_{\Lzqqpe{}(\omb)}\leq c$$
all uniformly in $k$ and $\sigma,m$ as well as $\omega$ \big(See \paulystaticcoefesti\big).
For $K>\fJ$ we obtain by \eqref{loesomeinsesm}
$$\bigg|\normabst\bigg|\loes_{\omega,1}\emsmz-(\eta-\loesom C_{M,\eta})
\Big(\sum_{k=1}^{K-1}(-\ie\omega)^{2k}(M-\ie\omega)(\turmd{q}{2k+1}{\sigma}{m}{-},0)$$
$$\qquad\qquad+\kappa_\sigma\,\omega^{N+2\sigma}\sum_{k=0}^{K-1}(-\ie\omega)^{2k}
(M-\ie\omega)(\turmd{q}{2k+1}{\sigma}{m}{+},0)\Big)\bigg|\normabst\bigg|_{\Lzqqpe{}(\omb)}$$
$$\leq c\,\sum_{k=K}^{\infty}|\omega|^{2k}\leq c\,|\omega|^{2K}\qquad.$$
Once again let us introduce a new short notation:\mylabel{tildeell}
$$u\psim{\ell}v\quad:\equi\quad\norm{u-v}_{\Lzqqpe{}(\omb)}\leq c\,|\omega|^\ell\qtext{uniformly w. r. t.}\omega\in\czpomd\ohne\{0\}$$
Using this new notation we have shown so far
\begin{align}
\begin{split}
\loes_{\omega,1}\emsmz&\psim{2K}
\sum_{k=1}^{K-1}(-\ie\omega)^{2k}(\eta-\loesom C_{M,\eta})(M-\ie\omega)(\turmd{q}{2k+1}{\sigma}{m}{-},0)\\
&\qquad+\kappa_\sigma\,\omega^{N+2\sigma}\sum_{k=0}^{K-1}(-\ie\omega)^{2k}(\eta-\loesom C_{M,\eta})(M-\ie\omega)(\turmd{q}{2k+1}{\sigma}{m}{+},0)
\end{split}\mylabel{loesomeinsesmendl}
\end{align}
and the only unknown $\omega$-behavior is hidden in the terms
$$\loesom C_{M,\eta}(M-\ie\omega)(\turmd{q}{2k+1}{\sigma}{m}{\pm},0)\qquad.$$
Using $C_{M^2,\eta}=M C_{M,\eta}+C_{M,\eta} M$ and \paulystaticremodd\, we compute
\begin{align}
&\qquad\loesom C_{M,\eta}(M-\ie\omega)(\turmd{q}{2k+1}{\sigma}{m}{\pm},0)\non\\
&=\loesom C_{M^2,\eta}(\turmd{q}{2k+1}{\sigma}{m}{\pm},0)-\loesom(M+\ie\omega)C_{M,\eta}(\turmd{q}{2k+1}{\sigma}{m}{\pm},0)\non\\
&=\loesom C(\turmd{q}{2k+1}{\sigma}{m}{\pm},0)-\loesom(M+\ie\omega\Lambda)C_{M,\eta}(\turmd{q}{2k+1}{\sigma}{m}{\pm},0)\non\\
&=\loesom C(\turmd{q}{2k+1}{\sigma}{m}{\pm},0)-C_{M,\eta}(\turmd{q}{2k+1}{\sigma}{m}{\pm},0)\non
\intertext{and then for $k\in\nzn$}
\begin{split}
&\qquad(\eta-\loesom C_{M,\eta})(M-\ie\omega)(\turmd{q}{2k+1}{\sigma}{m}{\pm},0)\\
&=-\loesom C(\turmd{q}{2k+1}{\sigma}{m}{\pm},0)+(M-\ie\omega)\eta(\turmd{q}{2k+1}{\sigma}{m}{\pm},0)\qquad.
\end{split}\mylabel{loesomCturmdzkpe}
\end{align}
If $k\geq1$ we have
\begin{align*}
M^2\eta(\turmd{q}{2k+1}{\sigma}{m}{\pm},0)
&=C_{M^2,\eta}(\turmd{q}{2k+1}{\sigma}{m}{\pm},0)+\eta(\turmd{q}{2k-1}{\sigma}{m}{\pm},0)\\
&=C(\turmd{q}{2k+1}{\sigma}{m}{\pm},0)+\eta(\turmd{q}{2k-1}{\sigma}{m}{\pm},0)\in\regqom{0}{\loc}\qquad.
\end{align*}
We note once more
\beq\eta(\turmd{q}{2k+1}{\sigma}{m}{\pm},0)\in\regqom{0}{\loc}\mylabel{etaturmdinregloc}\eeq
by \paulystaticremodd\,. Thus according to \paulystaticitloes\, $\loes^2$ may be applied to
$ M^2\eta(\turmd{q}{2k+1}{\sigma}{m}{\pm},0)$ and we obtain
$$\eta(\turmd{q}{2k+1}{\sigma}{m}{\pm},0)=\loes^2\big(C(\turmd{q}{2k+1}{\sigma}{m}{\pm},0)+\eta(\turmd{q}{2k-1}{\sigma}{m}{\pm},0)\big)\qquad.$$
By \eqref{etaturmdinregloc} $\loes$ and $\loes^2$ are even well defined on
$\eta(\turmd{q}{2k-1}{\sigma}{m}{\pm},0)$\,, such that
$$\eta(\turmd{q}{2k+1}{\sigma}{m}{\pm},0)=\loes^2C(\turmd{q}{2k+1}{\sigma}{m}{\pm},0)+\loes^2\eta(\turmd{q}{2k-1}{\sigma}{m}{\pm},0)$$
holds. A short induction shows
$$\eta(\turmd{q}{2k+1}{\sigma}{m}{\pm},0)=\sum_{\ell=1}^{k}\loes^{2\ell}C(\turmd{q}{2k+3-2\ell}{\sigma}{m}{\pm},0)+\loes^{2k}\eta(\turmd{q}{1}{\sigma}{m}{\pm},0)\qquad,$$
and hence using $\Lambda^\me=\id$ on $\supp\eta$
\begin{align*}
\eta(\turmd{q}{2k+1}{\sigma}{m}{\pm},0)
&=\sum_{\ell=1}^{k}\loesn\loes^{2\ell-1}C(\turmd{q}{2k+3-2\ell}{\sigma}{m}{\pm},0)
+\loesn\loes^{2k-1}\eta(\turmd{q}{1}{\sigma}{m}{\pm},0)\qquad,\\
M\eta(\turmd{q}{2k+1}{\sigma}{m}{\pm},0)
&=\sum_{\ell=1}^{k}\loesn\loes^{2\ell-2}C(\turmd{q}{2k+3-2\ell}{\sigma}{m}{\pm},0)
+\loesn\loes^{2k-2}\eta(\turmd{q}{1}{\sigma}{m}{\pm},0)\qquad.
\end{align*}
We remind of $\eta(\turmd{q}{1}{\sigma}{m}{\pm},0)=\epmsmn$ and $C(\turmd{q}{1}{\sigma}{m}{\pm},0)=\epmsmz$\,.
Putting all together yields for $k\geq1$
\begin{align}
&\qquad(M-\ie\omega)\eta(\turmd{q}{2k+1}{\sigma}{m}{\pm},0)\non\\
&=(\loesn\loes^{2k-2}-\ie\omega\loesn\loes^{2k-1})\epmsmn\mylabel{mMpiomneta}\\
&\qquad\qquad\qquad+\sum_{\ell=1}^{k}(\loesn\loes^{2\ell-2}-\ie\omega\loesn\loes^{2\ell-1})
C(\turmd{q}{2k+3-2\ell}{\sigma}{m}{\pm},0)\qquad.\non
\end{align}
Inserting this formula into \eqref{loesomCturmdzkpe} and all this together into \eqref{loesomeinsesmendl}
we obtain
\begin{align}\begin{split}
&\qquad\loes_{\omega,1}\emsmz\\
&\psim{2K}\bSm+\kappa_\sigma\,\omega^{N+2\sigma}\bSp
+\kappa_\sigma\,\omega^{N+2\sigma}\big(-\loesom\epmsmz+(M-\ie\omega)\epmsmn\big)\qquad,
\end{split}\mylabel{loesomeinsesmendlzwei}\end{align}
where $\bSpm:=-\bSpme+\bSpmz+\bSpmd$ and
\begin{align*}
\bSpme&:=\sum_{k=1}^{K-1}(-\ie\omega)^{2k}\loesom C(\turmd{q}{2k+1}{\sigma}{m}{\pm},0)\qquad,\\
\bSpmz&:=\sum_{k=1}^{K-1}(-\ie\omega)^{2k}(\loesn\loes^{2k-2}-\ie\omega\loesn\loes^{2k-1})\epmsmn\qquad,\\
\bSpmd&:=\sum_{k=1}^{K-1}(-\ie\omega)^{2k}\sum_{\ell=1}^{k}
(\loesn\loes^{2\ell-2}-\ie\omega\loesn\loes^{2\ell-1})C(\turmd{q}{2k+3-2\ell}{\sigma}{m}{\pm},0)\qquad.
\end{align*}
Obviously
\beq\bSpmz=\sum_{k=2}^{2K-1}(-\ie\omega)^k\loesn\loes^{k-2}\epmsmn\qquad.\mylabel{Spmzeins}\eeq
In the double sums $\bSpmd$ we substitute $\ell$ by $j(\ell):=k-\ell+1$\,,
interchange the sums and again substitute $k$ by $i(k):=k-j$\,.
Then we denote the pair $(j,i)$ again by $(k,\ell)$\,. We get
\begin{align*}
\bSpmd&=\sum_{k=1}^{K-1}(-\ie\omega)^{2k}\sum_{\ell=0}^{2K-2k-1}
(-\ie\omega)^\ell\loesn\loes^\ell C(\turmd{q}{2k+1}{\sigma}{m}{\pm},0)
\intertext{and thus}
\bSpme-\bSpmd&=\sum_{k=1}^{K-1}
(-\ie\omega)^{2k}\loes_{\omega,2K-2k-1}C(\turmd{q}{2k+1}{\sigma}{m}{\pm},0)\qquad.
\end{align*}
We have $C(\turmd{q}{2k+1}{\sigma}{m}{\pm},0)\in\regqom{0}{\vox}$ and also for all $k\geq1$ and $j\leq2K$
as well as $\tilde{s}\in(2K-N/2,\infty)\ohne\pI$ according to Lemma \ref{Regortho}
$$C(\turmd{q}{2k+1}{\sigma}{m}{\pm},0)\in\regqom{j}{\tilde{s}}\qquad,$$
because for all $(\ell,\gamma,n)$ with $\ell\leq2K$ we may compute
$$\skpboml{C(\turmd{q}{2k+1}{\sigma}{m}{\pm},0)}{\Egnl}
=\skpboml{C(\turmd{q}{2k+1}{\sigma}{m}{\pm},0)}{\Hgnl}=0$$
using Lemma \ref{Tuermeorthoreg}, the expansions from Corollary \ref{loesEsmHgnkor}
and observing $2k+1\geq3$\,. In particular for $1\leq k\leq K-1$
$$C(\turmd{q}{2k+1}{\sigma}{m}{\pm},0)\in\regqom{2K-2k}{\tilde{s}}$$
holds. Therefore Theorem \ref{Regsatz} (i) yields uniformly in $\omega$
\big(and $k,\sigma,m$ by \paulystaticcoefesti\big)
\begin{align*}
&\qquad\normb{\loes_{\omega,2K-2k-1}C(\turmd{q}{2k+1}{\sigma}{m}{\pm},0)}_{\Lzqqpe{}(\omb)}\\
&\leq c\,|\omega|^{2K-2k}\normb{C(\turmd{q}{2k+1}{\sigma}{m}{\pm},0)}_{\Lzqqpeom{\tilde{s}}}
\leq c\,|\omega|^{2K-2k}
\end{align*}
and we obtain
$$\bSpme-\bSpmd\psim{2K}(0,0)\qquad.$$
Since $\emsmn=\loes^2\emsmz$ by Lemma \ref{Projlemmae}
we see from \eqref{Spmzeins}
$$\bSm_\text{II}=\sum_{k=2}^{2K-1}(-\ie\omega)^k\loesn\loes^k\emsmz$$
and inserting all in \eqref{loesomeinsesmendlzwei} \big(using \eqref{Spmzeins} again\big) we arrive at
\begin{align}
&\qquad\loes_{\omega,2K-1}\emsmz=\loes_{\omega,1}\emsmz-\bSm_\text{II}\non\\
&\psim{2K}\kappa_\sigma\,\omega^{N+2\sigma}\Big(\big(\sum_{k=2}^{2K-1}
(-\ie\omega)^k\loesn\loes^{k-2}+(M-\ie\omega)\big)\epsmn-\loesom\epsmz\Big)\qquad.\non
\intertext{For each $\nz\ni j\leq\fJ+1$ choosing some $K\in\nzn$ with $2K\geq j$ we finally obtain}
\begin{split}
&\qquad\loes_{\omega,j-1}\emsmz\\
&\psim{j}\kappa_\sigma\,\omega^{N+2\sigma}\Big(\big(\sum_{k=2}^{j-N-2\sigma-1}
(-\ie\omega)^k\loesn\loes^{k-2}+(M-\ie\omega)\big)\epsmn-\loesom\epsmz\Big)\qquad.
\end{split}\mylabel{loesomjmeMquadesm}
\end{align}
It remains to identify the terms. The equation
$$\epsme=M\eta(\turmd{q}{1}{\sigma}{m}{+},0)=C_{M,\eta}(\turmd{q}{1}{\sigma}{m}{+},0)+\eta(0,\turmr{q+1}{0}{\sigma}{m}{+})\qquad,$$
\eqref{etaturmdinregloc} and \paulystaticitloes\, show
\beq\loesn\epsme=\epsmn\qquad.\mylabel{loesMetaeta}\eeq
Using \eqref{loesMetaeta} in \eqref{loesomjmeMquadesm} we get
\begin{align}\begin{split}
&\qquad\loes_{\omega,j-1}\emsmz\\
&\psim{j}\kappa_\sigma\,\omega^{N+2\sigma}\big(\sum_{k=0}^{j-N-2\sigma-1}(-\ie\omega)^k\Lambda^\me\loes^k
\epsme-\loesom\epsmz\big)\qquad.
\end{split}\mylabel{loesomjmeMquadesmz}\end{align}
Since $\epsmz\in\regqom{0}{\vox}$ we may look at
$$(0,h):=\epsme-\loesn\epsmz$$
utilizing \paulystaticitloescor. With $M(0,h)=(0,0)$ and $\rot\mu h=0$ we have
$$\mu h\in\dH{q+1}{<-\Nh-\sigma}{\mu^\me}(\Omega)\cap\bqpeom^\bot\qqtext{,}h-\turmr{q+1}{0}{\sigma}{m}{+}\in\Lzqpeom{>-\Nh}\qquad.$$
Hence $\Hsm=h$ by Lemma \ref{EsmHgnlemma}\,. Finally \eqref{loesomjmeMquadesmz} turns to
\begin{align}
&\qquad\loes_{\omega,j-1}\emsmz\non\\
&\psim{j}\,\kappa_\sigma\,\omega^{N+2\sigma}\big(\sum_{k=0}^{j-N-2\sigma-1}(-\ie\omega)^k\Lambda^\me\Hsmk
-\loes_{\omega,j-N-2\sigma-1}\epsmz\big)\mylabel{loesomjmeMquadesmd}\\
&=:\kappa_\sigma\,\omega^{N+2\sigma}\cA^{j-N-2\sigma-1}_{\omega,\sigma,m}\qquad.\non
\intertext{Similar calculations but using the forms $(\fE^{2,\omega}_{\sigma,m},\fH^{2,\omega}_{\sigma,m})$
from \eqref{EfettZwei}, \eqref{HfettZwei} and looking at $(0,\turmr{q+1}{2k+1}{\sigma}{m}{\pm})$
yield a corresponding estimate for $\loesom\hmsmz$\,, i.e.}
&\qquad\loes_{\omega,j-1}\hmsmz\non\\
&\psim{j}\,\kappa_\sigma\omega^{N+2\sigma}\big(\sum_{k=0}^{j-N-2\sigma-1}(-\ie\omega)^k\Lambda^\me\Esmk
-\loes_{\omega,j-N-2\sigma-1}\hpsmz\big)\mylabel{loesomjmeMquadhsmr}\\
&=:\kappa_\sigma\,\omega^{N+2\sigma}\cB^{j-N-2\sigma-1}_{\omega,\sigma,m}\non\qquad.
\end{align}

\begin{lem}\mylabel{asymz}
Let $\fJ\in\nzn$ and $s\in(\fJ+1/2,\fJ+N/2)\ohne\pI$\,. Then for all bounded subdomains $\omb$ of $\om$
\begin{align*}
&\qquad\Big|\normabst\Big|\loes_{\omega,\fJ-1}\FG\\
&-\sum_{(k,\sigma,m)\in\tilde{\Theta}^q_{\fJ-1-N}}(-\ie\omega)^{N+k}\kappa_{k,\sigma}
\skpboml{\FG}{\Esmkmzse}\cB^{\fJ-1-N-k}_{\omega,\sigma,m}\\
&-\sum_{(k,\sigma,m)\in\tilde{\Theta}^{q+1}_{\fJ-1-N}}(-\ie\omega)^{N+k}\kappa_{k,\sigma}
\skpboml{\FG}{\Hsmkmzse}\cA^{\fJ-1-N-k}_{\omega,\sigma,m}\Big|\normabst\Big|_{\Lzqqpe{}(\omb)}\\
&\qquad\qquad\qquad\qquad=\calO\big(|\omega|^\fJ\big)\normb{\FG}_{\Lzqqpesom}
\end{align*}
holds uniformly with respect to $\omega\in\czpomd\ohne\{0\}$ and $\FG\in\regqnsom$\,.
Here we set $\kappa_{k,\sigma}:=\ie^{2k-2\sigma+N}\kappa_\sigma$ and $\tilde{\Theta}^q_j:=\setb{(k,\sigma,m)\in\nz^3_0}{2\sigma\leq k\leq j\,\wedge\,1\leq m\leq\mu^q_\sigma}$\,.
In particular for $j\leq\min\{\fJ,N\}$
$$\normb{\loes_{\omega,j-1}\FG}_{\Lzqqpe{}(\omb)}=\calO\big(|\omega|^j\big)\cdot\normb{\FG}_{\Lzqqpesom}\qquad.$$
\end{lem}

\begin{proof}
We insert the asymptotics of \eqref{loesomjmeMquadesmd}, \eqref{loesomjmeMquadhsmr}
in the estimates of Lemma \ref{asyme} with $j:=\fJ-k$\,.
The sums range over $\Theta^{q,\fJ}_s$ and $\Theta^{q+1,\fJ}_s$\,.
In particular for $(k,\sigma,m)\in\Theta^{q,\fJ}_s$ we get
$$0\leq k\leq\fJ-1\qqtext{,}0\leq\sigma<s-N/2-k-1\qqtext{,}1\leq m\leq\mu^q_\sigma\qquad.$$
Additionally we have the condition $k+2\sigma+N\leq\fJ-1$ since higher order terms may be swallowed
by the $\calO$-term. Because $\fJ+1/2<s<\fJ+N/2$ we only sum over
\begin{align*}
0\leq k&\leq\fJ-1-N\qquad,\\
0\leq\sigma&\leq\min\Big\{s-\Nh-k-1,\frac{\fJ-1-N-k}{2}\Big\}=\frac{\fJ-1-N-k}{2}\qquad,\\
1\leq m&\leq\mu^q_\sigma\qquad.
\end{align*}
We interchange the sums over $k$ and $\sigma$\,, set $\ell(k):=k+2\sigma$\,,
interchange $\sigma$ and $\ell$ and finally denote $\ell$ again by $k$\,.
This proves the first assertion. Once more recalling
\beq\epsmz\,,\,\hpsmz\in\regqom{0}{\vox}\mylabel{CDCReinsinregvox}\eeq
we apply \paulytimeharmPzeroiv\, and get
$$\loesom\epsmz\,,\,\loesom\hpsmz\psim{0}(0,0)\qquad.$$
Thus $\cA^{\ell}_{\omega,\sigma,m}\,,\,\cB^{\ell}_{\omega,\sigma,m}\psim{0}(0,0)$\,,
which yields the second assertion.
\end{proof}

In the following we often use without further reference an uniqueness result for
asymptotic expansions.

\begin{lem}\mylabel{eindeutigasym}
Let $\tilde{L},L\in\nzn$ and $x_{-\tilde{L}},\dots,x_L$ be elements of some normed space $X$\,.
Moreover, let
$$\norm{\sum_{\ell=-\tilde{L}}^{L}\omega^\ell x_\ell}_X=o\big(|\omega|^L\big)$$
hold uniformly with respect to $\czp\ohne\{0\}\ni\omega\to0$\,. Then all $x_\ell$ vanish.
\end{lem}

According to \eqref{loesomjmeMquadesmd}, \eqref{loesomjmeMquadhsmr} we have
$$\cA^k_{\omega,\sigma,m}=\bX^k_{\sigma,m}(\omega)-\loes_{\omega,k}\epsmz\qtext{,}
\cB^k_{\omega,\sigma,m}=\bY^k_{\sigma,m}(\omega)-\loes_{\omega,k}\hpsmz$$
with polynomials
\beq\bX^k_{\sigma,m}(\omega):=\sum_{\ell=0}^k(-\ie\omega)^\ell\Lambda^\me\Hsml\qtext{,}
\bY^k_{\sigma,m}(\omega):=\sum_{\ell=0}^k(-\ie\omega)^\ell\Lambda^\me\Esml\mylabel{bXbYDefi}\eeq
of degree $k$ in $\omega$\,.
By \eqref{CDCReinsinregvox} we may apply Lemma \ref{asymz} with $\fJ:=j$ to
$$\FG:=\epgnz\,,\,\hpgnuz\qquad,$$
which yields the asymptotics
\begin{align}
&\qquad\loes_{\omega,j-1}\epgnz\non\\
&\psim{j}(-\ie\omega)^N\sum_{(k,\sigma,m)\in\tilde{\Theta}^q_{j-1-N}}
(-\ie\omega)^k\beta^{k,\sigma,m}_{e,\gamma,n}\big(\bY^{j-1-N-k}_{\sigma,m}(\omega)
-\loes_{\omega,j-1-N-k}\hpsmz\big)\mylabel{asymCDXYloesom}\\
&\qquad+(-\ie\omega)^N\sum_{(k,\sigma,m)\in\tilde{\Theta}^{q+1}_{j-1-N}}
(-\ie\omega)^k\alpha^{k,\sigma,m}_{e,\gamma,n}\big(\bX^{j-1-N-k}_{\sigma,m}(\omega)
-\loes_{\omega,j-1-N-k}\epsmz\big)\non
\intertext{and}
&\qquad\loes_{\omega,j-1}\hpgnuz\non\\
&\psim{j}(-\ie\omega)^N\sum_{(k,\sigma,m)\in\tilde{\Theta}^q_{j-1-N}}
(-\ie\omega)^k\beta^{k,\sigma,m}_{h,\gamma,\nu}\big(\bY^{j-1-N-k}_{\sigma,m}(\omega)
-\loes_{\omega,j-1-N-k}\hpsmz\big)\mylabel{asymCRXYloesom}\\
&\qquad+(-\ie\omega)^N\sum_{(k,\sigma,m)\in\tilde{\Theta}^{q+1}_{j-1-N}}
(-\ie\omega)^k\alpha^{k,\sigma,m}_{h,\gamma,\nu}\big(\bX^{j-1-N-k}_{\sigma,m}(\omega)
-\loes_{\omega,j-1-N-k}\epsmz\big)\qquad,\non
\end{align}
where
\begin{align}
\beta^{k,\sigma,m}_{e,\gamma,n}&:=\kappa_{k,\sigma}\skpoml{\epgnz}{\Esmkmzse}\qquad,\mylabel{defbetaD}\\
\alpha^{k,\sigma,m}_{e,\gamma,n}&:=\kappa_{k,\sigma}\skpoml{\epgnz}{\Hsmkmzse}\qquad,\mylabel{defalphaD}\\
\beta^{k,\sigma,m}_{h,\gamma,\nu}&:=\kappa_{k,\sigma}\skpoml{\hpgnuz}{\Esmkmzse}\qquad,\mylabel{defbetaR}\\
\alpha^{k,\sigma,m}_{h,\gamma,\nu}&:=\kappa_{k,\sigma}\skpoml{\hpgnuz}{\Hsmkmzse}\qquad.\mylabel{defalphaR}
\end{align}
Thus there exist polynomials $\bXs^{\ell}_{\gamma,n}(\omega)$ and $\bYs^{\ell}_{\gamma,\nu}(\omega)$
of degree $\ell$ in $\omega$\,, such that
$$\loes_{\omega,j-1}\epgnz\psim{j}\bXs^{j-1}_{\gamma,n}(\omega)\qqtext{,}\loes_{\omega,j-1}\hpgnuz\psim{j}\bYs^{j-1}_{\gamma,\nu}(\omega)$$
hold. Since
$$\loes_{\omega,j}\epgnz\psim{j}\loes_{\omega,j-1}\epgnz
\qqtext{,}\loes_{\omega,j}\hpgnuz\psim{j}\loes_{\omega,j-1}\hpgnuz$$
the coefficients of $\bXs^{\ell}_{\gamma,n}(\omega)$ and $\bYs^{\ell}_{\gamma,\nu}(\omega)$ do not depend on $\ell$\,.
Consequently there exist forms
$$X^\ell_{\gamma,n}\,,\,Y^\ell_{\gamma,\nu}\in\Lzqqpelocom\qquad,$$
such that
\begin{align}
\cA^{j-1}_{\omega,\gamma,n}&\,\psim{j}\,\bX^{j-1}_{\gamma,n}(\omega)-\bXs^{j-1}_{\gamma,n}(\omega)=:\sum_{\ell=0}^{j-1}(-\ie\omega)^\ell X^\ell_{\gamma,n}\qquad,\mylabel{KoeffXdef}\\
\cB^{j-1}_{\omega,\gamma,\nu}&\,\psim{j}\,\bY^{j-1}_{\gamma,\nu}(\omega)-\bYs^{j-1}_{\gamma,\nu}(\omega)=:\sum_{\ell=0}^{j-1}(-\ie\omega)^\ell Y^\ell_{\gamma,\nu}\qquad.\mylabel{KoeffYdef}
\end{align}
We obtain immediately
\begin{align*}
\loes_{\omega,j-1}\emsmz&\psim{j}\kappa_\sigma\,\omega^{N+2\sigma}\sum_{\ell=0}^{{j-N-2\sigma-1}}(-\ie\omega)^\ell X^\ell_{\sigma,m}\qquad,\\
\loes_{\omega,j-1}\hmsmz&\psim{j}\kappa_\sigma\,\omega^{N+2\sigma}\sum_{\ell=0}^{{j-N-2\sigma-1}}(-\ie\omega)^\ell Y^\ell_{\sigma,m}
\intertext{and}
\loes_{\omega,j-1}\epgnz&\psim{j}\bXs^{j-1}_{\gamma,n}(\omega)\qqtext{,}
\loes_{\omega,j-1}\hpgnuz\psim{j}\bYs^{j-1}_{\gamma,n}(\omega)\qquad.
\end{align*}
\eqref{asymCDXYloesom} and \eqref{asymCRXYloesom} yield for $1\leq j\leq N$
\beq\loes_{\omega,j-1}\epgnz\,,\,\loes_{\omega,j-1}\hpgnuz\,\psim{j}\,(0,0)\mylabel{loesomjottmeDR}\eeq
and therefore we get for $\ell=0,\dots,N-1$
\beq X^\ell_{\gamma,n}=\Lambda^\me\Hgnl\qqtext{,}Y^\ell_{\gamma,\nu}=\Lambda^\me\Egnul\qquad.\mylabel{XgleichHYgleichE}\eeq
The higher order coefficients $X^\ell_{\gamma,n},Y^\ell_{\gamma,\nu}$
may be computed recursively utilizing \eqref{asymCDXYloesom}, \eqref{asymCRXYloesom}.
In particular we have for $j\geq N$
\begin{align}
\begin{split}
\bXs^{j}_{\gamma,n}(\omega)&=(-\ie\omega)^N\sum_{(k,\sigma,m)\in\tilde{\Theta}^q_{j-N}}(-\ie\omega)^k\beta^{k,\sigma,m}_{e,\gamma,n}\sum_{\ell=0}^{j-N-k}(-\ie\omega)^\ell Y^\ell_{\sigma,m}\\
&\qquad\qquad+(-\ie\omega)^N\sum_{(k,\sigma,m)\in\tilde{\Theta}^{q+1}_{j-N}}(-\ie\omega)^k\alpha^{k,\sigma,m}_{e,\gamma,n}\sum_{\ell=0}^{j-N-k}(-\ie\omega)^\ell X^\ell_{\sigma,m}\qquad,
\end{split}\mylabel{XsgleichYYsXXs}\\
\begin{split}
\bYs^{j}_{\gamma,\nu}(\omega)&=(-\ie\omega)^N\sum_{(k,\sigma,m)\in\tilde{\Theta}^q_{j-N}}(-\ie\omega)^k\beta^{k,\sigma,m}_{h,\gamma,\nu}\sum_{\ell=0}^{j-N-k}(-\ie\omega)^\ell Y^\ell_{\sigma,m}\\
&\qquad\qquad+(-\ie\omega)^N\sum_{(k,\sigma,m)\in\tilde{\Theta}^{q+1}_{j-N}}(-\ie\omega)^k\alpha^{k,\sigma,m}_{h,\gamma,\nu}\sum_{\ell=0}^{j-N-k}(-\ie\omega)^\ell X^\ell_{\sigma,m}\qquad.
\end{split}\mylabel{YsgleichYYsYYs}
\end{align}

Using \eqref{bXbYDefi} and \eqref{KoeffXdef}, \eqref{KoeffYdef}
we get the following recursion for the forms $X^\ell_{\gamma,n}$\,, $Y^\ell_{\gamma,\nu}$ and $\ell\geq N$
\begin{align}
\begin{split}
X^\ell_{\gamma,n}&=\Lambda^\me\Hgnl-\sum_{(k,\sigma,m)\in\tilde{\Theta}^q_{\ell-N}}\beta^{k,\sigma,m}_{e,\gamma,n} Y^{\ell-N-k}_{\sigma,m}\\
&\qquad\qquad\qquad\qquad-\sum_{(k,\sigma,m)\in\tilde{\Theta}^{q+1}_{\ell-N}}\alpha^{k,\sigma,m}_{e,\gamma,n} X^{\ell-N-k}_{\sigma,m}\qquad,
\end{split}\mylabel{RekX}\\
\begin{split}
Y^\ell_{\gamma,\nu}&=\Lambda^\me\Egnul-\sum_{(k,\sigma,m)\in\tilde{\Theta}^q_{\ell-N}}\beta^{k,\sigma,m}_{h,\gamma,\nu} Y^{\ell-N-k}_{\sigma,m}\\
&\qquad\qquad\qquad\qquad-\sum_{(k,\sigma,m)\in\tilde{\Theta}^{q+1}_{\ell-N}}\alpha^{k,\sigma,m}_{h,\gamma,\nu} X^{\ell-N-k}_{\sigma,m}\qquad.
\end{split}\mylabel{RekY}
\end{align}
Additionally we obtain
\begin{align}
X^\ell_{\gamma,n}\,,\,Y^\ell_{\gamma,\nu}&\in\Lin\{\Lambda^\me\Hgnl,\Lambda^\me\Egnul\}\non\\
&\qquad+\Lin\setb{X^{\ell-N-k}_{\sigma,m}\,,\,Y^{\ell-N-\tilde{k}}_{\tilde{\sigma},\tilde{m}}}{(k,\sigma,m)\in\tilde{\Theta}^{q+1}_{\ell-N}\,\wedge\,(\tilde{k},\tilde{\sigma},\tilde{m})\in\tilde{\Theta}^q_{\ell-N}}\non\\
&\subset\Lin\{\Lambda^\me\Hgnl,\Lambda^\me\Egnul\}
+\Lin\set{X^k_{\sigma,m}\,,\,Y^k_{\sigma,\tilde{m}}}{k+2\sigma\leq\ell-N}\non
\intertext{and a short induction shows}
\begin{split}
X^\ell_{\gamma,n}\,,\,Y^\ell_{\gamma,\nu}&\in\Lin\{\Lambda^\me\Hgnl,\Lambda^\me\Egnul\}\\
&\qquad+\Lin\set{\Lambda^\me\Esmk,\Lambda^\me\Hsmsk}{k+2\sigma\leq\ell-N}\qquad.
\end{split}\mylabel{MengeXY}
\end{align}
Moreover, our coefficient forms satisfy
\beq M X^0_{\gamma,n}=M Y^0_{\gamma,\nu}=(0,0)\mylabel{MXYgleichNull}\eeq
and for $\ell\leq N-1$
\beq\Lambda^\me M X^\ell_{\gamma,n}=X^{\ell-1}_{\gamma,n}\qqtext{,}\Lambda^\me M Y^\ell_{\gamma,\nu}=Y^{\ell-1}_{\gamma,\nu}\qquad.\mylabel{MXYgleichXY}\eeq
Once again by induction these equations hold true for all $\ell\geq N$\,.

We may formulate the main result of step one. For this let the coefficients
$X^\ell_{\gamma,n}$ and $Y^\ell_{\gamma,\nu}$ be defined recursively by \eqref{XgleichHYgleichE},
\eqref{RekX}, \eqref{RekY} and

\begin{defini}\mylabel{asymsatzlokaldef}
Let $\fJ\in\nzn$\,, $s\in(\fJ+1/2,\infty)\ohne\pI$ and $\FG\in\Lzqqpesom$\,.
Then for $j=0,\dots,\fJ-1-N$ we define the `correction operators'
\begin{align*}
\hat{\Gamma}_j\FG&:=\sum_{(k,\sigma,m)\in\tilde{\Theta}^q_j}\kappa_{k,\sigma}\skpboml{\FG}{\Esmkmzse}Y^{j-k}_{\sigma,m}\\
&\qquad\qquad+\sum_{(k,\sigma,m)\in\tilde{\Theta}^{q+1}_j}\kappa_{k,\sigma}\skpboml{\FG}{\Hsmkmzse}X^{j-k}_{\sigma,m}\qquad.
\end{align*}
\end{defini}

\begin{theo}\mylabel{asymsatzlokal}
Let $\fJ\in\nzn$ and $s\in(\fJ+1/2,\infty)\ohne\pI$\,.
Then for all bounded subdomains $\omb\subset\Omega$ the asymptotic
$$\normb{\loes_{\omega,\fJ-1}\FG
-(-\ie\omega)^N\sum_{j=0}^{\fJ-1-N}(-\ie\omega)^j\hat{\Gamma}_j\FG}_{\Lzqqpe{}(\omb)}$$
$$=\calO\big(|\omega|^\fJ\big)\normb{\FG}_{\Lzqqpesom}$$
holds uniformly with respect to $\FG\in\regqnsom$ and $\omega\in\czpomd\ohne\{0\}$\,.
\end{theo}

\begin{rem}\mylabel{asymsatzlokalbemGamma}
\begin{itemize}
\item[\bf(i)] The coefficients $X^\ell_{\gamma,n}$\,, $Y^\ell_{\gamma,\nu}$ have to be computed only for
$\ell,2\gamma\leq\fJ-1-N$ and $n=1,\dots,\mu^{q+1}_\gamma$\,, $\nu=1,\dots,\mu^q_\gamma$\,.
\item[\bf(ii)] Because of \eqref{MXYgleichNull}, \eqref{MXYgleichXY} the correction operators satisfy
$$\Lambda^\me M\,\hat{\Gamma}_j=\hat{\Gamma}_{j-1}\qquad,$$
where $\hat{\Gamma}_{-1}:=0$\,.
\item[\bf(iii)] By \eqref{MengeXY} we have
$$\hat{\Gamma}_j\FG\in\corom{q}{j}:=\Lin\set{\Lambda^\me\Esmk\,,\,\Lambda^\me\Hsnk}{k+2\sigma\leq j}\qquad.$$
Hence the correction operators
$$\hat{\Gamma}_j\,:\,\Lzqqpesom\,\To\,\corom{q}{j}$$
are degenerated (and clearly continuous).
\item[\bf(iv)] $\hat{\Gamma}_j\FG=(0,0)$ for $\FG\in\regqJsom$\,.
\end{itemize}
\end{rem}

\begin{rem}\mylabel{asymsatzlokalbemXYRek}
Utilizing the representations of $\Esmk\,,\,\Hsmk$ from Corollary \ref{loesEsmHgnkor}
and the orthogonality properties from Lemma \ref{Tuermeorthoreg} we may obtain a more detailed
recursive definition of the coefficient forms $X^\ell_{\gamma,n}\,,\,Y^\ell_{\gamma,\nu}$\,.
Namely looking at \eqref{defbetaD} and keeping in mind $k-2\sigma\geq0$ we see that
$\beta^{k,\sigma,m}_{e,\gamma,n}$ vanishes for odd $k-2\sigma+1$\,, i.e. even $k$\,.
However, for even $k-2\sigma+1$\,, i.e. odd $k$\,,  we get by Corollary \ref{loesEsmHgnkor} and
Lemma \ref{Tuermeorthoreg}
$$\beta^{k,\sigma,m}_{e,\gamma,n}=-\kappa_{k,\sigma}\,\xi^{k-2\sigma+1,\sigma,m}_{(1,\gamma,n,-)}\qquad.$$
Accordingly we achieve for odd $k$
$$\alpha^{k,\sigma,m}_{h,\gamma,\nu}=-\kappa_{k,\sigma}\,\zeta^{k-2\sigma+1,\sigma,m}_{(1,\gamma,\nu,-)}
\qtext{,}\alpha^{k,\sigma,m}_{e,\gamma,n}=\beta^{k,\sigma,m}_{h,\gamma,\nu}=0$$
and for even $k$
$$\alpha^{k,\sigma,m}_{h,\gamma,\nu}=0
\qtext{,}
\alpha^{k,\sigma,m}_{e,\gamma,n}=-\kappa_{k,\sigma}\,\xi^{k-2\sigma+1,\sigma,m}_{(1,\gamma,n,-)}
\qtext{,}
\beta^{k,\sigma,m}_{h,\gamma,\nu}=-\kappa_{k,\sigma}\,\zeta^{k-2\sigma+1,\sigma,m}_{(1,\gamma,\nu,-)}
\quad.$$
Now our recursion \eqref{RekX}, \eqref{RekY} appears in a more explicit shape\mylabel{XYRekexpl}
\begin{align*}
X^\ell_{\gamma,n}&:=\Lambda^\me\Hgnl
+\sum_{\substack{(k,\sigma,m)\in\tilde{\Theta}^q_{\ell-N}\\k\text{\rm\, odd}}}\kappa_{k,\sigma}\, \xi^{k-2\sigma+1,\sigma,m}_{(1,\gamma,n,-)}\, Y^{\ell-N-k}_{\sigma,m}\\
&\qquad\qquad\qquad\qquad\qquad+\sum_{\substack{(k,\sigma,m)\in\tilde{\Theta}^{q+1}_{\ell-N}\\k\text{\rm\, even}}}\kappa_{k,\sigma}\, \xi^{k-2\sigma+1,\sigma,m}_{(1,\gamma,n,-)}\, X^{\ell-N-k}_{\sigma,m}\qquad,\\
Y^\ell_{\gamma,\nu}&:=\Lambda^\me\Egnul
+\sum_{\substack{(k,\sigma,m)\in\tilde{\Theta}^q_{\ell-N}\\k\text{\rm\, even}}}\kappa_{k,\sigma}\, \zeta^{k-2\sigma+1,\sigma,m}_{(1,\gamma,\nu,-)}\, Y^{\ell-N-k}_{\sigma,m}\\
&\qquad\qquad\qquad\qquad\qquad+\sum_{\substack{(k,\sigma,m)\in\tilde{\Theta}^{q+1}_{\ell-N}\\k\text{\rm\, odd}}}\kappa_{k,\sigma}\, \zeta^{k-2\sigma+1,\sigma,m}_{(1,\gamma,\nu,-)}\, X^{\ell-N-k}_{\sigma,m}\qquad.
\end{align*}
\end{rem}

\begin{proof}
W. l. o. g. let $s\in(\fJ+1/2,\fJ+N/2)\ohne\pI$\,.
We insert the asymptotics \eqref{KoeffXdef}, \eqref{KoeffYdef} into the estimates of Lemma \ref{asymz}.
Introducing the new variable $j(\ell):=\ell+k$ and ordering the sums according to
$j$\,, $k$\,, $\sigma$\,, $m$ we have proved the theorem.
The assertions of the two remarks are easy consequences of the definition of the correction operators
and \eqref{MengeXY}, \eqref{MXYgleichNull}, \eqref{MXYgleichXY}.
\end{proof}

The first step is completed.

\subsection{Second step}

By the results obtained in \paulydecosecres\, we get the following essential decomposition:

\begin{lem}\mylabel{decolem}
Let $s>1-N/2$ and $s+1\notin\pI$ as well as $t\leq s$ and $t<N/2$\,.
Then every $\FG\in\Lzqqpesom$ may be uniquely decomposed into
\begin{align*}
\FG&=\Lambda\FrGd+\FdGr\qquad,
\intertext{where $\FrGd$ and $\FdGr$ are uniquely decomposed into}
\FrGd&=(\overset{\circ}{b},b)+\FrGdt+
\big(\sum_{I\in\cIb^{q,0}_{s}}\varphi_{I}(\vartheta^q_I)^\me\eta R^{q}_I,
\sum_{I\in\cIb^{q+1,0}_{s}}\psi_{I}\vartheta^{q+1}_I\eta D^{q+1}_I\big)\quad,\\
\FdGr&=\FdGrt+\big(\sum_{I\in\cIb^{q,0}_{s}}\varphi_{I}\pdiv\eta R^{q+1}_{{}_1I},
\sum_{I\in\cIb^{q+1,0}_{s}}\psi_{I}\rot\eta D^{q}_{{}_1I}\big)
\intertext{with constants $\varphi_{I},\psi_{I}\in\cz$\,,
where $\vartheta^q_I:=\ie\big(q'+\ei(I)\big)^{1/2}\big(q+\ei(I)\big)^{-1/2}$\,. Moreover}
(\overset{\circ}{b},b)&\in\Lin\bonq\times\Lin\bqpe\subset\ronqvoxnom\times\dqpevoxnom\qquad,\\
\FrGdt&\in\ronqsom\times\dqpesom\qquad,\\
\FrGd&\in\triqsom:=(\Lin\bonq\times\Lin\bqpe)
\dotplus\big(\bR{q}{s}{0}(\cIb^{q,0}_{s},\om)\times\bD{q+1}{s}{0}(\cIb^{q+1,0}_{s},\om)\big)\\
&\qquad\subset\;\ronqtnom\times\dqpetnom\qquad,\\
\FdGrt&\in\regqnsom\qquad,\\
\FdGr&\in\regqom{-1}{s}:=\bD{q}{s}{0}(\cIb^{q,0}_{s},\om)\times\bR{q+1}{s}{0}(\cIb^{q+1,0}_{s},\om)\\
&\qquad=\;\regqnsom\dotplus\big(\pdiv\eta\calR^{q+1}({}_1\cIb^{q,0}_{s})\times\rot\eta\calD^{q}({}_1\cIb^{q+1,0}_{s})\big)\\
&\qquad\subset\;\regqntom
\end{align*}
and all projections are continuous. We denote the projection $\FG\mapsto\FrGd$ by $\Pi$ and
the projection $\FG\mapsto\FdGr$ by $\Pi_{\reg}=\id-\Lambda\Pi$\,.\\
Shortly written we get
$$\Lzqqpesom=\big(\Lambda\triqsom\dotplus\regqom{-1}{s}\big)\cap\Lzqqpesom\qquad,$$
where $\triqsom=\Pi\,\Lzqqpesom$ and $\regqom{-1}{s}=\Pi_{\reg}\Lzqqpesom$\,.
\end{lem}

\begin{rem}\mylabel{decorem}
This lemma still holds true for $\tau$-$\pc{1}$-admissible transformations $\Lambda$ 
if we assume $\tau>0$\,, $\tau>s+1-N/2$ and $\tau\geq-s-1$\,.
\end{rem}

Let us consider for $\fJ\in\nzn$ and $s\in(\fJ+1/2,\infty)\ohne\pI$
some $\FG\in\Lzqqpesom$\,. We decompose $\FG$ according to the latter lemma.
$\FrGd$ solves
$$(M+\ie\omega\Lambda)\FrGd=\ie\omega\Lambda\FrGd$$
and satisfies the boundary, integrability and radiation condition since $t>-1/2$\,. Thus for the `trivial projection' we have
\beq\ie\omega\loesom\Lambda\FrGd=\FrGd\qqtext{, i.e.}\ie\omega\loesom\Lambda\Pi=\Pi\qquad.\mylabel{loesompi}\eeq
Looking at the `regular projection' we see that Theorem \ref{asymsatzlokal} determines the
asymptotic of $\FdGrt$ completely. Therefore it remains to compute the asymptotics of
$$(f_d,g_r):=\big(\sum_{I\in\cIb^{q,0}_{s}}\varphi_{I}\pdiv\eta R^{q+1}_{{}_1I},\sum_{J\in\cJb^{q+1,0}_{s}}\psi_{J}\rot\eta D^{q}_{{}_1J}\big)\qquad.$$
We note for $I=(-,0,\sigma,m)$ and $J=(-,0,\sigma,n)$
$$\eta R^{q+1}_{{}_1I}=\hmsm\qqtext{,}\eta D^{q}_{{}_1J}=\emsn\qquad.$$
Thus we have to calculate the asymptotics of
$$(\pdiv\eta R^{q+1}_{{}_1I},0)=M\hmsmn=\hmsme\qqtext{,}(0,\rot\eta D^{q}_{{}_1J})=M\emsnn=\emsne\qquad.$$
But this is quite easy since $\emsne$ and $\hmsme$ inherit the asymptotics of $\emsnz$\,, $\hmsmz$
derived in the first step. We discuss for example $\emsne$\,. $\emsne$ satisfies the boundary,
integrability and radiation condition as well as solves
$$(M+\ie\omega\Lambda)\emsne=\emsnz+\ie\omega\emsne\qquad.$$
This yields
$$\emsne=\loesom\emsnz+\ie\omega\loesom\emsne\qquad.$$
By Lemma \ref{Projlemmae} we get
\begin{align}
\loesom\emsne&=\frac{1}{\ie\omega}(\emsne-\loesom\emsnz)=\frac{1}{-\ie\omega}\loes_{\omega,0}\emsnz\non\\
&=\frac{1}{-\ie\omega}\big(\loes_{\omega,j}\emsnz+\sum_{\ell=1}^{j}(-\ie\omega)^\ell\loesn\loes^\ell\emsnz\big)\non\\
&=\frac{1}{-\ie\omega}\big(\loes_{\omega,j}\emsnz+\sum_{\ell=0}^{j-1}(-\ie\omega)^{\ell+1}\loesn\loes^\ell\emsne\big)\non
\intertext{and hence}
\loes_{\omega,j-1}\emsne&=\frac{1}{-\ie\omega}\loes_{\omega,j}\emsnz\qquad.\non
\intertext{By \eqref{loesomjmeMquadesmd} and \eqref{KoeffXdef} we obtain}
\begin{split}
\loes_{\omega,j-1}\emsne&\psim{j}\ie\kappa_\sigma\omega^{N+2\sigma-1}\sum_{\ell=0}^{j-N-2\sigma}(-\ie\omega)^\ell X^\ell_{\sigma,n}\\
&=\kappa_{0,\sigma}(-\ie\omega)^{N-1}\sum_{\ell=0}^{j-N-2\sigma}(-\ie\omega)^{\ell+2\sigma} X^\ell_{\sigma,n}
\end{split}\mylabel{asymee}
\intertext{and similarly}
\loes_{\omega,j-1}\hmsme&=\kappa_{0,\sigma}(-\ie\omega)^{N-1}\sum_{\ell=0}^{j-N-2\sigma}(-\ie\omega)^{\ell+2\sigma} Y^\ell_{\sigma,m}\qquad.\mylabel{asymhe}
\end{align}
In order to compute the coefficients $\varphi_I$\,, $\psi_J$ in terms of $\FG$ we have the following lemma.
Please compare to Lemma \ref{Regorthorechnung}.

\begin{lem}\mylabel{coefflem}
Let $s\in(1-N/2,\infty)\ohne\pI$ and $\FG\in\Lzqqpesom$\,. Then for all appropriate $\sigma,m$ and $\ell\geq1$
\begin{align*}
\text{\bf(i)}&&\skpboml{\FG}{\Esmn}&=\skp{F}{\Esm}_{\Lzqom{}}=-\varphi_I\qquad,\\
&&\skpboml{\FG}{\Hsmn}&=\skp{G}{\Hsm}_{\Lzqpeom{}}=-\psi_J\qquad,\\
\text{\bf(ii)}&&\skpboml{\FG}{\Esml}&=\skpboml{\FdGrt}{\Esml}\qquad,\\
&&\skpboml{\FG}{\Hsml}&=\skpboml{\FdGrt}{\Hsml}
\end{align*}
hold, where $I=(-,0,\sigma,m)$ resp. $J=(-,0,\sigma,m)$\,.
\end{lem}

\begin{rem}\mylabel{coefflemrem}
Here of course `appropriate' means that all $\Esmn$\,, $\Hsmn$\,, $\Esml$\,, $\Hsml$ are elements
of $\Lzqqpeom{-s}$\,.
More precisely we may pick $I=(-,0,\sigma,m)\in\cIb^{q,0}_s$ resp. $J=(-,0,\sigma,m)\in\cIb^{q+1,0}_s$
and $(\ell,\sigma,m)\in\Theta^{q,\ell}_s$ resp. $(\ell,\sigma,m)\in\Theta^{q+1,\ell}_s$\,.
\end{rem}

\begin{proof}
Let us first discuss {\bf(i)} and $\varphi_I$\,. During the proof we denote $I=(-,0,\sigma,m)$\,.
The representation of $F$ in Lemma \ref{decolem}
may be written as \big(see \paulydecotheomaindeltaeps\big)
\beq F=\eps\overset{\circ}{b}+\eps\hat{F}_r+\hat{F}_d
-\ie\sum_{I\in\cIb^{q,0}_{s}}\varphi_{I}(q'+\sigma)^{-1/2}\eps\Delta_\eps\eta P^q_I\qquad,\mylabel{Fdeco}\eeq
where $P^q_I:=(q+\sigma)^{1/2}R^q_{{}_2I}+\ie(q'+\sigma)^{1/2}D^q_{{}_2I}$
is a potential form and by \paulydecosecres
$$\hat{F}_r\in\bRqsom=\rot\pr{q-1}{s-1}{}{\circ}(\om)\qqtext{,}\hat{F}_d\in\bDqsom=\pdiv\pdi{q+1}{s-1}{}{}(\om)\qquad.$$
Since $\eps\Delta_\eps=\eps\rot\pdiv+\pdiv\rot=\Delta$ on $\supp\eta$ and $\Delta P^q_I=0$ we obtain
$$\eps\Delta_\eps\eta P^q_I=CP^q_I\qquad.$$
Clearly we have
$$\skp{F}{\Egn}_{\Lzqom{}}=-\ie\sum_{I\in\cIb^{q,0}_{s}}\varphi_{I}(q'+\sigma)^{-1/2}\skp{CP^q_I}{\Egn}_{\Lzqom{}}$$
and utilizing the expansion of $\Esm$ from Corollary \ref{loesEsmHgnkor} as well as
Lemma \ref{Tuermeorthoreg} and Remark \ref{Tuermeorthoregrem} we get
\begin{align*}
\skp{F}{\Egn}_{\Lzqom{}}&=
\sum_{I\in\cIb^{q,0}_{s}}\varphi_{I}\big((\vartheta^q_\sigma)^\me\skp{CR^q_{{}_2I}}{\Egn}_{\Lzqom{}}
+\skp{CD^q_{{}_2I}}{\Egn}_{\Lzqom{}}\big)\\
&=-\varphi_{(-,0,\gamma,n)}\qquad,
\end{align*}
since the sums vanish except for $R^q_{{}_2I}=\turmr{q}{2}{\gamma}{n}{-}$
resp. $D^q_{{}_2I}=\turmd{q}{2}{\gamma}{n}{-}$\,.
The other assertion of {\bf(i)} for $\psi_J$ may be shown in a similar way.

To prove {\bf(ii)} we write
$$F=\eps\overset{\circ}{b}+\eps\check{F}_r+\tilde{F}_d
+\sum_{I\in\cIb^{q,0}_{s}}\varphi_{I}\big((\vartheta^q_\sigma)^\me\rot\eta D^{q-1}_{{}_1I}+\pdiv\eta R^{q+1}_{{}_1I}\big)$$
with some $\check{F}_r\in\bRqsom$\,. For any $\ell$ we have
$$\skpboml{(\eps\overset{\circ}{b},0)}{\Esml}=\skpboml{(\eps\check{F}_r,0)}{\Esml}=0\qquad.$$
Thus it remains to show for all $I$
$$\skpboml{(Q^q_I,0)}{\Esml}=0\qquad,$$
where $Q^q_I:=(\vartheta^q_\sigma)^\me\rot\eta D^{q-1}_{{}_1I}+\pdiv\eta R^{q+1}_{{}_1I}$\,.
In order to prove this we compute
$$\ie(q'+\sigma)^{1/2}Q^q_I=CP^q_I-\pdiv C_{\rot,\eta}P^q_I-\rot C_{\pdiv,\eta}P^q_I$$
and obtain directly $\skpboml{(\rot C_{\pdiv,\eta}P^q_I,0)}{\Esml}=0$\,.
With the second term on the right hand side we proceed as follows. We write $\Esml=M\Lambda^\me\Esmle$
and since $C_{\rot,\eta}$ is compactly supported partial integration yields
\begin{align*}
&\qquad\skpboml{(\pdiv C_{\rot,\eta}P^q_I,0)}{\Esml}\\
&=-\skpboml{(0,\rot\pdiv C_{\rot,\eta}P^q_I)}{\Esmle}\\
&=-\skpboml{(0,\rot CP^q_I)}{\Esmle}\\
&\qquad\qquad+\ie(q'+\sigma)^{1/2}\skpboml{(0,\rot\pdiv\eta R^{q+1}_{{}_1I})}{\Esmle}\\
&=\skpboml{(CP^q_I,0)}{\Esml}\\
&\qquad\qquad+\ie(q'+\sigma)^{1/2}\skpboml{(0,CR^{q+1}_{{}_1I})}{\Esmle}\qquad,
\end{align*}
where the last equation follows once more by \paulystaticremodd. Consequently
$$\skpboml{(Q^q_I,0)}{\Esml}=-\skpboml{(0,CR^{q+1}_{{}_1I})}{\Esmle}\qquad.$$
Because $\ell\geq1$ the scalar products $\skpboml{(0,CR^{q+1}_{{}_1I})}{\Esmle}$ vanish once again
by Corollary \ref{loesEsmHgnkor}, Lemma \ref{Tuermeorthoreg} and Remark \ref{Tuermeorthoregrem}.
We note that since $\ell\geq1$ the scalar products $\skpboml{(CP^q_I,0)}{\Esml}$
vanish as well, though this is not necessary for the proof.
The other assertions of {\bf(ii)} are shown analogously.
\end{proof}

%
%
%

Putting all together we obtain the main result of step two.

\begin{defini}\mylabel{asymsatzlokalzweidef}
Let $\fJ\in\nzn$\,, $s\in(\fJ+1/2,\infty)\ohne\pI$ and $\FG\in\Lzqqpesom$\,.
For $j=0,\dots,\fJ-N$ we define the `correction operators'
\begin{align*}
\tilde{\Gamma}_j\FG&:=-\sum_{2\sigma\leq j\,,\,m}\kappa_{0,\sigma}\skpboml{\FG}{\Esmn}Y^{j-2\sigma}_{\sigma,m}\\
&\qquad\qquad-\sum_{2\sigma\leq j\,,\,m}\kappa_{0,\sigma}\skpboml{\FG}{\Hsmn}X^{j-2\sigma}_{\sigma,m}
\end{align*}
as well as $\Gamma_0:=\tilde{\Gamma}_0$ and $\Gamma_j:=\hat{\Gamma}_{j-1}+\tilde{\Gamma}_j$
for $j=1,\dots,\fJ-N$\,.
\end{defini}

\begin{theo}\mylabel{asymsatzlokalzwei}
Let $\fJ\in\nzn$ and $s\in(\fJ+1/2,\infty)\ohne\pI$\,.
Then for all bounded subdomains $\omb\subset\Omega$ the asymptotic
$$\big|\normabst\big|\loesom\FG+(-\ie\omega)^\me\Pi\FG-\sum_{j=0}^{\fJ-1}(-\ie\omega)^j\loesn\loes^j\Pi_{\reg}\FG$$
$$-(-\ie\omega)^{N-1}\sum_{j=0}^{\fJ-N}(-\ie\omega)^{j}\Gamma_j\FG\big|\normabst\big|_{\Lzqqpe{}(\omb)}
=\calO\big(|\omega|^\fJ\big)\normb{\FG}_{\Lzqqpesom}$$
holds uniformly with respect to $\FG\in\Lzqqpesom$ and $\omega\in\czpomd\ohne\{0\}$\,.
\end{theo}

\begin{rem}\mylabel{asymsatzlokalzweirem}
\begin{itemize}
\item[\bf(i)] By Lemma \ref{Regorthorechnung} on $\regqnsom$ we have $\tilde{\Gamma}_j=0$
and thus $\Gamma_j=\hat{\Gamma}_{j-1}$ for $j\geq1$ as well as $\Gamma_0=0$\,.
\item[\bf(ii)] $\Pi\FG=0$ and $\Pi_{\reg}\FG=\FdGrt=\FG$ hold for $\FG\in\regqnsom$\,.
\item[\bf(iii)] Because of \eqref{MXYgleichNull}, \eqref{MXYgleichXY} the correction operators satisfy
$\Lambda^\me M\,\tilde{\Gamma}_j=\tilde{\Gamma}_{j-1}$\,, where $\tilde{\Gamma}_{-1}:=0$\,.
Thus we have $\Lambda^\me M\,\Gamma_j=\Gamma_{j-1}$\,.
\item[\bf(iv)] From Definition \ref{asymsatzlokalzweidef} we get
$$\tilde{\Gamma}_j\FG\in\Lin\set{X^{j-2\sigma}_{\sigma,m}\,,\,Y^{j-2\sigma}_{\sigma,m}}{2\sigma\leq j}$$
and thus the correction operators
$$\tilde{\Gamma}_j\,,\,\Gamma_j\,:\,\Lzqqpesom\,\To\,\corom{q}{j}$$
are degenerated (and clearly continuous).
\end{itemize}
\end{rem}

\begin{proof}
From the arguments above and the continuity of the projections from Lemma \ref{decolem}
$$\loesom\FG+(-\ie\omega)^\me\FrGd$$
$$-\sum_{j=0}^{\fJ-1}(-\ie\omega)^j\loesn\loes^j\FdGrt-\sum_{j=0}^{\fJ-1}(-\ie\omega)^j\loesn\loes^j(f_d,g_r)$$
$$-(-\ie\omega)^N\sum_{j=0}^{\fJ-1-N}(-\ie\omega)^{j}\hat{\Gamma}_j\FdGrt$$
$$+(-\ie\omega)^{N-1}\sum_{\sigma\leq s-N/2\,,\,m}\kappa_{0,\sigma}\skpboml{\FG}{\Esmn}\sum_{\ell=0}^{\fJ-N-2\sigma}(-\ie\omega)^{\ell+2\sigma} Y^\ell_{\sigma,n}$$
$$+(-\ie\omega)^{N-1}\sum_{\sigma\leq s-N/2\,,\,m}\kappa_{0,\sigma}\skpboml{\FG}{\Hsmn}\sum_{\ell=0}^{\fJ-N-2\sigma}(-\ie\omega)^{\ell+2\sigma} X^\ell_{\sigma,n}$$
may be estimated by $c|\omega|^\fJ\normb{\FG}_{\Lzqqpesom}$
in the $\Lzqqpe{}(\omb)$-norm uniformly in $\omega$ and $\FG$\,.
We note $I=(-,0,\sigma,m)\in\cIb^{q,0}_{s}$\,, if and only if $J=(-,0,\sigma,m)\in\cJb^{q+1,0}_{s}$\,,
if and only if  $\sigma\leq s-N/2$ and $m=1,\dots$\,.
By Lemma \ref{coefflem} and Definition \ref{asymsatzlokaldef} we have
$\hat{\Gamma}_j\FdGrt=\hat{\Gamma}_j\FG$\,. Rearranging the latter two terms we obtain
$$\Big|\normabst\Big|\loesom\FG+(-\ie\omega)^\me\FrGd-\sum_{j=0}^{\fJ-1}(-\ie\omega)^j\loesn\loes^j\FdGr$$
$$-(-\ie\omega)^{N-1}\big(\sum_{j=1}^{\fJ-N}(-\ie\omega)^{j}\hat{\Gamma}_{j-1}\FG
+\sum_{j=0}^{\fJ-N}(-\ie\omega)^{j}\tilde{\Gamma}_j\FG\big)\Big|\normabst\Big|_{\Lzqqpe{}(\omb)}$$
$$=\calO\big(|\omega|^\fJ\big)\normb{\FG}_{\Lzqqpesom}\qquad,$$
which completes the proof.
\end{proof}

\subsection{Third step}

Now we approach estimates in weighted norms.
For this we compare our solutions with the solutions of the homogeneous whole space case.
Let us denote $L_\omega:=\loesom$ in the special case $\om=\rN$ and $\Lambda=\id$\,. $L_\omega$
is well defined on $\Lzqqpe{>\frac{1}{2}}$ for all $\cz_+\ohne\{0\}$ by \paulytimeharmtheofred\,
since there are no eigensolutions in this case. We obtain

\begin{lem}\mylabel{ganzraumasymptotik}
Let $\fJ\in\nz$ and $s>\fJ-1/2$ as well as $t<\min\{s,N/2\}-\fJ-1$\,. Then for $j=0,\dots,\fJ-1$
there exist bounded linear operators
$$\Phi_j\in B(\Lzqqpes,\Lzqqpet)\qqtext{,}\Psi_j\in B(\qLz{q-1,q+2}{s},\Lzqqpet)$$
and a constant $c>0$\,, such that
$$\normB{L_\omega\FG-\sum_{j=0}^{\fJ-1}(-\ie\omega)^j\big(\Phi_j\FG+(-\ie\omega)^\me\Psi_j(\pdiv F,\rot G)\big)}_{\Lzqqpet}$$
$$\leq c|\omega|^{\fJ}\Big(\normb{\FG}_{\Lzqqpes}+\frac{1}{|\omega|}\normb{(\pdiv F,\rot G)}_{\qLz{q-1,q+2}{s}}\Big)$$
holds uniformly with respect to $\omega\in\cz_+\ohne\{0\}$ and $\FG\in\Dqs\times\Rqpes$\,.
The assertion holds also true for $\fJ=0$ and $s,-t>1/2$\,, $t\leq s-(N+1)/2$\,.
\end{lem}

\begin{proof}
Using the fundamental solution for the scalar Helmholtz equation in $\rN$
$$\Phi_{\omega,\nu}(x)=\varphi_{\omega,\nu}\big(|x|\big)\qqtext{with}\varphi_{\omega,\nu}(t)=c_N\omega^\nu t^{-\nu}H_\nu^1(\omega t)\qquad,$$
where the constant $c_N$ only depends on the dimension $N$ and
$H_\nu^1(z)$ denotes Hankel's first function of index $\nu:=(N-2)/2$\,,
see \paulytimeharmsecasym, we may represent $\EH:=L_\omega\FG$
by \paulytimeharmfundrep, i.e.
\begin{align}
E_I&=G\star\rot\Phi^I_{\omega,\nu}+(-\ie\omega)F\star\Phi^I_{\omega,\nu}
-(-\ie\omega)^\me\pdiv F\star\pdiv\Phi^I_{\omega,\nu}\qquad,\mylabel{EuI}\\
H_J&=F\star\pdiv\Phi^J_{\omega,\nu}+(-\ie\omega)G\star\Phi^J_{\omega,\nu}
-(-\ie\omega)^\me\rot G\star\rot\Phi^J_{\omega,\nu}\qquad,\mylabel{HuJ}
\end{align}
if $E=E_I\pd x^I$ and $H=H_J\pd x^J$ as well as $\Phi^I_{\omega,\nu}:=\Phi_{\omega,\nu}\pd x^I$\,.
Here $\star$ is the convolution in $\rN$ for forms,
which simply is the sum of the scalar convolutions of their Euclidean components.
Utilizing Taylor's expansion theorem we get constants $c_j,c_j'\in\cz$
and functions $\text{\rm rem}_\fJ,\widetilde{\text{\rm rem}}_\fJ$\,,
such that for $t\in\rzp$ and $\omega\in\cz_+$ the expansions
\begin{align*}
\varphi_{\omega,\nu}(t)&=t^{2-N}\sum_{j=0}^{\fJ-2}c_j(\omega t)^j+\text{\rm rem}_\fJ(\omega t)t^{1-N+\fJ}\omega^{\fJ-1}\qquad,\\
\varphi_{\omega,\nu}'(t)&=t^{1-N}\sum_{j=0}^{\fJ-1}c_j'(\omega t)^j+\widetilde{\text{\rm rem}}_\fJ(\omega t)t^{1-N+\fJ}\omega^{\fJ}
\end{align*}
hold. The remainder functions $\text{\rm rem}_\fJ(z)$ and $\widetilde{\text{\rm rem}}_\fJ(z)$
are uniformly bounded with respect to $z\in\cz_+$ and the bounds only depend on $N$ and $\fJ$\,. Inserting these
Taylor representations into \eqref{EuI}, \eqref{HuJ} we obtain
\begin{align}\begin{split}
\EH&=\sum_{j=0}^{\fJ-1}(-\ie\omega)^j\big(\Phi_j\FG+(-\ie\omega)^\me\Psi_j(\pdiv F,\rot G)\big)\\
&\qquad\qquad+\omega^{\fJ}\big(\text{\rm Rem}_{\omega,\fJ}\FG+\frac{1}{\omega}\widetilde{\text{\rm Rem}}_{\omega,\fJ}(\pdiv F,\rot G)\big)\qquad,
\end{split}\mylabel{Taylordarstellung}\end{align}
where $\Phi_j$ and $\Psi_j$ resp. $\text{\rm Rem}_{\omega,\fJ}$ and $\widetilde{\text{\rm Rem}}_{\omega,\fJ}$
are convolution operators with integral kernels of shape
$b_j(x,y)|x-y|^{j+1-N}$ for $j=0,\dots,\fJ-1$ resp. $b_\text{\rm Rem}(x,y,\omega)|x-y|^{\fJ+1-N}$\,.
The kernel parts $b_j(x,y)$ are uniformly bounded with respect to $x,y\in\rN$ and independent of $\omega$\,.
Moreover, the kernel parts $b_\text{\rm Rem}(x,y,\omega)$
are uniformly bounded with respect to $x,y\in\rN$ and $\omega\in\cz_+$\,.
Thus it remains to show that the kernels $|x-y|^{j+1-N}$\,, $j=0,\dots,\fJ$ generate bounded linear
operators from $\Lzs$ to $\Lzt$\,. All kernels belong to $\text{L}^1_{\loc}$ and grow with $j$ if
$|x-y|>1$\,. Therefore we only have to discuss the worst kernel $|x-y|^{\fJ+1-N}$\,.
The assertion follows now by \cite[Lemma 1]{mcowen} and \cite[Lemma 13]{linelaz} as well as some case studies.
For a more detailed proof we refer to \cite[Sektionen 5.1-5.3]{paulydiss}.
\end{proof}

According to \cite[Theorem 4]{sphharm} there exist continuous projections
\begin{align*}
\pi&:\Lzqqpes\to\rqsn\times\dqpesn&&;\qquad\FG\mapsto(F_\text{\rm R},G_\text{\rm D})&&,\\
\pi_{\reg}&:\Lzqqpes\to\dqsn\times\rqpesn&&;\qquad\FG\mapsto(F_\text{\rm D},G_\text{\rm R})&&,\\
\pi_{\calS}&:\Lzqqpes\to\calS^q_s\times\calS^{q+1}_s&&;\qquad\FG\mapsto(F_\calS,G_\calS)&&,
\end{align*}
such that each $\FG\in\Lzqqpes$ may be uniquely decomposed as
$$\FG=(F_\text{\rm R},G_\text{\rm D})+(F_\text{\rm D},G_\text{\rm R})+(F_\calS,G_\calS)\qquad.$$

\begin{cor}\mylabel{ganzraumasymptotikkor}
Let $\fJ\in\nz$ and $s>\fJ-1/2$ as well as $t<\min\{s,N/2-1\}-\fJ-1$\,. Then
$$\big|\normabst\big|L_\omega\FG
-\sum_{j=0}^{\fJ-1}(-\ie\omega)^j\Phi_j(F_\text{\rm D}+F_\calS,G_\text{\rm R}+G_\calS)
+(-\ie\omega)^\me(F_\text{\rm R},G_\text{\rm D})$$
$$-\sum_{j=-1}^{\fJ-1}(-\ie\omega)^j\Psi_{j+1}(\pdiv F_\calS,\rot G_\calS)\big|\normabst\big|_{\Lzqqpet}
=\calO\big(|\omega|^{\fJ}\big)\normb{\FG}_{\Lzqqpes}$$
holds uniformly with respect to $\omega\in\cz_+\ohne\{0\}$ and $\FG\in\Lzqqpes$\,, if $\omega$ ranges in a bounded set.
The estimate remains valid even for $\fJ=0$ and $s,-t>1/2$\,, $t\leq s-(N+1)/2$\,.
\end{cor}

\begin{proof}
We easily see $L_\omega(F_\text{\rm R},G_\text{\rm D})=-(-\ie\omega)^\me(F_\text{\rm R},G_\text{\rm D})$\,.
\big(Compare to \eqref{loesompi}.\big)
Moreover, Lemma \ref{ganzraumasymptotik} may be applied to $(F_\text{\rm D},G_\text{\rm R})$ and we get
$$\normb{L_\omega(F_\text{\rm D},G_\text{\rm R})-\sum_{j=0}^{\fJ-1}(-\ie\omega)^j\Phi_j(F_\text{\rm D},G_\text{\rm R})}_{\Lzqqpet}
=\calO\big(|\omega|^{\fJ}\big)\normb{\FG}_{\Lzqqpes}\qquad.$$
Furthermore, Lemma \ref{ganzraumasymptotik} may also be applied to
$(F_\calS,G_\calS)$ but with $\fJ+1$ instead of $\fJ$ as well as $\tilde{s}:=s+1$ and $\tilde{t}:=t$\,.
We achieve
$$\normb{L_\omega(F_\calS,G_\calS)-\sum_{j=0}^{\fJ}(-\ie\omega)^j\Phi_j(F_\calS,G_\calS)
-\sum_{j=-1}^{\fJ-1}(-\ie\omega)^j\Psi_{j+1}(\pdiv F_\calS,\rot G_\calS)}_{\Lzqqpet}$$
\begin{align*}
&\leq c|\omega|^{\fJ+1}\Big(\normb{(F_\calS,G_\calS)}_{\Lzqqpe{\tilde{s}}}+\frac{1}{|\omega|}\normb{(\pdiv F_\calS,\rot G_\calS)}_{\qLz{q-1,q+2}{\tilde{s}}}\Big)\\
&\leq c|\omega|^{\fJ}\normb{(F_\calS,G_\calS)}_{\pdi{q}{\tilde{s}}{}{}\times\pr{q+1}{\tilde{s}}{}{}}\qquad.
\end{align*}
Now the $\Lzqqpet$-norm of the term $(-\ie\omega)^{\fJ}\Phi_j(F_\calS,G_\calS)$
may be swallowed by the right hand side,
which itself can be further estimated by $\calO\big(|\omega|^{\fJ}\big)\normb{\FG}_{\Lzqqpes}$
since $\calS^q_s$ are finite dimensional subspaces of $\cqun$ for all $q$ and $s$\,.
Putting all together yields the desired assertion.
\end{proof}

We are able to formulate the main result of this section:

\begin{theo}\mylabel{asymsatzweighted}
Let $\fJ\in\nz$ and $s\in(\fJ+1/2,\infty)\ohne\pI$ as well as $t<\min\{s,N/2-1\}-\fJ-1$\,.
Then the asymptotic
$$\big|\normabst\big|\loesom\FG+(-\ie\omega)^\me\Pi\FG-\sum_{j=0}^{\fJ-1}(-\ie\omega)^j\loesn\loes^j\Pi_{\reg}\FG$$
$$-(-\ie\omega)^{N-1}\sum_{j=0}^{\fJ-N}(-\ie\omega)^{j}\Gamma_j\FG\big|\normabst\big|_{\Lzqqpetom}=\calO\big(|\omega|^\fJ\big)\normb{\FG}_{\Lzqqpesom}$$
holds uniformly with respect to $\FG\in\Lzqqpesom$ and $\omega\in\czpomd\ohne\{0\}$\,.
This asymptotic holds for $\fJ=0$ as well, 
if we replace the assumptions on the weight $t$ by $t<-1/2$ and $t\leq s-(N+1)/2$\,.
\end{theo}

\begin{proof}
Let us define operators ${\calK}_j$ via
\begin{align}
{\calK}_{-1}&:=-\Pi\qquad,\non\\
{\calK}_j&:=\loesn\loes^j\Pi_{\reg}&&,&j&=0,\dots,N-2\qquad,\non\\
{\calK}_j&:=\loesn\loes^j\Pi_{\reg}+\Gamma_{j-N+1}&&,&j&=N-1,\dots,\fJ-1\non
\intertext{and for $J\leq\fJ-1$}
\loes_{\omega,J}^{\calK}&:=\loesom-\sum_{j=-1}^{J}(-\ie\omega)^j{\calK}_j\qquad.\mylabel{deflomjk}
\end{align}

Then we have to show uniformly with respect to $\omega$ and $\FG$
$$\normb{\loes_{\omega,\fJ-1}^{\calK}\FG}_{\Lzqqpetom}
=\calO\big(|\omega|^\fJ\big)\normb{\FG}_{\Lzqqpesom}\qquad.$$
From now on all estimates are to be understood uniformly with respect to $\omega$ and $\FG$\,.
We want to combine the asymptotics in local norms proved in the second step
with the whole space asymptotics in weighted norms from the latter corollary.
Since we have by Theorem \ref{asymsatzlokalzwei} for every bounded subdomain $\omb$ of $\Omega$
\beq\normb{\loes_{\omega,\fJ-1}^{\calK}\FG}_{\Lzqqpe{}(\omb)}
=\calO\big(|\omega|^\fJ\big)\normb{\FG}_{\Lzqqpesom}\qquad,\mylabel{loesomkabschbound}\eeq
we get immediately
$$\normb{(1-\eta)\loes_{\omega,\fJ-1}^{\calK}\FG}_{\Lzqqpetom}=\calO\big(|\omega|^\fJ\big)\normb{\FG}_{\Lzqqpesom}$$
and it remains to estimate $\normb{\eta\loes_{\omega,\fJ-1}^{\calK}\FG}_{\Lzqqpet}$\,.

To do so let $\omega\in\czpomd\ohne\{0\}$ and $\FG\in\Lzqqpesom$\,.
According to \paulydecotheomaindeltaeps\, \big(compare with Lemma \ref{decolem} and \eqref{Fdeco}\big) we decompose
$$\FG=\Lambda\big(\ub{(\overset{\circ}{b},b)+(\hat{F}_r,\hat{G}_d)}_{\bds=:\hat{\Pi}\FG\eds}\big)
+\ub{(\hat{F}_d,\hat{G}_r)}_{\bds=:\hat{\Pi}_{\reg}\FG\eds}
+\ub{\big(\sum_{I\in\cIb^{q,0}_{s}}\hat{\varphi}_{I}CP^q_I,
\sum_{J\in\cIb^{q+1,0}_{s}}\hat{\psi}_{J}CP^{q+1}_J\big)}_{\bds=:\hat{\Pi}_C\FG\eds}$$
with continuous projections, where $CP^q_I\in\cqun$
and $(\overset{\circ}{b},b)\in\Lin\bonq\times\Lin\bqpe$ as well as
$$(\hat{F}_r,\hat{G}_d)\in\bRom{q}{s}{0}\times\bDom{q+1}{s}{0}\qqtext{,}
(\hat{F}_d,\hat{G}_r)\in\regqnsom\qquad.$$
By Lemma \ref{coefflem} we have
$$|\hat{\varphi}_{I}|+|\hat{\psi}_{J}|\leq c\normb{\FG}_{\Lzqqpesom}\qquad.$$
The trick is to apply this decomposition using another cut-off function $\tilde{\eta}$\,,
which satisfies $\supp\nabla\tilde{\eta}\subset\ol{Z_{r_3,r_4}}$\,,
whereas $\supp\nabla\eta\subset\ol{Z_{r_1,r_2}}$\,. More precisely we set
$$\tilde{\eta}:=\check{\eta}\circ r\qqtext{,}\check{\eta}(t):=\mbox{\boldmath$\eta$}\big(1+\frac{t-r_3}{r_4-r_3}\big)$$
and note $C=C_{\Delta,\tilde{\eta}}$ in this case. Since
$$\loesom\Lambda\hat{\Pi}\FG=-(-\ie\omega)^\me\hat{\Pi}\FG$$
it suffices to discuss $\eta\loesom\FGd$ with $\FGd:=(\hat{\Pi}_{\reg}+\hat{\Pi}_C)\FG$\,.

$\eta\loesom\FGd\in\Rqkmeh\times\Dqpekmeh$ satisfies the radiation condition and solves
\beq(M+\ie\omega)\eta\loesom\FGd=(M+\ie\omega\Lambda)\eta\loesom\FGd=\fg\mylabel{MaxEHFGfg}\eeq
with $\fg:=(\eta+C_{M,\eta}\loesom)\FGd\in\Lzqqpes$\,. (Without further comments here
and in the following we often identify forms with their extensions by zero to $\rN$\,.)
Thus $\eta\loesom\FGd=L_\omega\fg$ or in another notation
\beq\eta\loesom=L_\omega(\eta\id+C_{M,\eta}\loesom)\mylabel{etaloesomL}\eeq
holds even on the whole space $\Lzqqpesom$\,. 
By Corollary \ref{ganzraumasymptotikkor} there exist bounded linear operators
$\Xi_{-1},\dots,\Xi_{\fJ-1}$ mapping $\Lzs$ to $\Lzt$\,, which satisfy
\beq\normb{L_{\omega,\fJ-1}\fg}_{\Lzqqpet}=\calO\big(|\omega|^{\fJ}\big)\normb{\fg}_{\Lzqqpes}\qquad,\mylabel{LomfJ}\eeq
where
\beq L_{\omega,\fJ-1}:=L_{\omega}-\sum_{j=-1}^{\fJ-1}(-\ie\omega)^j\Xi_j\qquad.\mylabel{definiLomfJ}\eeq
Moreover, $\normb{\fg}_{\Lzqqpes}$ can be further estimated by
$$\normb{\fg}_{\Lzqqpes}\leq c\Big(\normb{\FG}_{\Lzqqpesom}
+\normb{C_{M,\eta}\loesom\FGd}_{\Lzqqpe{}(\supp\nabla\eta)}\Big)\qquad.$$
Using \paulytimeharmPzeroiv\, we can estimate
$$\normb{\loesom\hat{\Pi}_{\reg}\FG}_{\Lzqqpe{}(\supp\nabla\eta)}
\leq c\normb{\hat{\Pi}_{\reg}\FG}_{\Lzqqpesom}\leq c\normb{\FG}_{\Lzqqpesom}\quad.$$
Looking at some term of $\hat{\Pi}_C\FG$ we see
$$C_{M,\eta}\loesom(CP^{q}_I,CP^{q+1}_J)
=C_{M,\eta}\loesom(\pdiv\rot\tilde{\eta}P^{q}_I,\rot\pdiv\tilde{\eta}P^{q+1}_J)$$
because
$$\loesom(\rot\pdiv\tilde{\eta}P^{q}_I,\pdiv\rot\tilde{\eta}P^{q+1}_J)
=\frac{1}{\ie\omega}\,(\rot\pdiv\tilde{\eta}P^{q}_I,\pdiv\rot\tilde{\eta}P^{q+1}_J)
\in\Lzqqpe{<\Nh}$$
and thus $C_{M,\eta}\loesom(\rot\pdiv\tilde{\eta}P^{q}_I,\pdiv\rot\tilde{\eta}P^{q+1}_J)=0$
since $\supp\nabla\eta\cap\supp\tilde{\eta}=\emptyset$\,.
With some $\tilde{s}\in(1/2,N/2)$ we have
$(\pdiv\rot\tilde{\eta}P^{q}_I,\rot\pdiv\tilde{\eta}P^{q+1}_J)\in\regqnom{\tilde{s}}$
and therefore we may utilize \paulytimeharmPzeroiv\, once more to estimate
$$\normb{C_{M,\eta}\loesom(CP^{q}_I,CP^{q+1}_J)}_{\Lzqqpe{}(\supp\nabla\eta)}
\leq c\normb{(\pdiv\rot\tilde{\eta}P^{q}_I,\rot\pdiv\tilde{\eta}P^{q+1}_J)}_{\Lzqqpeom{\tilde{s}}}\leq c\quad.$$
Hence we obtain
$$\normb{C_{M,\eta}\loesom\hat{\Pi}_{C}\FG}_{\Lzqqpe{}(\supp\nabla\eta)}\leq c\normb{\FG}_{\Lzqqpesom}\qquad.$$
Putting all together yields
$$\normb{\fg}_{\Lzqqpes}\leq c\normb{\FG}_{\Lzqqpesom}$$
and after inserting in \eqref{LomfJ}
$$\normb{L_{\omega,\fJ-1}\fg}_{\Lzqqpet}=\calO\big(|\omega|^{\fJ}\big)\normb{\FG}_{\Lzqqpesom}\qquad.$$
Now using \eqref{definiLomfJ} and \eqref{MaxEHFGfg}, \eqref{etaloesomL} we obtain
$$\eta\loesom\FGd=L_{\omega,\fJ-1}\fg+\sum_{j=-1}^{\fJ-1}(-\ie\omega)^j\Xi_j(\eta+C_{M,\eta}\loesom)\FGd\qquad.$$
Collecting terms and utilizing \eqref{deflomjk} gives

$$\eta\loesom\FG+(-\ie\omega)^\me\hat{\Pi}\FG
-\sum_{j=-1}^{\fJ-1}(-\ie\omega)^j\Xi_j\eta(\hat{\Pi}_{\reg}+\hat{\Pi}_C)\FG$$
\beq-\sum_{j=-1}^{\fJ-1}\sum_{k=-1}^{\fJ-1-j}(-\ie\omega)^{j+k}\Xi_jC_{M,\eta}{\calK}_k(\hat{\Pi}_{\reg}
+\hat{\Pi}_C)\FG\mylabel{etaloesomFGabsch}\eeq
$$=L_{\omega,\fJ-1}\fg
+\sum_{j=-1}^{\fJ-1}(-\ie\omega)^j\Xi_jC_{M,\eta}\loes_{\omega,\fJ-1-j}^{\calK}\FGd\qquad.$$
Moreover, the continuity of the operators $\Xi_j$ from $\Lzs$ to $\Lzt$
as well as Theorem \ref{asymsatzlokalzwei} yield
\begin{align*}
\normb{\Xi_jC_{M,\eta}\loes_{\omega,\fJ-1-j}^{\calK}\FGd}_{\Lzqqpet}
&\leq c\normb{\loes_{\omega,\fJ-1-j}^{\calK}\FGd}_{\Lzqqpe{}(\supp\nabla\eta)}\\
&\leq c|\omega|^{\fJ-j}\normb{\FGd}_{\Lzqqpesom}\qquad.
\end{align*}
Therefore the right hand side of \eqref{etaloesomFGabsch} behaves in the $\Lzqqpet$-norm like
$$\calO\big(|\omega|^\fJ\big)\normb{\FG}_{\Lzqqpesom}$$
and thus so does the left hand side. By \eqref{loesomkabschbound}
$$\normb{\eta\loes_{\omega,\fJ-1}^{\calK}\FG}_{\Lzqqpe{}(\omb)}=\calO\big(|\omega|^\fJ\big)\normb{\FG}_{\Lzqqpesom}$$
holds also for every bounded domain $\omb\subset\rN$\,.
Applying Lemma \ref{eindeutigasym} we find that the left hand side of \eqref{etaloesomFGabsch} equals
$\eta\loes_{\omega,\fJ-1}^{\calK}\FG$ and this yields finally
$$\normb{\eta\loes_{\omega,\fJ-1}^{\calK}\FG}_{\Lzqqpet}=\calO\big(|\omega|^\fJ\big)\normb{\FG}_{\Lzqqpesom}\qquad,$$
which completes our proof.
\end{proof}

\subsection{Fourth step}

We are ready to face the proof of the Main Theorem.
This last step may be done by an abstract argument similar to \cite[Lemma 12]{complete}.
Thus our aim is to identify two Banach spaces $X$ and $Y$\,, such that
all operators involved in our asymptotic belong to $B(X,Y)$\,, which denotes the space of all
bounded linear operators from $X$ to $Y$\,. Good candidates are 
$$X=\Lzqqpesom\qqtext{,}Y=\Lzqqpetom$$
as well as 
$$B(X,Y)=B_{s,t}:=B\big(\Lzqqpesom,\Lzqqpetom\big)$$ 
for some $s>t$\,.

By Theorem \ref{asymsatzweighted} and \eqref{deflomjk} the ingredients of our asymptotic are the linear
operators $\loesom$ and ${\calK}_j$\,, $j=-1,\dots,\fJ-1$\,, i.e. the operators
$$\Lambda\qtext{,}\Pi\qtext{,}\loesom\qqtext{and}\loes^j\qtext{,}j=0,\dots,\fJ$$
as well as the correction operators
$$\Gamma_j\qtext{,}j=0,\dots,\fJ-N\qquad.$$
Moreover, the correction operators $\Gamma_j$ map to $\corom{q}{j}$ and their coefficients
are given by the scalar products
$$s^{E,\ell}_{\sigma,m}\FG:=\skpboml{\FG}{\Esml}\qtext{,}s^{H,\ell}_{\sigma,m}\FG:=\skpboml{\FG}{\Hsml}$$
and the numbers $\xi^{\ell,\sigma,m}_{1,\gamma,n,-}$\,, $\zeta^{\ell,\sigma,m}_{1,\gamma,n,-}$\,.
(See the definitions and remarks around Theorem \ref{asymsatzlokal} and Theorem \ref{asymsatzlokalzwei}.)

Let us assume $\omega$ to be small enough and still for the moment the perturbation
$\Lambdad$ to be compactly supported. Then clearly $\Lambda\in B_{s,t}$ for all $s\geq t$\,.
By Lemma \ref{decolem}
\begin{align}\begin{split}
\Pi&\in B\big(\Lzqqpesom,\ronqtnom\times\dqpetnom\big)\subset B_{s,t}\qquad,\\
\Pi_{\reg}&\in B\big(\Lzqqpesom,\regqom{-1}{s}\big)\subset B_{s,t}
\end{split}\mylabel{PiPiregBst}\end{align}
for all weights $s\in(1-N/2,\infty)\ohne\pI$ and $t\leq s$\,, $t<N/2$\,,
since the natural embedding $\regqom{-1}{s}\subset\regqntom$ is continuous.
By \paulytimeharmtheofred \big(see also \paulytimeharmcorfred\big) for $s,-t>1/2$
\beq\loesom\in B\big(\Lzqqpesom,\Ronqtom\times\Dqpetom\big)\subset B_{s,t}\qquad.\mylabel{loesomBst}\eeq
By \paulystaticitloes
\begin{align}
\loes^j&\in B\big(\regqom{-1}{s},\regqntom\big)\mylabel{loesjBst}
\intertext{and therefore}
\loes^j\Pi_{\reg}&\in B_{s,t}\mylabel{loesjregBst}
\end{align}
for $s\in(j-N/2,\infty)\ohne\pI$ and $t\leq s-j$\,, $t<N/2-j$\,.
By Lemma \ref{loesEsmHgnlemma}
\beq\Esmk\,,\,\Hsmk\in\Lzqqpetom\mylabel{EHBst}\eeq
for $t<-k-\sigma-N/2$ and thus
\beq\corom{q}{j}=\Lin\set{\Lambda^\me\Esmk\,,\,\Lambda^\me\Hsnk}{k+2\sigma\leq j}\subset\Lzqqpetom\mylabel{corBst}\eeq
for all $t<-j-N/2$\,.
Moreover, $s^{E,k-2\sigma+1}_{\sigma,m}$\,, $s^{H,k-2\sigma+1}_{\sigma,m}$ and
$s^{E,0}_{\sigma,m}$\,, $s^{H,0}_{\sigma,m}$
are continuous linear functionals on $\Lzqqpesom$
for $s\in(j+1+N/2,\infty)\ohne\pI$ since again by Lemma \ref{loesEsmHgnlemma}
$$\Esmkmzse\,,\,\Hsmkmzse\,,\,\Esmn\,,\,\Hsmn\in\Lzqqpeom{<\sigma-k-1-\Nh}\subset\Lzqqpeom{-s}$$
for $0\leq2\sigma\leq k\leq j$\,. This yields for $s\in(j+1+N/2,\infty)\ohne\pI$ and $t<-j-N/2$
\beq\Gamma_j\in B\big(\Lzqqpesom,\corom{q}{j}\big)\subset B_{s,t}\qquad.\mylabel{GBst}\eeq

Now we weaken our assumptions on the perturbations $\Lambdad$\,, such that they do not have to be
compactly supported anymore. Thus let us assume $\Lambda$ to be $\tau$-$\pc{1}$-admissible.
By \paulytimeharmPzero\, $\pP$\,, the generalized point spectrum of $\calM$\,,
does not accumulate at zero for $\tau>(N+1)/2$\,, i.e. $\loesom$ is well defined for small $\omega$\,.
Furthermore, the following assertions still hold true:\\
\hspace*{1cm}\eqref{PiPiregBst} for $\tau>0$\,, $\tau>s+1-N/2$ and $\tau\geq-s-1$\\
\hspace*{1cm}\eqref{loesomBst} for $\tau>1$\\
\hspace*{1cm}\eqref{loesjBst} for $\tau>0$\,, $\tau>s-N/2$ and $\tau\geq j-s-1$\\
\hspace*{1cm}\eqref{loesjregBst} if \eqref{PiPiregBst} and \eqref{loesjBst}\\
\hspace*{1cm}\eqref{EHBst} for $\tau>k+\sigma$ and $\tau\geq N/2-1$\\
\hspace*{1cm}\eqref{corBst} for $\tau>j$ and $\tau\geq N/2-1$\\
\hspace*{1cm}\eqref{GBst} for $\tau>j+1$ and $\tau\geq N/2-1$\\
Collecting the values for $s,t$ and $\tau$ we obtain

\begin{lem}\mylabel{alleopsBsttau}
Let $\fJ\in\nzn$ and $s>\max\{1/2,\fJ+1-N/2\}$ with $s\notin\pI$ as well as
$t<\min\{-1/2,N/2-\fJ\}$\,, $t\leq s-\fJ$\,. Moreover, let
$$\tau>\max\big\{(N+1)/2,s+1-N/2\big\}\qquad.$$
Then for $j=0,\dots,\fJ-1$ and $i=0,\dots,\fJ-N$
$$\loesom\,,\,\Pi\,,\,\loesn\loes^j\Pi_{\reg}\,,\,\Gamma_i\in B_{s,t}$$
and thus also for $j=-1,\dots,\fJ-1$
$$\calK_j\,,\,\loes_{\omega,j}^{\calK}\in B_{s,t}\qquad.$$
\end{lem}

Now we approximate $\Lambdad$ by compactly supported perturbations. 
For this purpose we define the cut-off functions
$\varphi_n:=(1-\mbox{\boldmath$\eta$})(r/n)$\,, $n\in\nz$ and
$$\Lambda_n:=\id+\varphi_n\Lambdad\qquad.$$
(We note $\restr{\varphi_n}{U_n}\equiv1$ and $\restr{\varphi_n}{A_{2n}}\equiv0$\,.)
Then $\Lambda_n$ converges to $\Lambda$ for $n\to\infty$ pointwise a.e.
and also in the operator norm of $B_{t,t}$ for all $t\in\rz$ since $\tau>0$\,.
Moreover, if $\tau>s-t$ this convergence also holds true in $B_{t,s}$\,.

From now on all operators, forms and numbers carrying an index $n$
correspond to the truncated transformation $\Lambda_n$\,.

By a short calculation and a regularity result,
e.g. \cite[Korollar 3.8 (ii)]{paulydiss}, we obtain

\begin{lem}\mylabel{gemdefbereich}
Let $\tau>0$\,. For all $s\in\rz$ and $r_2\leq n\in\nz$
\begin{align*}
\ronqsom\cap\eps^\me\dqsom&\subset\ronqsom\cap\eps_n^\me\dqsom\qquad,\\
\mu^\me\ronqpesom\cap\dqpesom&\subset\mu_n^\me\ronqpesom\cap\dqpesom
\end{align*}
hold with continuous embeddings, whose norms do not depend on $n$\,.
\end{lem}

Let $s,t$ and $\tau$ satisfy the assumptions of Lemma \ref{alleopsBsttau} and $\tau>s-t$\,.
Applying the latter lemma our static operators ${}_\eps\statMax^q_{s-1}$ and ${}_{\eps_n}\statMax^q_{s-1}$
from \paulystatictheoaltmax\, are well defined
on their common domain of definition $D({}_\eps\MAX^q_{s-1})$\,.
A long but straight forward computation shows that ${}_{\eps_n}\statMax^q_{s-1}$ converges to
${}_\eps\statMax^q_{s-1}$ in the operator norm of $B\big(D({}_\eps\MAX^q_{s-1}),\bWom{q}{s}\big)$\,.
Therefore also the inverse operators converge in the operator norm and
clearly the same holds true for ${}^\mu\statMax^{q+1}_{s-1}$ and ${}^{\mu_n}\statMax^{q+1}_{s-1}$\,.

Since $\loesn$ consists of the inverses of ${}_\eps\statMax^q_{s-1}$ and ${}^\mu\statMax^{q+1}_{s-1}$
also ${}_n\loesn$ converges to $\loesn$ in the operator norm. Thus ${}_n\loes$
converges to $\loes$ in the operator norm and the same holds true for their powers.
By Lemma \ref{EsmHgnlemma} and Lemma \ref{loesEsmHgnlemma}
${}_n\Esmk$ resp. ${}_n\Hsmk$ converge to $\Esmk$ resp. $\Hsmk$ in the corresponding $\Lzqqpetom$\,.
Looking at the representations in Lemma \ref{loesEsmHgnlemma} the coefficients
${}_n\xi^{k,\sigma,}_{}$ and ${}_n\zeta^{k,\sigma,}_{}$ of ${}_n\Esmk$ and ${}_n\Hsmk$ converge to
$\xi^{k,\sigma,}_{}$ and $\zeta^{k,\sigma,}_{}$ in $\cz$\,. Hence also the correction operators
${}_n\Gamma_j$ converge to $\Gamma_j$ in the operator norm.
Furthermore, it follows that the projections $\Pi_n$ and $\Pi_{\reg,n}$
converge to $\Pi$ and $\Pi_{\reg}$ in the respective operator norms.

It remains to discuss the time-harmonic solution operator.
For $\omega$ small enough and $\FG\in\Lzqqpeom{>\frac{1}{2}}$
$$\EH:=\loesom\FG\,,\,\EHn:={}_n\loesom\FG\in\Ronqkmehom\times\Dqpekmehom\qquad.$$
Consequently the difference form $\eh:=\EH-\EHn$ satisfies the radiation condition and solves
$$(M+\ie\omega\Lambda)\eh=\ie\omega(\Lambda_n-\Lambda)\EHn\qquad.$$
For $\tau>1$ we have $(\Lambda_n-\Lambda)\EHn\in\Lzqqpeom{>\frac{1}{2}}$ and thus
$$\eh=\ie\omega\loesom(\Lambda_n-\Lambda)\EHn\qquad,$$
i.e.
$$\loesom-\,{}_n\loesom=\ie\omega\loesom(\Lambda_n-\Lambda)\,{}_n\loesom\qquad.$$
\big(Interchanging $\Lambda$ and $\Lambda_n$ yields
$\loesom(\Lambda_n-\Lambda)\,{}_n\loesom=\,{}_n\loesom(\Lambda_n-\Lambda)\loesom$\,.\big)
Since $\Lambda_n\to\Lambda$ in $B_{t,s}$ we obtain ${}_n\loesom\to\loesom$ in $B_{s,t}$\,.

Summing up and using Theorem \ref{asymsatzweighted} we finally achieve

\begin{lem}\mylabel{opskonvergenz}
Let $\fJ\in\nz$ and $s\in(\fJ+1/2,\infty)\ohne\pI$ as well as
$t<\min\{-1/2,N/2-\fJ-2\}$\,. Moreover, let
$$\tau>\max\big\{(N+1)/2,s-t\big\}\qquad.$$
Then for $\omega\in\czpomd\ohne\{0\}$ and $j=-1,\dots,\fJ-1$
\begin{align*}
{}_n\loesom&\xrightarrow{n\to\infty}\loesom&&\text{in}&&B_{s,t}
&&,&\Lambda_n&\xrightarrow{n\to\infty}\Lambda&&\text{in}&&B_{t,s}\qquad,\\
{}_n{\calK}_j&\xrightarrow{n\to\infty}{\calK}_j&&\text{in}&&B_{s,t}
\end{align*}
as well as for all $n$
$$\loesom-\,{}_n\loesom=\ie\omega\loesom(\Lambda_n-\Lambda)\,{}_n\loesom$$
and
$$\norm{{}_n\loes_{\omega,\fJ-1}^{\calK}}_{B_{s,t}}=\calO\big(|\omega|^\fJ\big)\qquad.$$
\end{lem}

Now we have to modify the result \cite[Lemma 12]{complete} slightly.

\begin{lem}\mylabel{genres}
Let $\fJ\in\nzn$\,, $\hat{\omega}>0$\,, $\omega\in\czpomd\ohne\{0\}$ and $X,Y$ be Banach spaces.
Moreover, let $\cloesom$\,, $\cloesom^{(n)}$ and $\cK_j$\,, $\cK_j^{(n)}$
for $j=-1,\dots,\fJ-1$ resp. $\cN^{(n)}$ for $n\in\nz$
be families of bounded linear operators from $X$ to $Y$ resp. $Y$ to $X$\,.
Furthermore, let
$$\cN^{(n)}\to0\qqtext{,}\cK_j^{(n)}\to\cK_j\qqtext{,}j=-1,\dots,\fJ-1$$
with convergence in the respective operator norms and
\beq\cloesom-\cloesom^{(n)}=\omega\cloesom\cN^{(n)}\cloesom^{(n)}\mylabel{omcalN}\eeq
as well as for all $n$
$$\cloesom^{(n)}-\sum_{j=-1}^{\fJ-1}\cK_j^{(n)}=\calO\big(|\omega|^\fJ\big)$$
as $\omega\to0$ with respect to the operator norm in $B(X,Y)$\,. Then also
$$\cloesom-\sum_{j=-1}^{\fJ-1}\cK_j=\calO\big(|\omega|^\fJ\big)$$
holds as $\omega\to0$ with respect to the operator norm in $B(X,Y)$\,.
\end{lem}

\begin{proof}
The proof is quite similar to the one of \cite[Lemma 12]{complete} and hence may be omitted here.
We just note that in the Maxwell case it is indispensable that \eqref{omcalN}
contains a term $\omega\cN^{(n)}$ contrary to just $\cN^{(n)}$
in the case of the Helmholtz equation or the equations of linear elasticity.
\end{proof}

We are ready for the

\vspace*{2mm}
\begin{proof}{\bf of the Main Theorem }
If we set $X:=\Lzqqpesom$ and $Y:=\Lzqqpetom$\,, then a combination
of Lemma \ref{opskonvergenz} and Lemma \ref{genres} yields our desired asymptotic
$$\norm{\loes_{\omega,\fJ-1}^{\calK}}_{B_{s,t}}=\calO\big(|\omega|^\fJ\big)\qquad,$$
which proves the main theorem for $\fJ\geq1$\,.
If $\fJ=0$ we have by Lemma \ref{decolem} and \eqref{loesompi}
$$\loes_{\omega,-1}^{\calK}=\loesom+(-\ie\omega)^\me\Pi=\loesom\Pi_{\reg}\qquad.$$
Hence \paulytimeharmPzeroiv\, yields the stated assertion.

To prove the first remark we compute
$M\loes_{\omega,j}^{\calK}=-\ie\omega\Lambda\loes_{\omega,j-1}^{\calK}$ for $0\leq j\leq\fJ-1$ and
$$M\loes_{\omega,-1}^{\calK}=M\loesom=M\loesom\Pi_{\reg}=\Pi_{\reg}
-\ie\omega\Lambda\loesom\Pi_{\reg}\qquad.$$
Moreover, we set $\tilde{s}:=s$\,, $\tilde{t}:=t$ or $\tilde{t}:=t+1$\,.
Then for $\fJ\geq2$ we may utilize the main theorem with $\fJ-1$\,, $\tilde{s}$\,, $\tilde{t}$
and for $\fJ\in\{0,1\}$ once more \paulytimeharmPzeroiv, which completes the proof of Remark A.
\end{proof}

Using the Main Theorem as well as Definitions \ref{asymsatzlokaldef} and \ref{asymsatzlokalzweidef}
and the corresponding remarks we note two final observations concerning the correction operators.

\begin{rem}\mylabel{remcoropreg}
Let $\fJ\in\nzn$ and $s,t,\tau$ be as in the Main Theorem. 
Moreover, let $\FG$ be an element of $\regqnsom$\,.
\begin{itemize}
\item[\bf(i)] For $\fJ\leq N$ we have the noncorrected asymptotic
\beq\normB{\big(\loesom-\sum_{j=0}^{\fJ-1}(-\ie\omega)^j\loesn\loes^{j}\big)\FG}_{\Lzqqpetom}
=\calO\big(|\omega|^\fJ\big)\normb{\FG}_{\Lzqqpesom}\quad.\mylabel{noncorasym}\eeq
\item[\bf(ii)] For $\fJ\geq N+1$ we know by Theorem \ref{Regsatz} and Lemma \ref{Regortho} that the noncorrected asymptotic
\eqref{noncorasym} holds true for $\FG\in\regqJsom$\,, i.e. for $\FG$ perpendicular \big(in $\Lzqqpeom{}$\big)
to all special growing forms $\Lambda^\me\Esmk,\Lambda^\me\Hsnk\in\Lzqqpeom{-s}$ with $1\leq k\leq\fJ$\,.
Albeit this condition is sufficient for \eqref{noncorasym} to hold it is not sharp.
Of course, the noncorrected asymptotic \eqref{noncorasym} holds true,
if and only if we have $\Gamma_j\FG=\hat{\Gamma}_{j-1}\FG=0$ for all $j=1,\dots,\fJ-N$\,,
i.e. $\FG$ must be perpendicular only to all $\Lambda^\me\Esmk,\Lambda^\me\Hsnk\in\Lzqqpeom{-s}$ with $1\leq k\leq\fJ-N$\,.
\end{itemize}
\end{rem}

\begin{rem}\mylabel{remcorop}
Let $\fJ\in\nzn$ and $s,t,\tau$ be as in the Main Theorem as well as $\FG$ be an element of $\Lzqqpesom$\,.
\begin{itemize}
\item[\bf(i)] For $\fJ\leq N-1$ we have the noncorrected asymptotic
\begin{align}
\begin{split}
&\normB{\big(\loesom+(-\ie\omega)^\me\Pi-\sum_{j=0}^{\fJ-1}(-\ie\omega)^j\loesn\loes^{j}\Pi_{\reg}\big)\FG}_{\Lzqqpetom}\\
&\qquad\qquad=\calO\big(|\omega|^\fJ\big)\normb{\FG}_{\Lzqqpesom}\,.
\end{split}\mylabel{noncorasymtwo}
\end{align}
\item[\bf(ii)] For $\fJ\geq N$ the noncorrected asymptotic \eqref{noncorasymtwo} holds true,
if and only if we have $\Gamma_j\FG=0$ for all $j=0,\dots,\fJ-N$\,,
i.e. if and only if $\FG$ is perpendicular to all $\Lambda^\me\Esmk,\Lambda^\me\Hsnk\in\Lzqqpeom{-s}$ with $0\leq k\leq\fJ-N$\,.
\end{itemize}
\end{rem}

\begin{acknow}
This research was supported by the {\it Deutsche Forschungsgemeinschaft}
via the project {\sf `We 2394: Untersuchungen der Spektralschar verallgemeinerter
Maxwell-Operatoren in unbeschr\"ankten Gebieten'}.

The author is particularly indebted to his academic teachers Norbert Weck and Karl-Josef Witsch 
for introducing him to the field and many fruitful discussions.
\end{acknow}

\end{document}